\documentclass[11pt]{article} 
\usepackage[utf8]{inputenc} 

\usepackage{geometry}
\usepackage{bbold} 
\usepackage{dsfont} 
\usepackage{booktabs} 
\usepackage{array}
\usepackage{paralist}
\usepackage{verbatim}
\usepackage{subfig}
\usepackage{fancyhdr}
\usepackage{bbm}
\usepackage{showlabels}

\usepackage{sectsty}
\allsectionsfont{\sffamily\mdseries\upshape} 
\usepackage[nottoc,notlof,notlot]{tocbibind} 
\usepackage[titles,subfigure]{tocloft} 

\usepackage{amsmath}
\usepackage{amsthm}
\usepackage{amsfonts}
\usepackage{amssymb}
\usepackage{mathrsfs}
\usepackage{t1enc}
\usepackage{graphicx,color}
\usepackage{epstopdf}
\usepackage{caption}
\usepackage[hidelinks]{hyperref}

\numberwithin{equation}{section}

\newcommand{\dotH}[1]{\dot{H}^{#1}_p}
\newcommand{\hilb}{\mathcal{H}}
\newcommand{\A}{\mathcal{A}}
\newtheorem{remark}{Remark}
\newtheorem{prop}{Proposition}
\newtheorem{thm}{Theorem}
\newtheorem{defn}{Definition}
\newcommand{\RR}{\mathbb{R}}
\newcommand{\NN}{\mathbb{N}}

\newcommand{\cost}{\mathcal{C}}
\newcommand{\Lag}{\mathcal{L}}
\definecolor{DarkGreen}{rgb}{0.1,0.35,0.1}

\title{Stabilising nontrivial solutions of the generalised Kuramoto-Sivashinsky equation using feedback and optimal control}
\author{
Susana N. Gomes, 
 Demetrios T. Papageorgiou,
 Grigorios A. Pavliotis \\
Department of Mathematics \\ Imperial College London \\ London, SW7 2AZ, UK.}

\begin{document}
\maketitle

\begin{abstract}
The problem of controlling and stabilising solutions to the Kuramoto-Sivashinsky equation is studied in this paper. We consider a generalised form of the equation in which the effects of an electric field and dispersion are included. Both the feedback and optimal control problems are studied. We prove that we can control arbitrary nontrivial steady states of the Kuramoto-Sivashinsky equation, including travelling wave solutions, using a finite number of point actuators. The number of point actuators needed is related to the number of unstable modes of the equation. Furthermore, the proposed control methodology is shown to be robust with respect to changing the parameters in the equation, e.g. the viscosity coefficient or the intensity of the electric field. We also study the problem of controlling solutions of coupled systems of Kuramoto-Sivashinsky equations. Possible applications to controlling thin film flows are discussed. Our rigorous results are supported by extensive numerical simulations.
\end{abstract}

\section{Introduction}
The Kuramoto-Sivashinsky (KS) equation on $L-$periodic domains
\begin{eqnarray}
u_t + u_{xxxx} + u_{xx} + uu_x = 0,\label{eq:KS}\\
u(x,t)=u(x+L,t),
\nonumber
\end{eqnarray}
is a paradigm evolution equation that has received considerable attention in recent years due to its wide applicability as well
as the rich and complex dynamics that it supports. The KS equation arises in many physical problems including
falling film flows \cite{Benney,SivMich,ShlangSiv,Hooper85}, two-fluid core-annular flows \cite{PMR90,CPS95},
flame front instabilities and reaction-diffusion-combustion dynamics \cite{Sivashinsky77,Sivashinsky83}, propagation of
concentration waves in chemical physics applications \cite{KT75,KT76,Kuramoto78}, and trapped ion mode dynamics in plasma
physics, \cite{Cohen}.
The KS equation \eqref{eq:KS} is one of the simplest partial differential equations (PDEs) that can 
produce complex dynamics including chaos - see
for example the numerical experiments in
\cite{HymanNico86,HNZ86,Jolly_et_al_90,Kevrekidis1990,Papageorgiou1991,Smyrlis1991,Wittenberg2002,WittenbergHolmes}.
Routes to chaos have been shown numerically to follow a Feigenbaum period-doubling cascade - see \cite{Smyrlis1991} where the
two universal Feigenbaum constants are also computed for the KS with three-digit accuracy. 
A detailed knowledge of the stationary, travelling and time-oscillatory solutions (typically chaotic) of \eqref{eq:KS} is significant in technological
applications that seek to enhance heat or mass transfer, for example. In this sense certain solutions are better than others
and a description of the solution phase space is a crucial step in constructing relevant control strategies that can access
unstable states, for instance, that may be desirable in applications.

In many studies equation
\eqref{eq:KS} is  scaled to $2\pi-$periodic domains according to the rescaling
\begin{equation}\label{rescaling}
 x^* = \frac{2\pi}{L}x, \quad t^* = \left(\frac{2\pi}{L}\right)^2 t, \quad u^* = \frac{L}{2\pi} u, \quad \delta^*= \frac{2\pi}{L}\delta,\quad \mu^*= \frac{2\pi}{L}\mu,
\end{equation}
to take the form (we drop the stars and use the same symbols for dependent and independent
variables) 
\begin{eqnarray}
u_t +\nu u_{xxxx} + u_{xx} + uu_x = 0,\label{eq:KS1}\\
u(x,t)=u(x+2\pi,t),\nonumber
\end{eqnarray}
where $\nu=(2\pi/L)^2$ is a positive parameter that decreases as the system size $L$ increases.
The mathematical interest in the KS equation (and related models - see below)
resides in the fact that it is a simple, one-dimensional equation exhibiting complex dynamics making it amenable to analysis
and also a good case study in the area of infinite-dimensional dynamical systems and their control.
The equation is of the active-dissipative type and instabilities are present depending on the value of $\nu$.
If $\nu>1$, it is well known \cite{Robinson2001,Sell2002,Tadmor1986,Temam1988} that the zero solution representing a flat film, is unique. However, when $\nu < 1$ the zero solution is linearly unstable and bifurcates into 
nonlinear states including steady states, travelling waves and solutions exhibiting
spatiotemporal chaos - the dynamical complexity increasing as $\nu$ decreases.
Some of these solutions are stable, and others are unstable \cite{Kevrekidis1990}.
In \cite{Frisch1986,Kevrekidis1990,Papageorgiou1993}, one can find studies of the
stability of steady states of the KS equation. 

There is an extensive literature on the behaviour of the solutions to the KS equation. Well posedness of solutions is studied, for instance, in \cite{Robinson2001,Tadmor1986,Temam1988}. 
It was proved in \cite{Constantin} that the long time dynamics of the KS equation are finite dimensional in the sense
that they are governed by a dynamical system of finite dimension which is at least as large as the number of
linearly unstable modes (this number scales with $L$ or $\nu^{-1/2}$ for \eqref{eq:KS} or \eqref{eq:KS1}, respectively);
these authors also proved that the solutions are attracted by a global attractor, a set of finite dimension.
Boundedness of solutions for general initial conditions was proved independently and by using distinct methods by
\cite{Collet1993,Goodman,Ilyashenko}. These studies also focussed on finding bounds for the dimension of the global attractor
by estimating $L^2-$norms of the solutions, starting with the odd-parity results of \cite{Nicolaenko1985} and those for general
initial data by \cite{Collet1993} along with more recent improvements in \cite{Bronski} and  \cite{Otto2009}. 
Analyticity of solutions in a strip in the complex plane around the real axis was also proved in \cite{Collet1993b}
and \cite{APS13} using different methods.

In the context of falling film flows there have been several studies to extend the KS equation by including additional
physical effects. Of most interest to the present study are the derivations in \cite{Tseluiko2006,Tseluiko2010} for film flow over
flat walls in the presence of electric fields applied perpendicular to the undisturbed interface. The resulting equation,
that also incorporates the effects of dispersion, is a generalisation of \eqref{eq:KS1} and takes the form
\begin{eqnarray}
u_t + \nu u_{xxxx} + \mu\hilb[u_{xxx}] + \delta u_{xxx} + u_{xx} + u u_x =0,\label{eq:KS2}\\
u(x,t)=u(x+2\pi,t),\qquad u(x,0)=u_0(x),\nonumber
\end{eqnarray}
where $\mu\ge 0$ measures the strength of the applied electric field and the parameter $\delta$ measures dispersive effects.
The linear operator $\hilb$ is the Hilbert transform operator
and represents flow destabilisation due to the electric field. 
On $2\pi-$periodic domains the definition of $\hilb$ is
\begin{eqnarray}\label{hilbert} \hilb[u](x) = \frac{1}{2\pi}PV\int_0^{2\pi}u(\xi)\cot\left(\frac{x - \xi}{2}\right)\,d\xi,
\label{eq:Hilbert}\end{eqnarray}
where $PV$ stands for the Cauchy principal value integral.
In the model analysed, the electric field needs to be found
by solving a harmonic problem above the film and calculating the Dirichlet to Neumann map of the solution to construct
the Maxwell stresses that interact with the hydrodynamics - see \cite{Tseluiko2006,Tseluiko2010} for the details.
In fact, for the linearised problem the eigenvalues $\lambda$
corresponding to the eigenfunctions $\exp(ikx)$ are
\begin{eqnarray}
\lambda=k^2+\mu k^2|k|-\nu k^4+i\delta k^3,\label{eq:lambda}
\end{eqnarray}
showing that the presence of the electric field destabilises the flow and increases the number of linearly unstable modes.
Note that instability is possible if $|k|<k_c=\frac{\mu + \sqrt{\mu^2+4\nu}}{2\nu}$, and so there are $2l+1$ unstable
modes where $l$ is the integer part of $k_c$ (this fact is used in Proposition \ref{prop1} later).
This additional destabilisation is important in what follows and makes the control problem more challenging.
The modified equation \eqref{eq:KS2} in the absence of dispersion ($\delta=0$),
has a similar dynamical behaviour to the KS equation \eqref{eq:KS1} but with chaotic dynamics appearing at higher values
of $\nu$ as $\mu$ increases. In fact, boundedness of solutions and an estimate of the dimension of the global attractor
have been proved in \cite{Tseluiko2007} for a class of more general operators whose symbols in Fourier space are 
such that the electric field term in \eqref{eq:lambda} is $|k|^\alpha$ with $3\le\alpha<4$.
On the other hand, in the absence of an electric field but with dispersion present, it is established that dispersion
acts to regularise the dynamics (even chaotic ones) into nonlinear travelling wave pulses - see \cite{Kawahara1983,Kawahara1985,
Akrivis2012} as well as \cite{Marc_IMA,Dmitri_IMA} for a weak interaction theory between pulses that are sufficiently separated.
 
More recently, Christofides \emph{et al.} \cite{Antoniades2001,Armaou2000,Christofides1998,Christofides2000a,Christofides2000,Dubljevic2010,Lou2003} 
showed how to stabilise the zero solution of the Kuramoto-Sivashinsky equation by using state feedback controls. 
They also proved that using linear feedback controls, it is possible to stabilise the zero solution using only 5 point actuated controls.
In addition they prove that the stabilisation is possible using only a certain number of observations of the solution instead of full
knowledge of the solution at all times, as long as the number of observations is equal to or exceeds the number of unstable modes. 
In further work utilising
nonlinear feedback controls \cite{Antoniades2001,Christofides2000a}, Christofides and co-workers formulated optimisation techniques 
and computed possible optimal states by analysing a large number of runs; a proof of the
existence of these optimal positions was not given, however.

In this work we use linear feedback controls and techniques similar to those in \cite{Armaou2000,Christofides2000a}, 
to stabilise non-uniform unstable steady states of generalised versions of the KS equation at small values
of $\nu$ that have not been attempted yet. The mathematical complication is due to the increase of the number
of linearly unstable modes as $\nu$ decreases and $\mu$ increases - see \eqref{eq:lambda}. We achieve stabilisation
of non-uniform states
by stabilising the zero solution of a modified PDE that is satisfied by a perturbation to the desired steady state. The resulting
equation to be controlled becomes
\[
u_t + \nu u_{xxxx} + \mu\hilb[u_{xxx}] + \delta u_{xxx} + u_{xx} + u u_x = \displaystyle{\sum_{i=1}^m b_i(x)f_i(t).}
\]
All our results are still valid in this as well as other cases, as long as the linear operator of the PDE has a self adjoint part, 
a well defined separation between stable and unstable modes,  and a bounded nonlinearity $\mathcal{N}(u)$ in an appropriate functional space. Such conditions are fairly generic in
physically derived systems and pose little restriction to our methodology.

In applications there may be some uncertainties in the estimation of the parameters of the equation, for example, if the intensity of the electric field or the dispersion
parameters are not known exactly. 
It is important, therefore, that the controls applied are robust, that is they still work even when these uncertainties are present. We use results from control theory \cite{Kautsky1985} to prove analytically that the controls are robust to uncertainties in $\nu, \, \delta$ and $\mu$, as long as the error in the prediction of the parameters is small enough, and present numerical simulations that demonstrate this point.

A natural question to address after robust stabilisation of the zero solution to the modified PDE is achieved, is whether this can be done in an optimal manner.
By this we mean stabilisation while minimising a cost functional that measures how close we are to the desired solution and how much energy we are spending with the controls.
This cost functional is of the form
\begin{eqnarray}
\cost\left(u, F\right) = \displaystyle{\frac{1}{2}\int_0^T \| u(\cdot,t)-\bar{u}\|^2 \ dt + \frac{1}{2} \| u(\cdot, T) - \bar{u}\|^2 } \displaystyle{+ \frac{\gamma}{2} \int_0^T \sum_{i=1}^m f_i(t)^2 \ dt},\label{eq:cost}
\end{eqnarray}
where $T$ is the final time of integration and $\bar{u}$ is the desired steady-state we are controlling.
As our control variables  we will consider the positions of the control actuators, following Lou and Christofides \cite{Lou2003}.
Note that the
presence of a nonlinearity in the PDE makes the problem non-convex and therefore we do not expect an optimal control to be unique. 
However, we only wish to prove the existence of an optimal control and to find computationally such optimal controls.

The methodology developed and implemented here can also be applied to systems of nonlinear coupled PDEs.
Of particular interest are systems of coupled Kuramoto-Sivashinsky equations that arise in applications
to interfacial fluid dynamics problems. Such equations were derived systematically using asymptotic methods in~\cite{Papaefthymiou2013}
to model the nonlinear stability of three immiscible viscous fluids of different properties flowing in a stratified
arrangement in a plane channel under the action of gravity and/or a driving pressure gradient.
The ensuing dynamics is very rich and in fact instabilities can emerge even in the absence of inertia,
unlike analogous two-fluid flows. Coupled nonlinear systems are mathematically significantly more challenging than scalar PDEs
since analytical results on global existence and estimates of solution norms, for example, are poorly understood.
Detailed computational results into the complexity of the solutions (especially their zero diffusion limits) 
of such coupled systems of KS equations can be found in \cite{Papaefthymiou2015}.
In the present study we consider the problem of feedback and of optimal control when the equations are coupled through the second
derivatives alone. This is a special case but arises in the derivation of the equations, see~\cite{Papaefthymiou2013}; more generally the nonlinear terms are also coupled and can cause
hyperbolic-elliptic transitions by supporting complex eigenvalues of the nonlinear flux functions, see \cite{Papaefthymiou2015}.
We find that solutions to such a system of ocupled PDEs can also be controlled through a linear feedback loop and also optimally.

\begin{figure}[h!]
 \centering
	\subfloat[$\mu = 0$]{\includegraphics[width=0.5\linewidth]{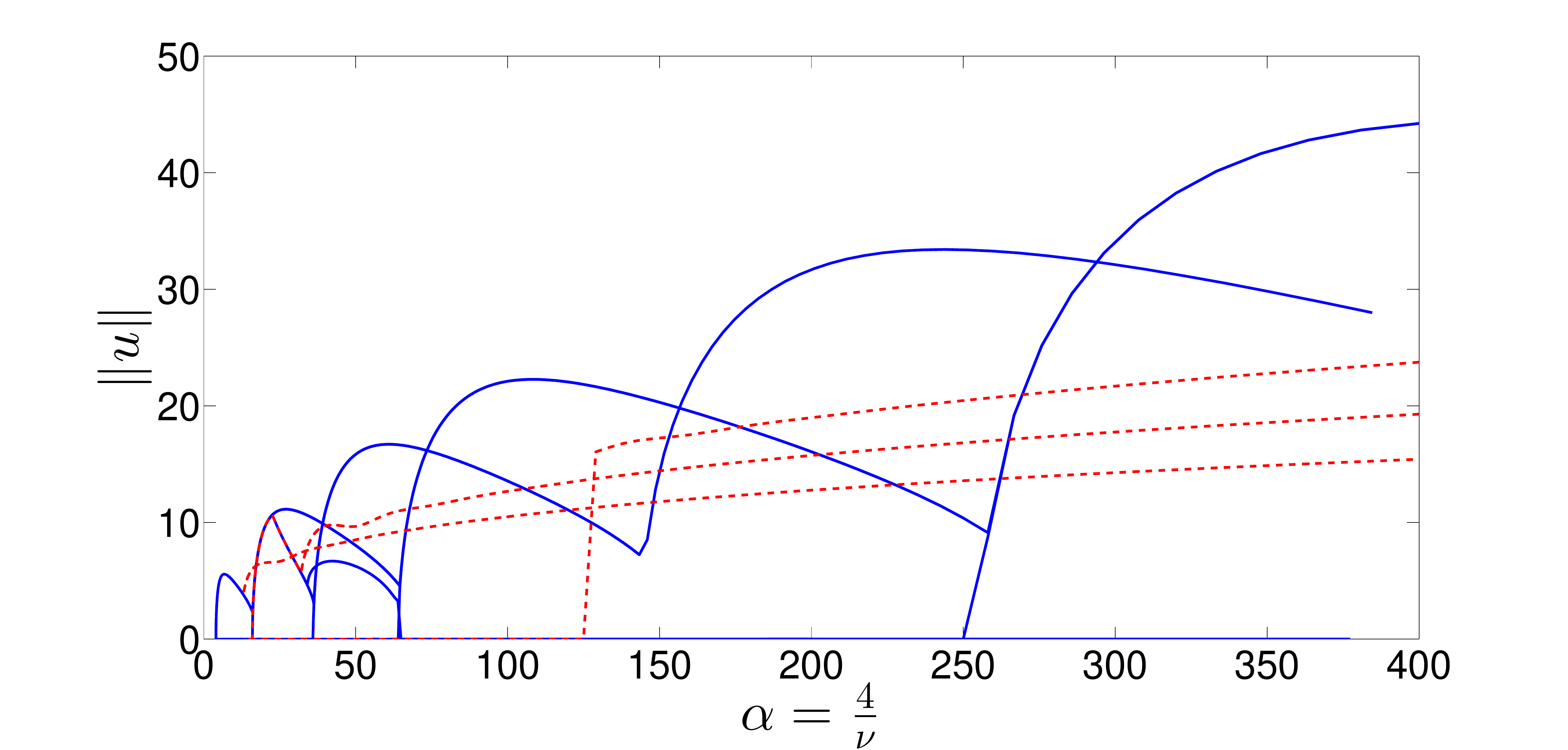}\label{fig:bifdiag3}}
	\subfloat[$\mu = 0.2$]{\includegraphics[width=0.5\linewidth]{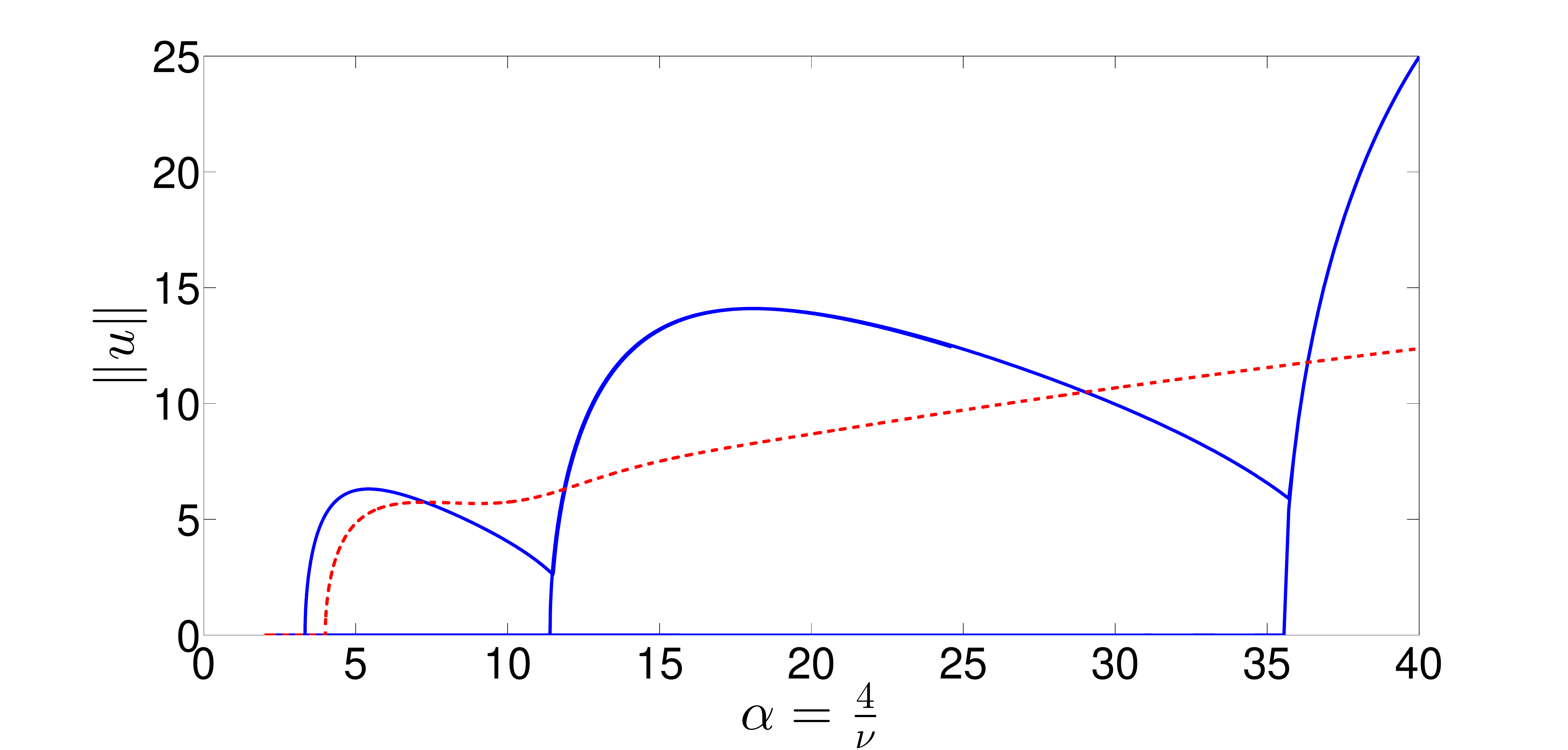}\label{fig:bifdiag4}}

	\subfloat[$\mu = 0.5$]{\includegraphics[width=0.5\linewidth]{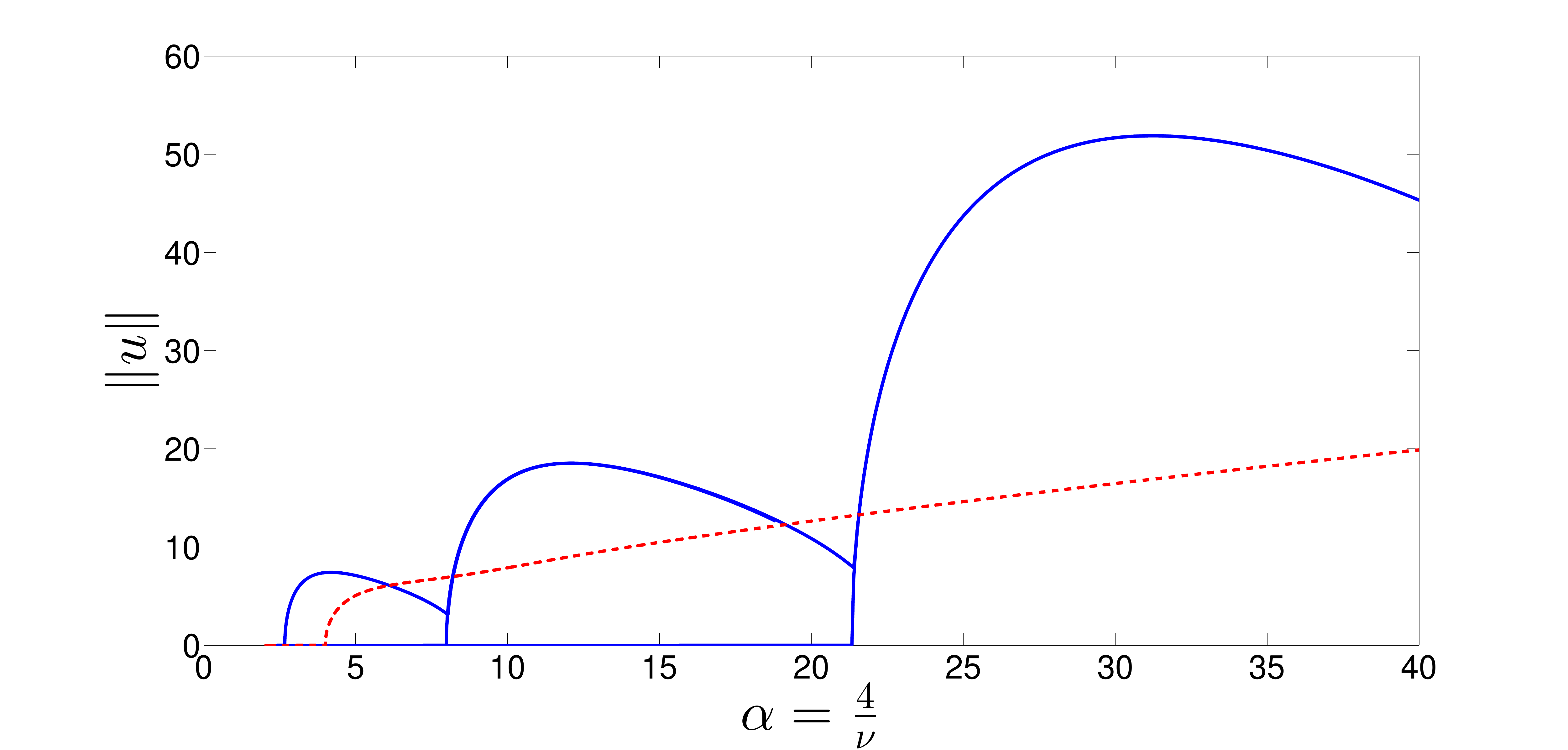}\label{fig:bifdiag5}}
	\subfloat[$\mu = 1$]{\includegraphics[width=0.5\linewidth]{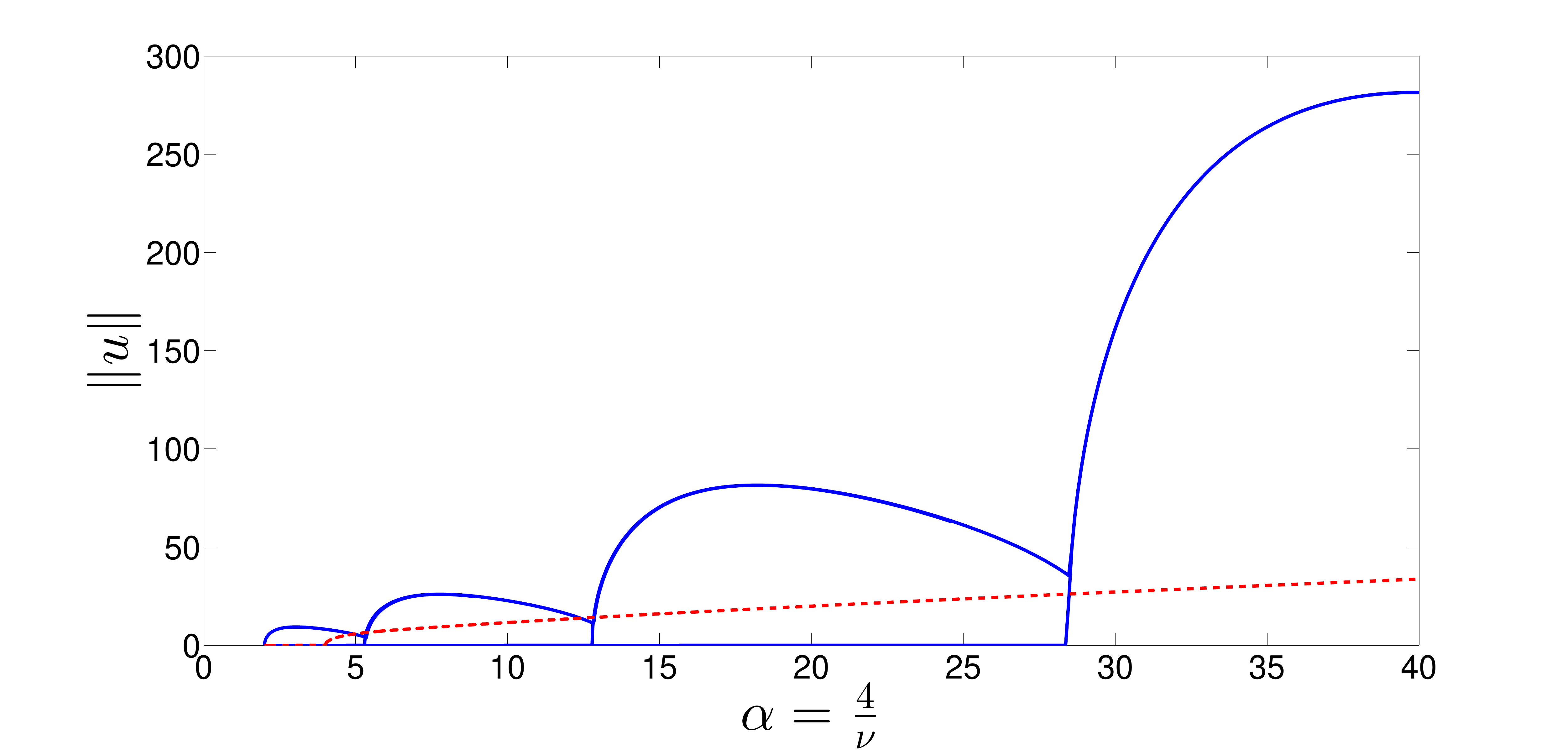}\label{fig:bifdiag6}}
	 \caption{Bifurcation diagram of the $L^2$-norm of the steady state solutions (solid curves - blue online) 
	 and travelling wave solutions (dashed curves - red dashed online) to the gKS equation \eqref{eq:KS2} (with $\delta=0$) in the presence of an electric field for $0.1\le\nu\le 1$, and $\mu = 0$ (\ref{fig:bifdiag3}), $\mu = 0.2$ (\ref{fig:bifdiag4}), $\mu = 0.5$ (\ref{fig:bifdiag5}), and $\mu = 1$(\ref{fig:bifdiag6}).}
	 \label{fig:mu>0}
 \end{figure}

The rest of the paper is organised as follows. In Section~\ref{sec:stabilize} we prove rigorously that nontrivial solutions to the generalised Kuramoto-Sivashinsky equation can be stabilised and analyse the robustness of the proposed controls to variations in the parameters of the equation and in Section~\ref{sec:Numerical} we present numerical experiments that confirm our results. In Section~\ref{sec:optimization} we prove that there exist optimal distributed controls for the KS equation, construct an algorithm to optimise the placement of the control actuators and show our numerical results. Finally, in Section~\ref{sec:coupled} we extend our results to a system of coupled KS equations. We discuss our results in Section~\ref{sec:conclusion}.

%
%

\section{Stabilisation of nontrivial unstable steady states using linear feedback controls}\label{sec:stabilize}

Our goal is to stabilise non-uniform unstable steady states or steady-state travelling wave solutions of equation \eqref{eq:KS2}. 
For the theoretical analysis of the feedback control problem for the KS equation we need $L^\infty$ bounds on the solution and its derivatives. 
To establish such estimates we will use well known $L^2$ bounds, together with the Sobolev embedding theorem.
Optimal estimates for the solution of the KS equation \eqref{eq:KS} in $(0,L)$ were obtained by Otto~\cite{Otto2009}, and
for the re-scaled $2\pi-$periodic KS equation \eqref{eq:KS1} these estimates can be expressed in terms of $\nu= \left(\frac{2\pi}{L}\right)^2$ to find
\begin{subequations}\label{L2estimates}
\begin{align}
\label{solution} \limsup_{t\rightarrow\infty} \|u(\cdot,t)\| & \leq \mathcal{O}(\nu^{-1/6}), \\
\label{1stderivative} \limsup_{t\rightarrow\infty} \|u_x(\cdot,t)\| & \leq \mathcal{O}\left(\nu^{1/2}\ln ^{5/3}\left(\nu^{-1/2}\right)\right), \\
\label{2ndderivative} \limsup_{t\rightarrow\infty} \|u_{xx}(\cdot,t)\| & \leq \mathcal{O}\left(\nu\ln ^{5/3}\left(\nu^{-1/2}\right)\right),
\end{align}	
\end{subequations}
where $\|\cdot\|=(\int_0^{2\pi} (\cdot)^2 \ dx)^{1/2}$ denotes the $L^2$-norm of the solution.
For the generalised equation \eqref{eq:KS2} but in the absence of dispersion ($\delta=0$),
Tseluiko \& Papageorgiou \cite{Tseluiko2007} use the backround flow method to obtain similar estimates in the presence of an electric field.  Their estimates are of the form
\begin{equation}\label{estimatesHT}
\|u\| \leq \left(\|u_0\| + \|\varphi\|\right)e^{-Dt} + C(\nu,\mu) + \|\varphi\|,
\end{equation}
where $C$ and $D$ are constants depending on $\mu$ and $\nu$, and $\varphi$ is a constructed function  with finite $L^2$-norm (we do not need to give it here). 
They also proved that for $u_0 \in \dotH{1}(0,2\pi)$, the first and second derivatives of the solution are bounded, and therefore, $u \in \dotH{2}(0,2\pi)$, where $\dotH{s}$ is the Sobolev space of $s-$times differentiable functions that are periodic and have zero mean.

It is also important to remark that in the case of the generalised KS
equation \eqref{eq:KS2}, with both $\delta$ and $\mu$ non-zero, it was proved in \cite{Frankel2008a,Frankel2008} that the $L^2$ norm of the solution is also bounded. 
In fact
\begin{equation}\label{estimatesDispersion}
 \limsup_{t\rightarrow\infty} \|u(\cdot,t)\| \leq \left\{\begin{array}{lc}
 \mathcal{O}(\nu^{-17/4}), &\textrm{ if } \nu < \nu_0,\\
C , &\textrm{ if } \nu \geq \nu_0,
\end{array}
\right.
\end{equation}
where $\nu_0$ depends on the symbol of the linear operator. Note that these are not optimal bounds.
 The authors also prove boundedness in $L^2$ of spatial derivatives of $u$ up to order $4$.
The estimates \eqref{L2estimates}-\eqref{estimatesDispersion} imply (by use of the Sobolev embedding theorem) that there exist constants $C_1, \, C_2$ that depend only on $\nu$ and $\mu$ such that
\begin{equation} \label{inftyestimate} \|u\|_{\infty} \leq C_1\|u\|_{H^2}, \quad \|u_x\|_{\infty} \leq C_2\|u_x\|_{H^1}.\end{equation}

The controlled 
generalised KS equation can now be introduced and will form the basis of our analysis and computations. 
This consists of a forced version of \eqref{eq:KS2} and reads
\begin{equation}\label{controlledGKS}
\left\{\begin{array}{rclr}
u_t + \nu u_{xxxx} + \mu\hilb[u_{xxx}] + \delta u_{xxx} + u_{xx} + u u_x &=& \displaystyle{\sum_{i=1}^m b_i(x)f_i(t),} &  x\in (0,2\pi), \, t>0,\\
u(x,0) &=& u_0(x), &  x\in (0,2\pi),\\
\displaystyle{\frac{\partial^j u}{\partial x^j}(x+2\pi,t)} &=& \displaystyle{\frac{\partial^j u}{\partial x^j}(x,t), } &  x\in (0,2\pi),\, t>0, \\
f_i(t) & \in & L^2(0,T). &
\end{array}\right.\end{equation}
We assume that the initial condition satisfies $u_0 \in \dotH{2}(0,2\pi)$, $m$ denotes the number of controls, $b_i(x)$, $i = 1,\dots,m$ are the control actuator functions and $f_i(t)$, $i = 1,\dots,m$ are the controls. We will use point actuator functions, which means that the functions $b_i(x)$ are delta functions centered at positions $x_i$, i.e. $b_i(x) = \delta(x-x_i)$, 
or a smooth approximation of such delta functions.

We use an argument similar to \cite{Armaou2000a,Christofides1998,Christofides2000} to prove that it is possible to stabilise nontrivial steady states of the generalised KS
equation~\eqref{eq:KS2}.
Using the Galerkin representation of $u$,
\begin{equation}\label{galerkintruncation}
u(x,t) =  \sum_{n = 1}^{\infty} u_n^s(t)\sin(nx) + \sum_{n = 0}^{\infty}u_n^c(t)\cos(nx),
\end{equation}
substituting into~\eqref{controlledGKS}, and taking the inner product with the functions $\frac{1}{\sqrt{2\pi}}$, $\frac{\sin(nx)}{\sqrt{\pi}}$ and $\frac{\cos(nx)}{\sqrt{\pi}}$, $n = 1,\dots,\infty$, we obtain the following infinite system of ODEs (dots denote time derivatives):
\begin{equation}\label{ODEsystem}
\left\{
\begin{array}{rclr}
\displaystyle \dot{u}_n^s &=& \displaystyle \left(-\nu n^4 +\mu n^3 +  n^2\right) u_n^s +  \delta n^3 u_n^c + g_n^s + \sum_{i = 1}^m b_{in}^sf_i(t) & n = 1,\dots,\infty \\
\displaystyle \dot{u}_n^c &=& \displaystyle \left(-\nu n^4 +\mu n^3 +  n^2\right) u_n^c  - \delta n^3 u_n^s+ g_n^c + \sum_{i = 1}^m b_{in}^cf_i(t) & n = 0,\dots,\infty 
\end{array}\right.
\end{equation}
where $b_{in}^s = \int_0^{2\pi}b_i(x)\sin(nx)dx$ and $b_{in}^c = \int_0^{2\pi}b_i(x)\cos(nx)dx$. The nonlinearities $g_n^s$ and
$g_n^c$ are given by \cite{Akrivis2011}
\[ g_n^s = \frac{n}{4\sqrt{\pi}} \sum_{j + k = n} (u^c_ju^c_k - u^s_ju^s_k) + \frac{n}{2\sqrt{\pi}}\sum_{j-k = n}(u^c_ju^c_k + u^s_ju^s_k),\quad n = 1,\dots,\infty,\]

\[ g_n^c = -\frac{n}{2\sqrt{\pi}} \sum_{j + k = n} u^c_ju^s_k  + \frac{n}{2\sqrt{\pi}}\sum_{j-k = n}(u^c_ju^s_k - u^s_ju^c_k), \quad n = 0,\dots,\infty.\]
In deriving the system \eqref{ODEsystem} we used $\hilb[\sin(x)](x) = -\cos(x)$ and $\hilb[\cos(x)](x) = \sin(x)$; these formulas
can be derived using contour integrations in the complex plane, for example.

We now define
$$ z^u = \left[\begin{array}{c}z_u^u \\ z_s^u \end{array}\right], $$ 
where 
$$z_u^u =  \left[u_0^c \ u_1^s \ u_1^c \cdots u_l^s \ u_l^c\right]^T,\qquad
z_s^u =  \left[u_{l+1}^s \ u_{l+1}^c \cdots \right]^T,
$$
where $z_u^u$ contains the coefficients of the (slow) unstable modes and  $z_s^u$ those of the (fast) stable modes.
In addition,
$$G =  \left[0 \ g_1^s \ g_1^c \ g_2^s \ g_2^c \ \cdots\right]^T,\quad 
D = \operatorname{diag}\left(0, \delta, -\delta, \cdots, \delta n^3, -\delta n^3, \cdots\right),\quad
F = \left[f_1(t) \ f_2(t) \ \cdots \ f_m(t)\right]^T.$$ 
Furthermore, we introduce the notation
\begin{equation}\label{A-B} A = \left[\begin{array}{cc} A_u & 0 \\ 0 & A_s \end{array}\right] \quad \textrm{ and } \quad B = \left[ \begin{array}{c} B_u \\  B_s\end{array}\right], \end{equation}
where $$ A_u = \operatorname{diag}(0,-\nu +\mu + 1, -\nu + \mu + 1, \cdots, -l^4\nu + \mu l^3 + l^2,-l^4\nu + \mu l^3 + l^2),$$
$$A_s = \operatorname{diag}(-(l+1)^4\nu + \mu (l+1)^3 +(l+1)^2,-(l+1)^4\nu + \mu (l+1)^3 + (l+1)^2,\cdots),$$
and
\begin{equation}\label{B}
B_u = \left[\begin{array}{cccc}
b_{10}^c & b_{20}^c & \cdots & b_{m0}^c \\
b_{11}^s & b_{21}^s & \cdots & b_{m1}^s \\
b_{11}^c & b_{21}^c & \cdots & b_{m1}^c \\
\vdots & \vdots & \cdots & \vdots\\
b_{1l}^s & b_{2l}^c & \cdots & b_{ml}^c \\
b_{1l}^s & b_{2l}^s & \cdots & b_{ml}^s 
\end{array}\right], \quad \quad B_s = \left[\begin{array}{cccc}
b_{1(l+1)}^s & b_{2(l+1)}^s & \cdots & b_{m(l+1)}^s \\
b_{1(l+1)}^c & b_{2(l+1)}^c & \cdots & b_{m(l+1)}^c \\
\vdots & \vdots & \cdots & \vdots
\end{array}\right].\end{equation}
We can rewrite the infinite dimensional system of ODEs~\eqref{ODEsystem} as
\begin{equation}
\dot{z}^u = Az^u + Dz^u + G + BF.
\end{equation}
We have the following result.
\begin{prop}\label{prop1}
Let $\bar{u}$ be a linearly unstable steady state or travelling wave solution of \eqref{eq:KS2} and let $2l+1$ be the number of unstable eigenvalues of the system
\begin{equation}\label{linearized} u_t = -\nu u_{xxxx} -  \mu\hilb[u_{xxx}] - u_{xx},\end{equation} $\textrm{i.e.},\,  l+1 > \frac{\mu + \sqrt{\mu^2+4\nu}}{2\nu} > l$. If \  $m =2l +1$ and  there exists a matrix $K\in\RR^{m\times m}$ such that all of the eigenvalues of the matrix $A_u + B_uK$ have negative real part, then the state feedback controls \begin{equation}\label{control} [f_1 \ \cdots \ f_m]^T = F = K(z^u_u-z^{\bar{u}}_u)\end{equation} stabilise $\bar{u}$.
\end{prop}
\begin{proof}
Let $u =  \bar{u} + v$ be a solution to \eqref{eq:KS2}. Substituting into \eqref{eq:KS2} and using the fact
that $\bar{u}$ is a steady state or travelling wave solution of \eqref{controlledGKS}, we obtain the following PDE for the perturbation $v$,
\begin{equation}\label{pertGKS}
v_t + \nu v_{xxxx} + \mu\hilb[v_{xxx}] + \delta v_{xxx} + v_{xx} + vv_x + (\bar{u}v)_x = 0,
\end{equation}
and in controlled form we have
\begin{equation}\label{pertGKScontrol}
v_t + \nu v_{xxxx}  + \mu\hilb[v_{xxx}] + \delta v_{xxx} + v_{xx} + vv_x + (\bar{u}v)_x = \sum_{i = 1}^m b_i(x)f_i(t).
\end{equation}
First we will prove that the given controls stabilise the zero solution of
\begin{equation}\label{linearizedGKS}
v_t = - \nu  v_{xxxx} + \mu\hilb[v_{xxx}] + v_{xx}.
\end{equation}
Note that the dispersion term does not affect instability and it is not necessary to include it in this part of the analysis. After applying a Galerkin truncation and the controls given by~\eqref{control}, we obtain
\begin{equation}\label{expstab}
\dot{z}^v = \left[\begin{array}{cc} A_u + B_u K & 0 \\ B_s K & A_s \end{array}\right]z^v = Cz^v.
\end{equation}
Since the eigenvalues of $A_u+B_uK$ have negative real part and the matrix multiplying $z^v$ is triangular,  
it follows that the zero solution to \eqref{expstab} is exponentially stable.

Next, following  \cite{Armaou2000a,Christofides1998,Christofides2000}, we use a Lyapunov argument to show that these controls 
stabilise the zero solution to equation~\eqref{pertGKS}.
We first use the fact that exponential stability of the system~\eqref{expstab} implies that there exists a positive constant $a$ such that the operator $\A(v) = -\nu v_{xxxx} - \mu \hilb[v_{xxx}]  - v_{xx} + \sum_{i = 1}^m b_i(x)K_{i\cdot}z^v_u$, where $K_{i\cdot}$ denotes the $i-$th row of the matrix $K$, satisfies
\begin{equation}\label{operator} (\A v,v) \leq -a \|v\|^2. \end{equation}
Defining $V(v) = \int_0^{2\pi} \frac{v^2}{2} \ dx$, it is easy to verify that $V(0) = 0$ and $V(v) > 0, \, \forall v>0$. Multiplying \eqref{pertGKS} by $v$ and integrating gives
\begin{equation} \label{Lyapunov}
\frac{d}{d t}\int_0^{2\pi} \frac{v^2}{2} \ dx = \int_0^{2\pi} vv_t \ dx = (\A v,v) - \delta\int_0^{2\pi} v_{xxx}v \ dx- \int_0^{2\pi} v^2v_x  \ dx - \int_0^{2\pi} v(\bar{u}v)_x \ dx.
\end{equation}
Integration by parts and and use of periodicity shows that 
the first two integrals on the right-hand side of \eqref{Lyapunov} are zero. 
It remains to obtain an estimate for the third integral. Again using integration by parts and periodicity gives
\[
-\int_0^{2\pi}(\bar{u}v)_x v\ dx = -\int_0^{2\pi}\bar{u}v_x v\ dx -\int_0^{2\pi}\bar{u}_x v^2\ dx  = -\frac{1}{2}\int_0^{2\pi} \bar{u}_x v^2 \ dx
\]
\[
\leq - \frac{\inf \bar{u}_x}{2}\int_0^{2\pi}v^2\ dx = -\frac{\inf \bar{u}_x}{2}\left\|v\right\|^2.
\]
Adding everything up, we obtain
\begin{equation}\label{estimate_stabilization}
\frac{1}{2}\frac{d}{dt}\left\|v\right\|^2\leq -\left(a + \frac{\inf \bar{u}_x}{2}\right)\left\|v\right\|^2.
\end{equation}
If the eigenvalues of the matrix $A_u+B_uK$ are chosen such that $2a+ \inf \bar{u}_x \geq 0$, we obtain that $\frac{d}{dt}V(v(t)) \leq 0$, which proves that $V$ is a Lyapunov function for the system and therefore the zero solution is stable.

Using controls \eqref{control}, we can therefore stabilise the nontrivial steady state $\bar{u}$ of the original equation.
\end{proof}

Using Proposition \ref{prop1}, we can conclude that in order to stabilise 
the steady state $\bar{u}$ of equation~\eqref{eq:KS2} we should
solve the PDE
\begin{equation}\label{stateequation}
u_t + \nu u_{xxxx} + \mu\hilb[u_{xxx}] +  \delta u_{xxx} +  u_{xx} + uu_x = \sum_{i = 1}^m b_i(x)K_{i\cdot}(z^u_u-z^{\bar{u}}_u).
\end{equation}

\begin{remark}\label{rmk1}
Since the solutions to the generalised Kuramoto-Sivashinsky equation are taken to be periodic with mean zero, it follows that
$\inf \bar{u}_x < 0$. Therefore, the constant $a$ in \eqref{estimate_stabilization} must be chosen large enough so that $a + \frac{\inf \bar{u}_x}{2}$ is positive. 
In the case when $\delta > 0$, we also need to account for the fact that the amplitude of the solutions (and therefore the absolute value of their derivatives) grows with $\delta$ \cite{Kawahara1998,Akrivis2012}. Further details can be found in Section~\ref{sec:Numerical}.
\end{remark}
\begin{remark}\label{rmk2}
The above proposition is clearly valid for the case when $\bar{u} = 0$, in which case the controls are $f_i(t) = K_{i\cdot}z^u_u$, as presented in \cite{Armaou2000a,Christofides1998,Christofides2000}. In the case when $\bar{u}$ is a travelling wave, the 
result follows using a time dependent $z^{\bar{u}}$. See Section~\ref{sec:Numerical} and in particular Equation~\eqref{conversion}.
\end{remark}

\begin{remark}\label{rmk3}
From estimates~\eqref{L2estimates}, it follows that the value $\inf |\bar{u}_x|$ is finite and therefore we can conclude~\eqref{estimate_stabilization}.
\end{remark}

\begin{remark}\label{rmk4}
Christofides et al. \cite{Armaou2000a,Christofides1998,Christofides2000} argued that due to the multiplicity of the eigenvalues of the Kuramoto-Sivashinsky equation being less or equal than $4$, one would only need $5$ controls to stabilise the zero solution of that equation. 
The same holds in our case: the multiplicity of the eigenvalues of the linear operator in \eqref{linearizedGKS} is also less than
or equal than 4, but numerical results suggest that we need to use $m = 2l+1$ controls, or at least a number of controls that is close to the number of unstable modes, see Fig.~\ref{fig:compareControls} and the discussion below.
\end{remark}

\begin{remark}\label{rmk5}
The fact that we are separating the system between stable and unstable modes implies 
that the matrix $B_u$ is square ($B_u\in\RR^{m\times m}$), and using $b_i(x) = \delta(x-x_i)$  means that $B_u$ has full rank. 
It follows that the Kalman rank condition~\cite{Zabczyk1992} is automatically verified and the matrix $K$ needed for the stabilisation will always exist.
\end{remark}

\subsection{Robustness of controls}

A natural and important question is whether the proposed control
methodology is robust with respect to changes (or uncertainty)
in the parameters $\nu$, $\mu$ and $\delta$ that appear in the equation. The
robustness of our method can be proved rigorously using techniques from
control theory, e.g.~\cite[Thm. 6]{Kautsky1985}, and we take this up next.

\begin{prop}\label{prop:bounds}
Let $\lambda_i$, $i=1,\dots,N$ be the eigenvalues of the matrix $C$ appearing in \eqref{expstab}, $X$ be the matrix of eigenvectors of $C$ ,
and let $\kappa(\cdot)$ denote its condition number. Then we have
\begin{equation}\label{normOfK}
\|K\|_2 \leq \frac{\left(\|A\|_2 + \displaystyle{\max_j}\left(|\lambda_j|\right)\kappa(X)\right)}{\sigma_m(B)}
\end{equation}
where $\sigma_m(B)$ is the $m$-th smallest singular value of $B$, which is defined in Equations~\eqref{A-B} and~\eqref{B}, and the solution $z^v$ to equation~\eqref{expstab} satisfies
\begin{equation}\label{normOfSolution}
\|z^v(t)\|\leq \kappa(X)\displaystyle{\max_j}\left(|e^{\lambda_jt}|\right)\|z^v(0)\|.
\end{equation}
\end{prop}

\begin{prop}\label{prop:robustness}
If the feedback matrix $K$ is such that Equation~\eqref{expstab} is exponentially stable, then the perturbed closed loop system matrix $A+BK+\Delta$ remains stable for all disturbances $\Delta$ which satisfy
\begin{equation}\label{pertMatrix}
\|\Delta\|_2 < \displaystyle{\min_{s = i\omega}} \sigma_N\left(sI -(A+BK)\right)=\colon \zeta(K),
\end{equation}
where
\begin{equation*}
\zeta(K)\geq \min_j Re\left(\frac{-\lambda_j}{\kappa(X)}\right),
\end{equation*}
and $Re(\cdot)$ denotes the real part.
\end{prop}
In particular, if there is an error in the estimation of the parameters $\nu$ and $\mu$, then
the feedback matrix $K$ will still stabilise the zero solution as long as the error in the parameter estimation
is bounded by $\zeta(K)$. We have studied
the robustness of the controls for stabilising steady states and travelling waves by combining Propositions~\ref{prop:bounds}
 and~\ref{prop:robustness}. We now present a summary of our results.
\subsubsection*{Variations in $\delta$}
As seen from the dispersion relation \eqref{eq:lambda},
variations in $\delta$ do not affect the stability of the solutions and consequently so they do not affect the matrix $K$. 
This implies that the matrix $\Delta$ is zero, and the zero solution to system \eqref{expstab}
is still stable. In the case when we are interested in stabilising travelling waves, 
we need to take into account the fact that the amplitude of the travelling waves increases with $\delta$ - see for example \cite[Fig.~1]{Kawahara1998}. Hence, if we overestimate the value of $\inf \bar{u}_x$ and take this into account when choosing the new eigenvalues, the stabilised solution should remain close to the desired travelling wave as demonstrated by our numerical
experiments in Figure \ref{fig:compareControls_delta}.

\subsubsection*{Variations in $\nu$ and $\mu$}

As seen from \eqref{eq:lambda}
variations in $\nu$ and $\mu$ can affect the stability of the solutions and the number of unstable modes. 
An increase in unstable modes in turn affects the number of controls needed since our theoretical results support that
 we need the same number of controls as unstable modes
as stated in Remark~\ref{rmk4}.
However, we have performed numerical experiments (see Fig.~\ref{fig:compareControls}) that show that using 
two less controls than predicted theoretically does not affect the stability of the solutions.
\begin{figure}[h!]
 \centering
	\subfloat[$t = 1$]{\includegraphics[width=0.5\linewidth]{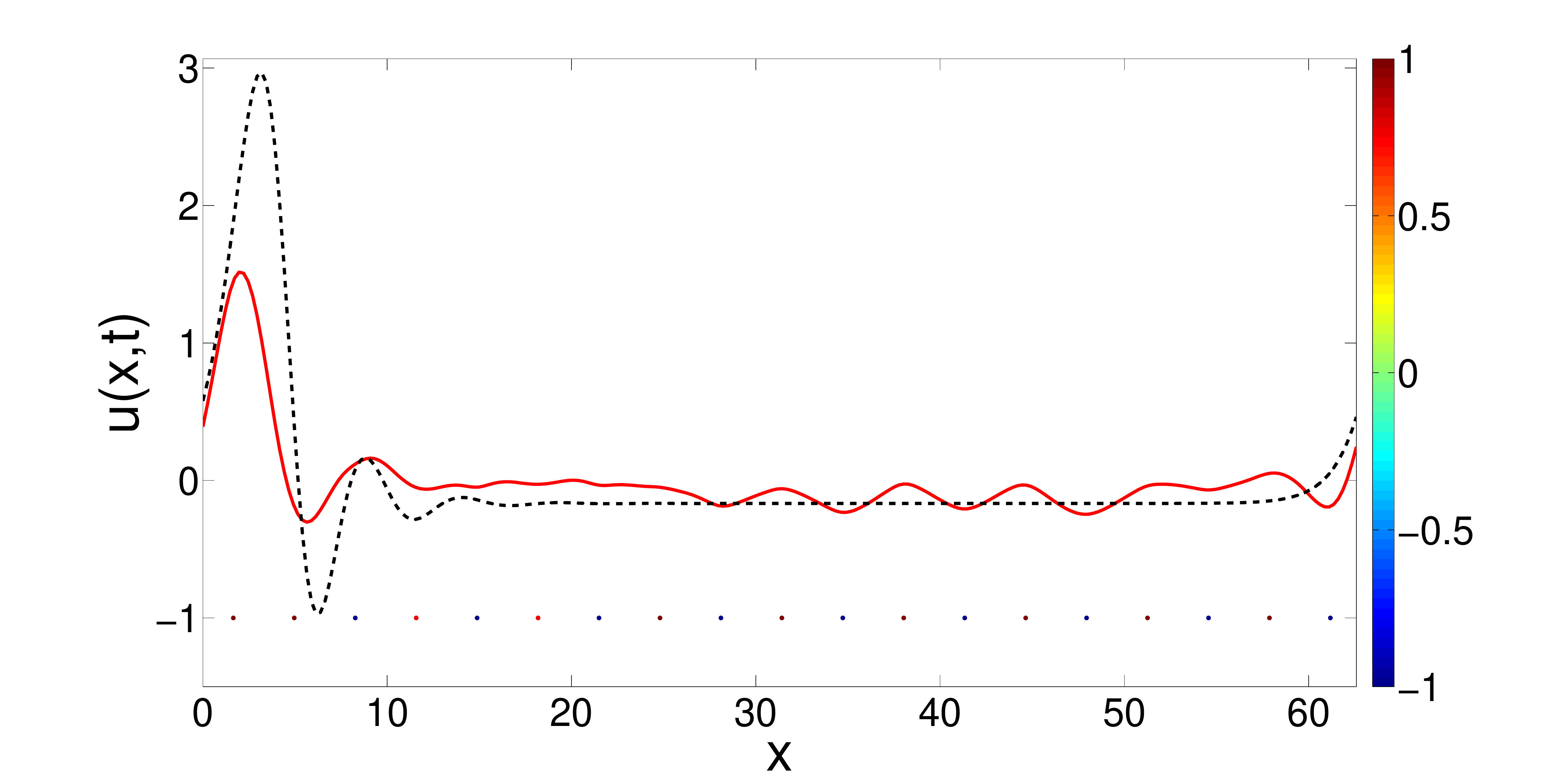}\label{fig:t1}}
	\subfloat[$t = 3$]{\includegraphics[width=0.5\linewidth]{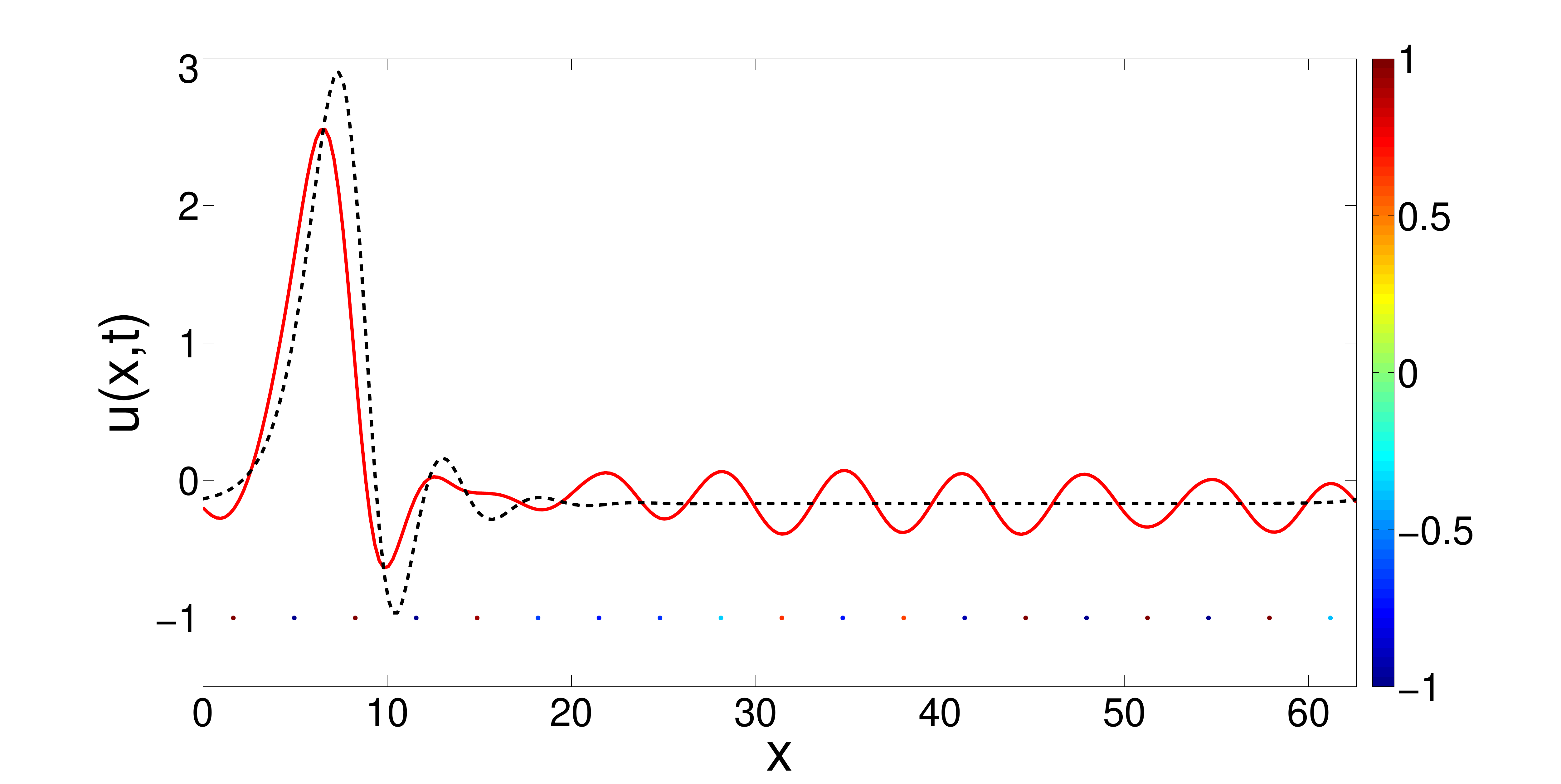}\label{fig:t3}}
	
	\subfloat[$t = 5$]{\includegraphics[width=0.5\linewidth]{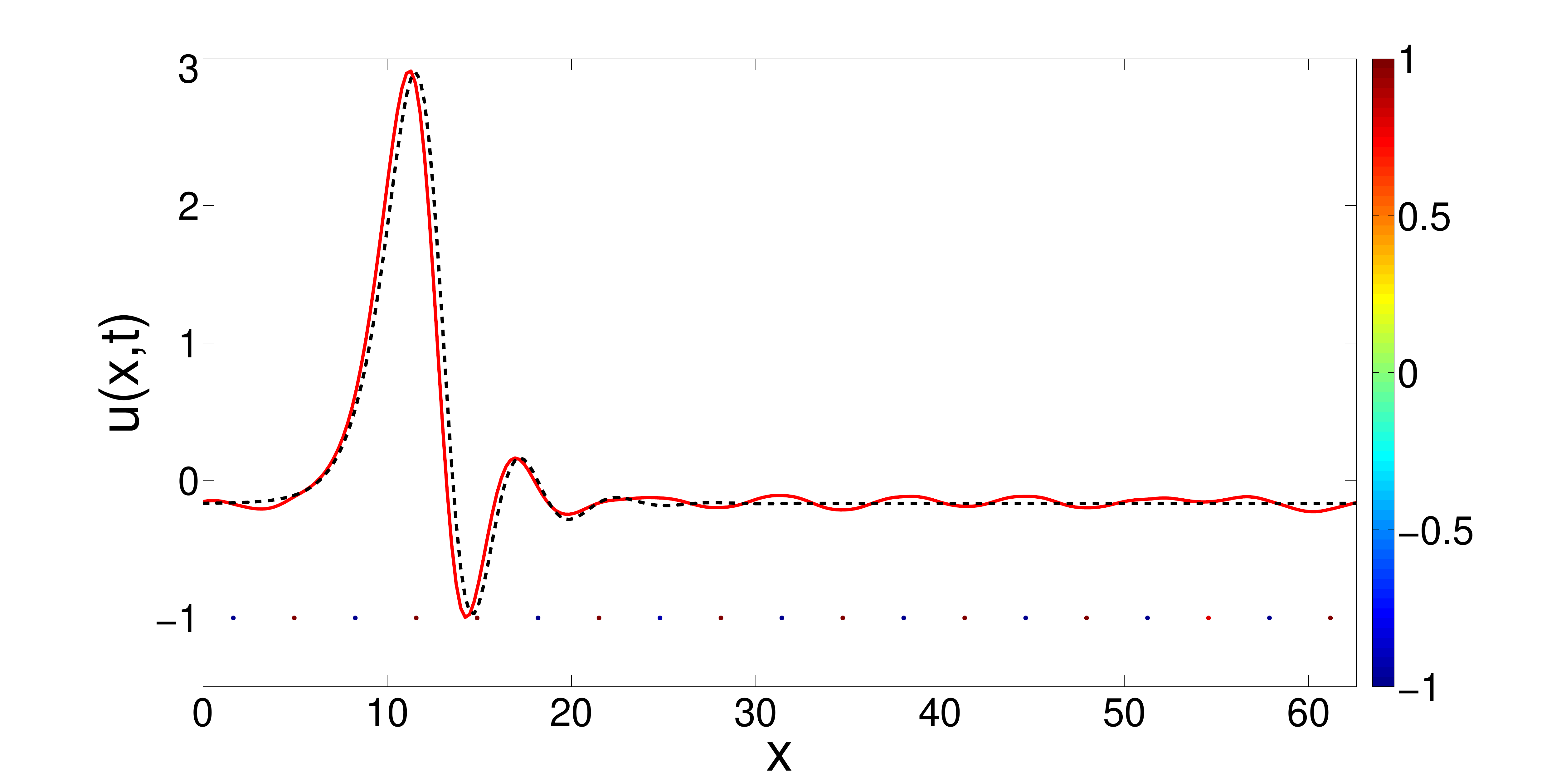}\label{fig:t5}}
	\subfloat[$t = 7$]{\includegraphics[width=0.5\linewidth]{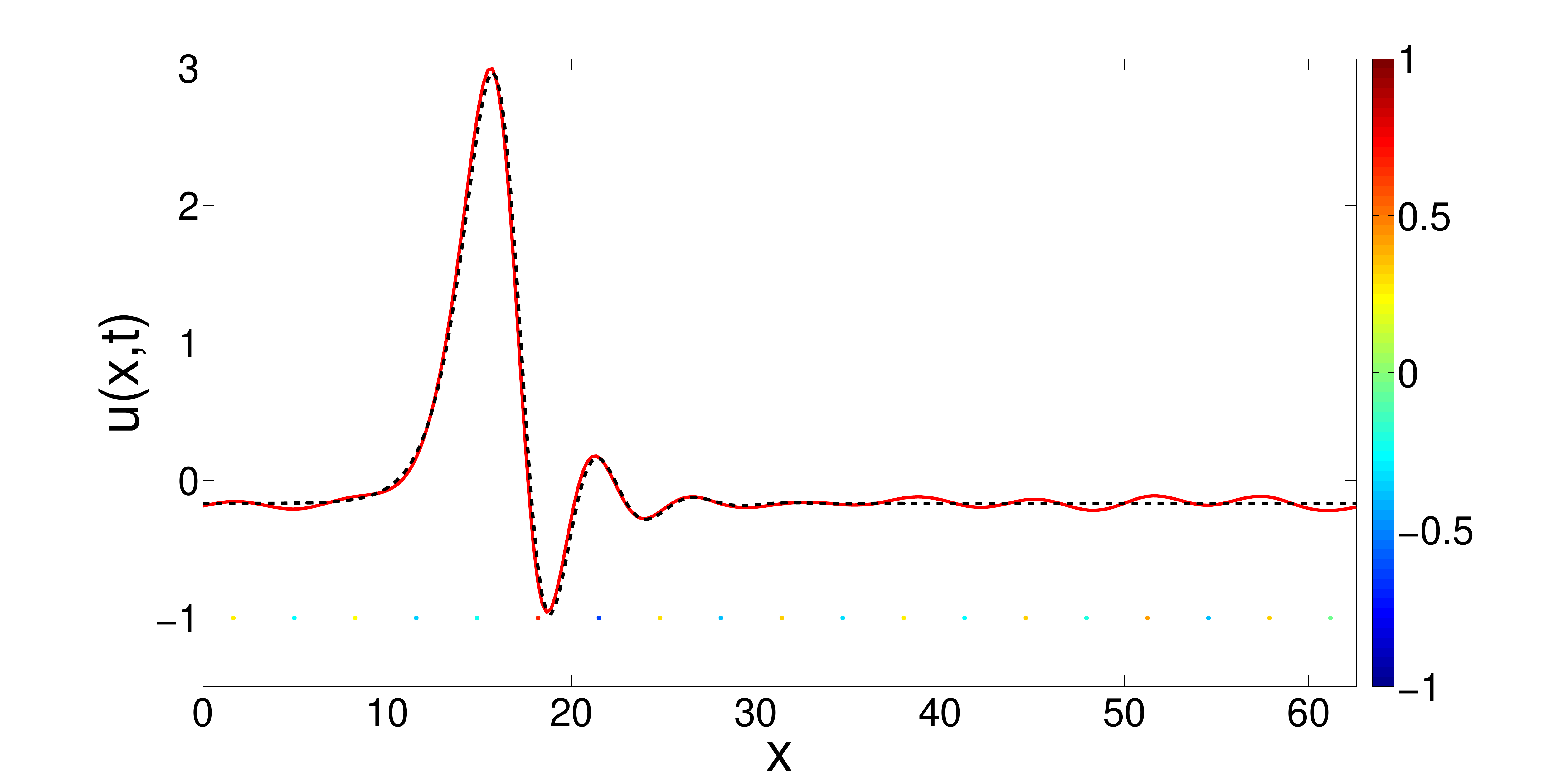}\label{fig:t7}}
	
	\subfloat[$t = 13$]{\includegraphics[width=0.5\linewidth]{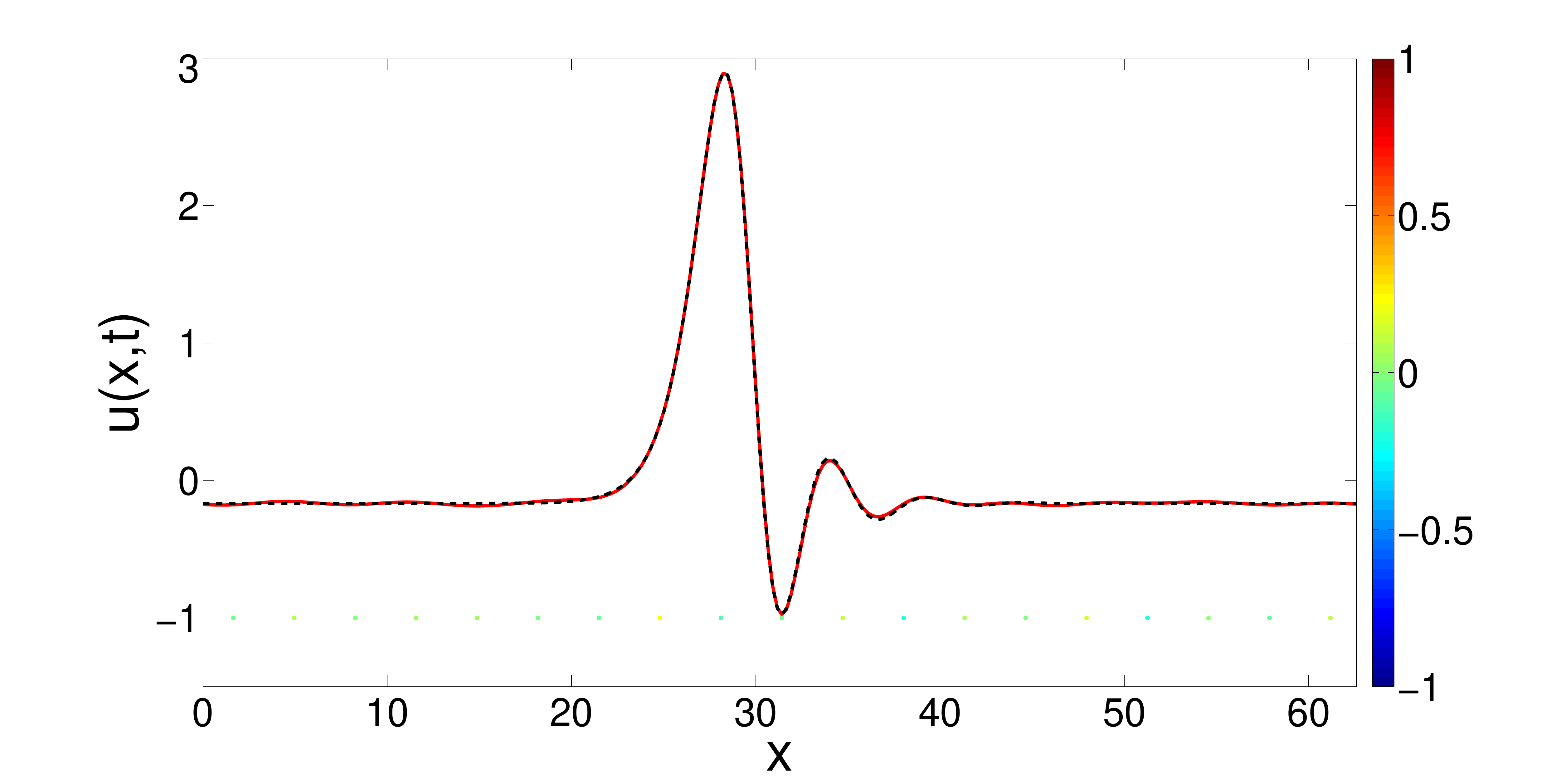}\label{fig:t13}}
	\subfloat[$t = 20$]{\includegraphics[width=0.5\linewidth]{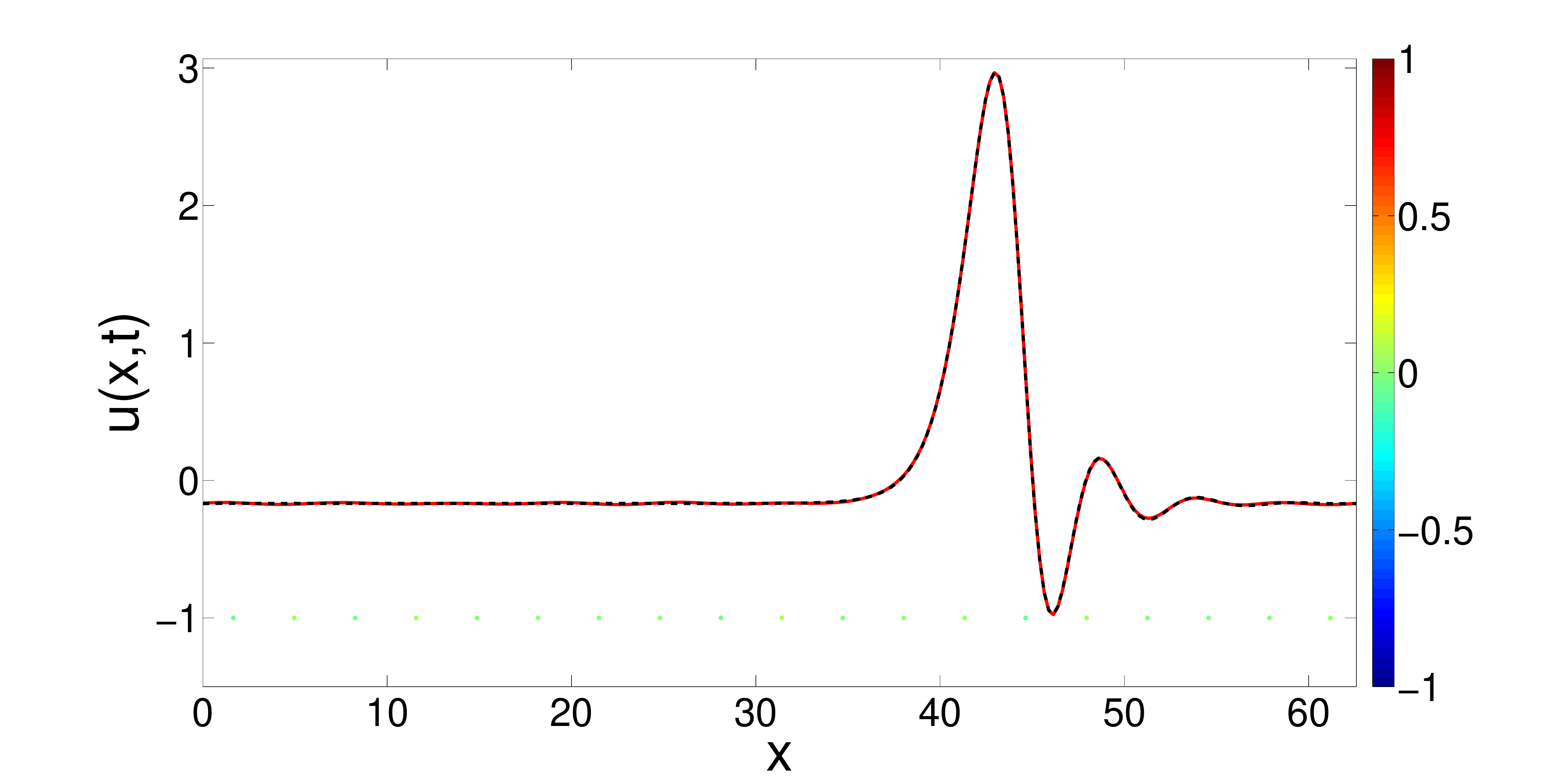}\label{fig:t20}}
	 \caption{Snapshots of the time evolution of the stabilised travelling wave solution in figure \ref{fig:delta0_1pulse} using $m=19$ controls instead of $m=21$ at different times. Red full line is the controlled solution, black dashed lines are the desired travelling wave and the dots represent the controls and their intensity.}
	 \label{fig:compareControls}
 \end{figure}

Now we consider the case where we have some uncertainty of amplitude $\epsilon_1$ and $\epsilon_2$ in the values of $\nu$ and $\mu$,
respectively,
\begin{eqnarray}
u_t = -(\nu + \epsilon_1)u_{xxxx} - (\mu + \epsilon_2)\hilb[u_{xxx}] - u_{xx} -uu_x + \sum_{i=1}^m b_i(x)K_{i\cdot}\left(z^u-z^{\bar{u}}\right).
\label{eq:uncertainty}
\end{eqnarray}
The controls have been chosen so that the solution to the equation is stabilised when $\epsilon_1=\epsilon_2=0$. 
Multiplying \eqref{eq:uncertainty} by $u$, integrating by parts and using Young's inequality, we find
\[
\frac{1}{2}\frac{d}{dt}\|u(\cdot,t)\|^2 \leq -\kappa \|u\|^2 - \epsilon_1\|u_{xx}\|^2 + \frac{\epsilon_2}{2}\left(\|u_x\|^2 + \|u_{xx}\|^2\right),
\]
where $\kappa = a + \frac{\inf \bar{u}_x}{2}$ is a constant.
On the other hand, the perturbation $-\epsilon_1 u_{xxxx} - \epsilon_2 \hilb[u_{xxx}]$ can be discretised and written as
\begin{equation}\label{Delta}
\Delta = \operatorname{diag}(0,-\epsilon_1 k^4 + \epsilon_2 k^3,-\epsilon_1 k^4 + \epsilon_2 k^3),
\end{equation}
$k=1,\dots,N/2$, and it follows that its Fr\"{o}benius norm is given by
\begin{eqnarray}
\|\Delta\|_2^2 = 2\sum_{k=1}^{N/2} k^6\left(-\epsilon_1 k +\epsilon_2\right)^2 = 2 \sum_{k=1}^{N/2}k^6\left(\epsilon_1^2 k^2 - 2\epsilon_1\epsilon_2 k + \epsilon_2^2\right).\label{eq:frobenius}
\end{eqnarray}
For stability we need \eqref{eq:frobenius} to satisfy estimate \eqref{pertMatrix} - see Proposition \ref{prop:robustness}.
Therefore, we have the following proposition.

\begin{prop}
Let $K$ be a matrix such that $A_u + B_u K$ has the prescribed (negative real part) eigenvalues, $\lambda_1,\dots,\lambda_m$, with $m=2l+1$, and let
\[
BK = \left[\begin{array}{cc} B_uK & 0 \\ B_sK & 0\end{array}\right].
\]
Then the perturbed system $A+BK + \Delta$, where $\Delta$ is given by \eqref{Delta},  is stable provided that
\[
\left(2 \sum_{k=1}^{N/2} k^6\left(\epsilon_1^2 k^2 - 2\epsilon_1\epsilon_2 k + \epsilon_2^2\right)\right)^{1/2} \leq \displaystyle{\min_{s = i\omega}}\,\sigma_N\left(sI -(A+BK)\right).
\]
\end{prop}

We have performed
numerical experiments to test the robustness of the controls, 
and in particular we focussed on robustness with respect to the parameters $\delta$ and $\nu$.
Numerical results are presented in Figures \ref{fig:compareControls}-\ref{fig:compareControls_delta} (results in these figures as well as in Figure \ref{fig:delta0},
are shown in the original unscaled domain of length $L$ - see \eqref{rescaling} for the transformations).
In figure \ref{fig:compareControls} we use the same parameter values as in Figure \ref{fig:delta0_1pulse} but we use $19$ controls instead
of $21$, i.e. two controls less than the number of unstable eigenvalues. The dashed curve is the desired travelling wave solution
and the solid curve (red online) is the controlled solution with $19$ controls. We conclude, therefore, that
our control methodology is robust with respect to a slight decrease in the number of controls. Note however, that the number of controls
cannot be significantly smaller than the number of unstable eigenvalues - for example, running the same numerical experiment with
$17$ controls did not yield satisfactory results in the sense that wavy perturbations observed in panels (b) and (c) were not suppressed.

A robustness test with respect to changes in $\nu$ (with $\delta=\mu=0$) is depicted in Figure \ref{fig:compareControls_nu}.
We begin with an unstable travelling wave at $\nu=0.013$ and wish to control it but by solving the KS equation with a reduced
value of $\nu=0.01$, i.e. we impose an uncertainty in the value of the parameter $\nu$ or equivalently in the shape of the desired solution.
The results again show robust behaviour with the two solutions being almost indistinguishable.
Finally in Figure \ref{fig:compareControls_delta} we present robustness experiments for $\mu=0$, $\nu=0.01$ and changes in
the dispersion parameter $\delta$ from $0.03$ to $0.04$, with equally accurate performance as before.

\begin{figure}[h!]
 \centering
	\subfloat[$t = 5$]{\includegraphics[width=0.5\linewidth]{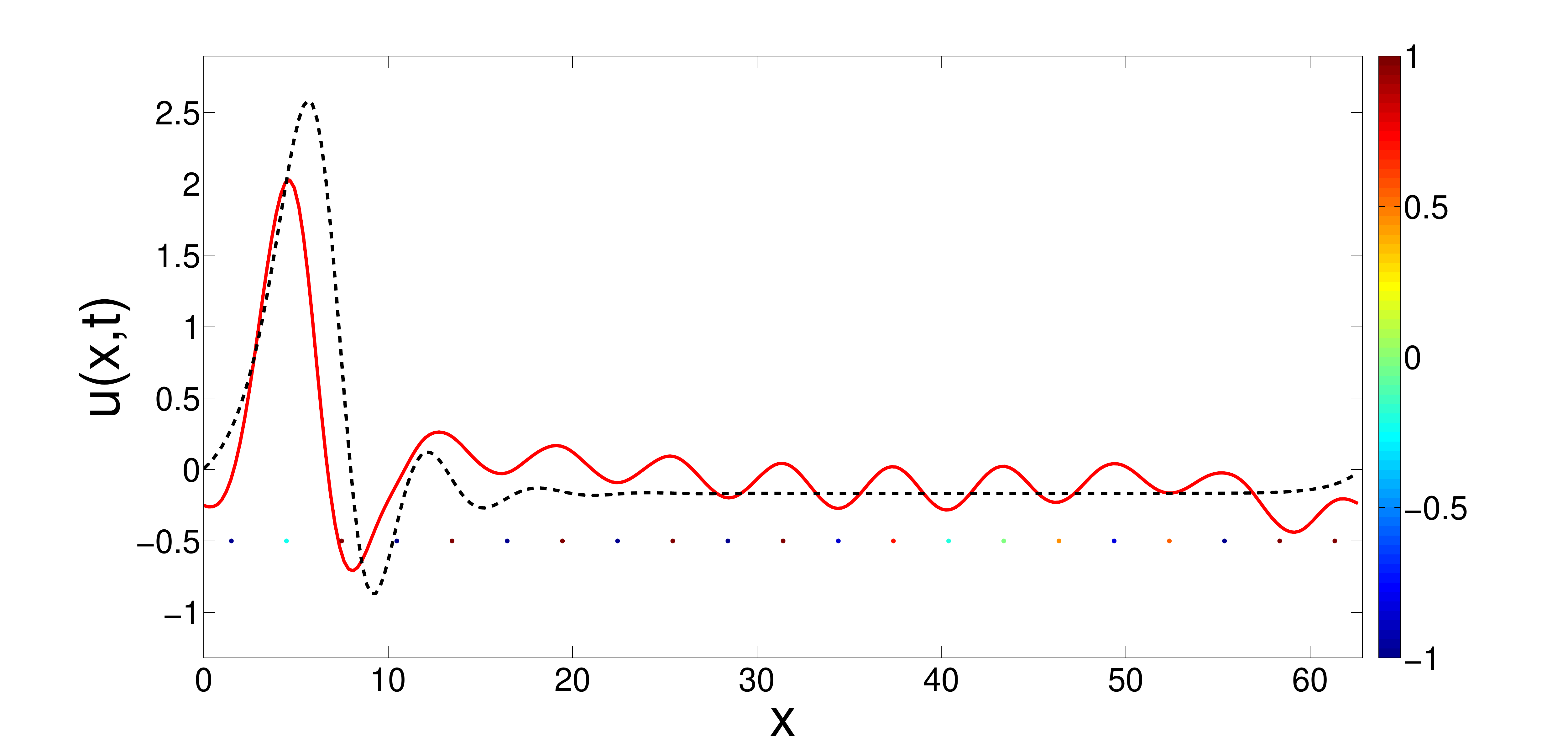}\label{fig:t5_nu}}
	\subfloat[$t = 20$]{\includegraphics[width=0.5\linewidth]{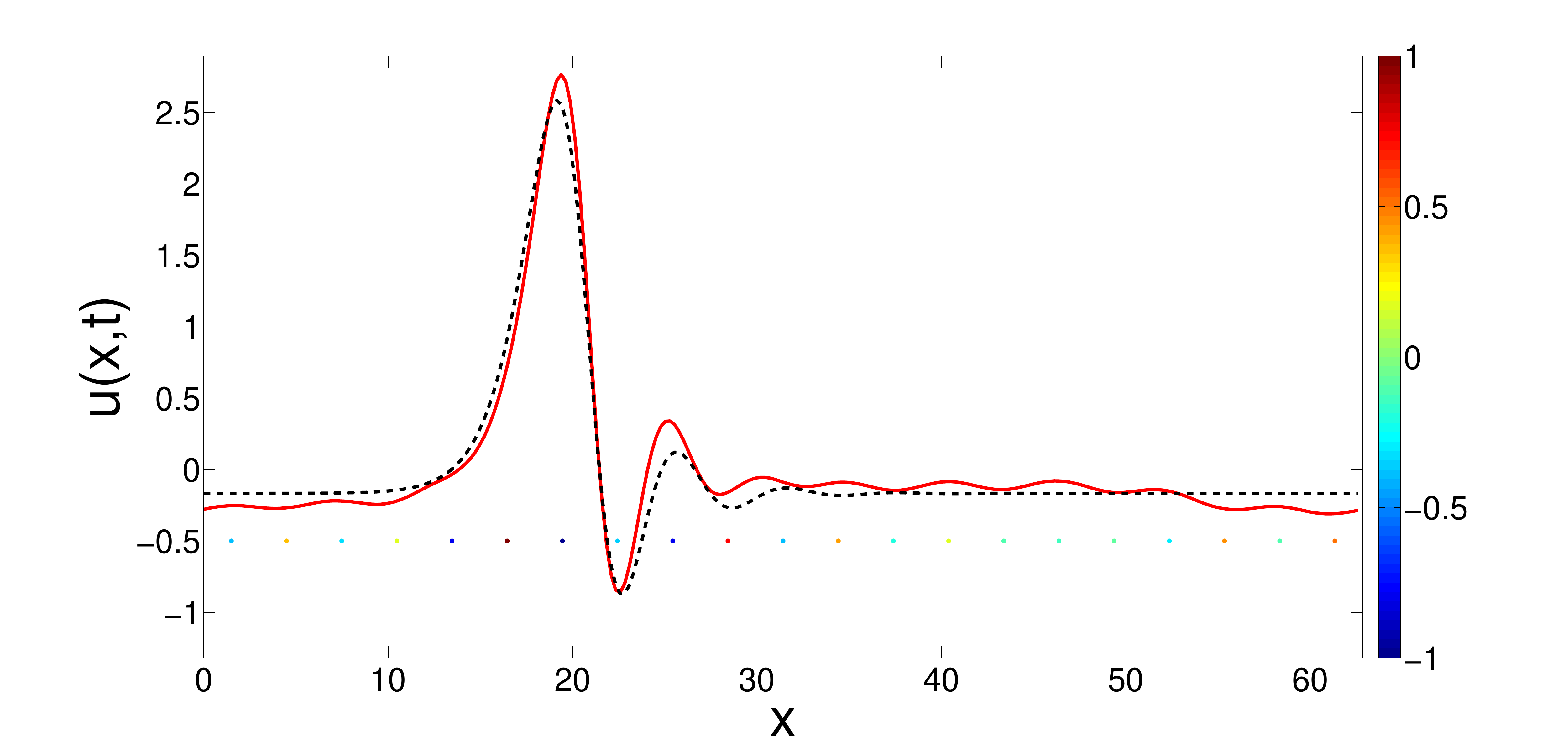}\label{fig:t20_nu}}

	\subfloat[$t = 30$]{\includegraphics[width=0.5\linewidth]{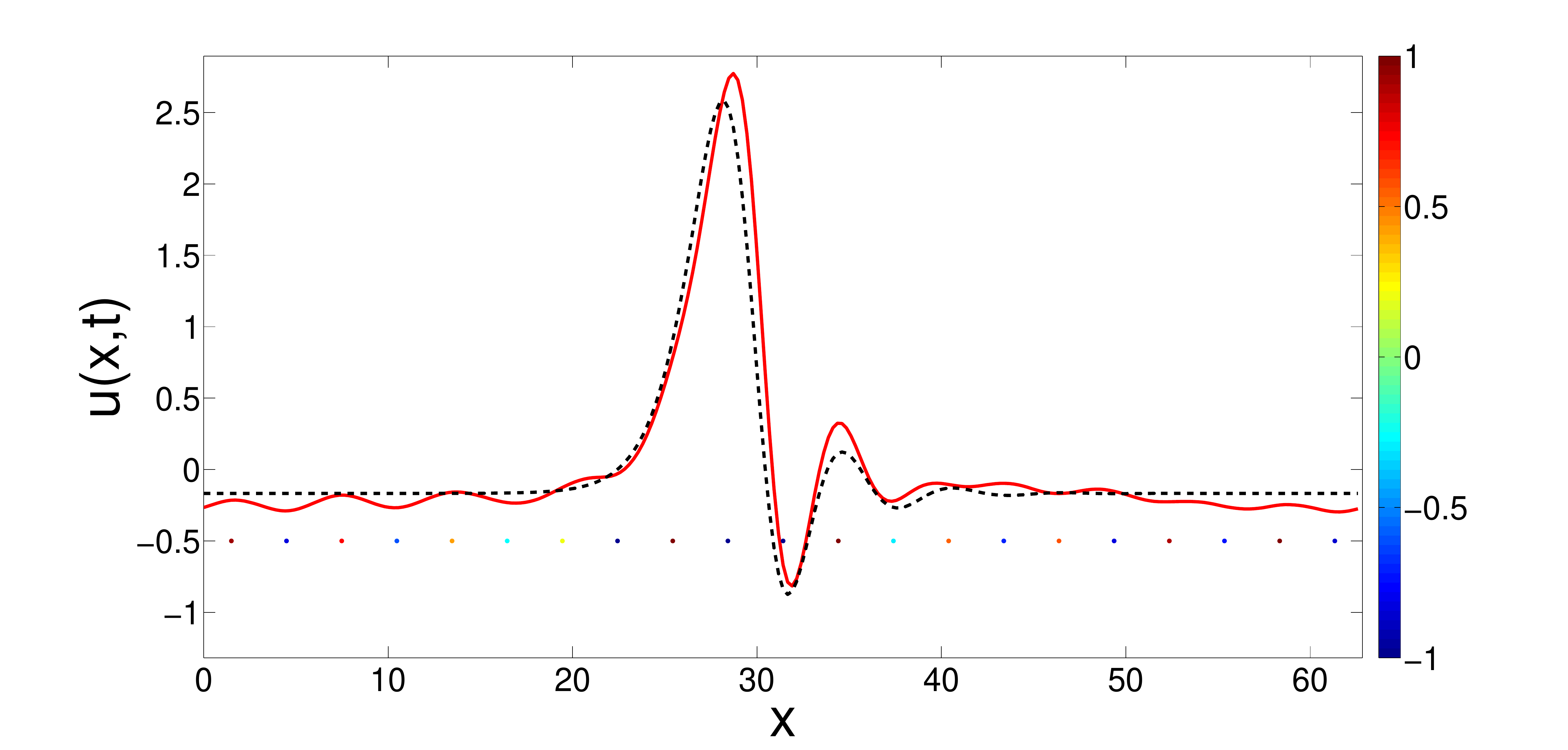}\label{fig:t30_nu}}
	\subfloat[$t = 60$]{\includegraphics[width=0.5\linewidth]{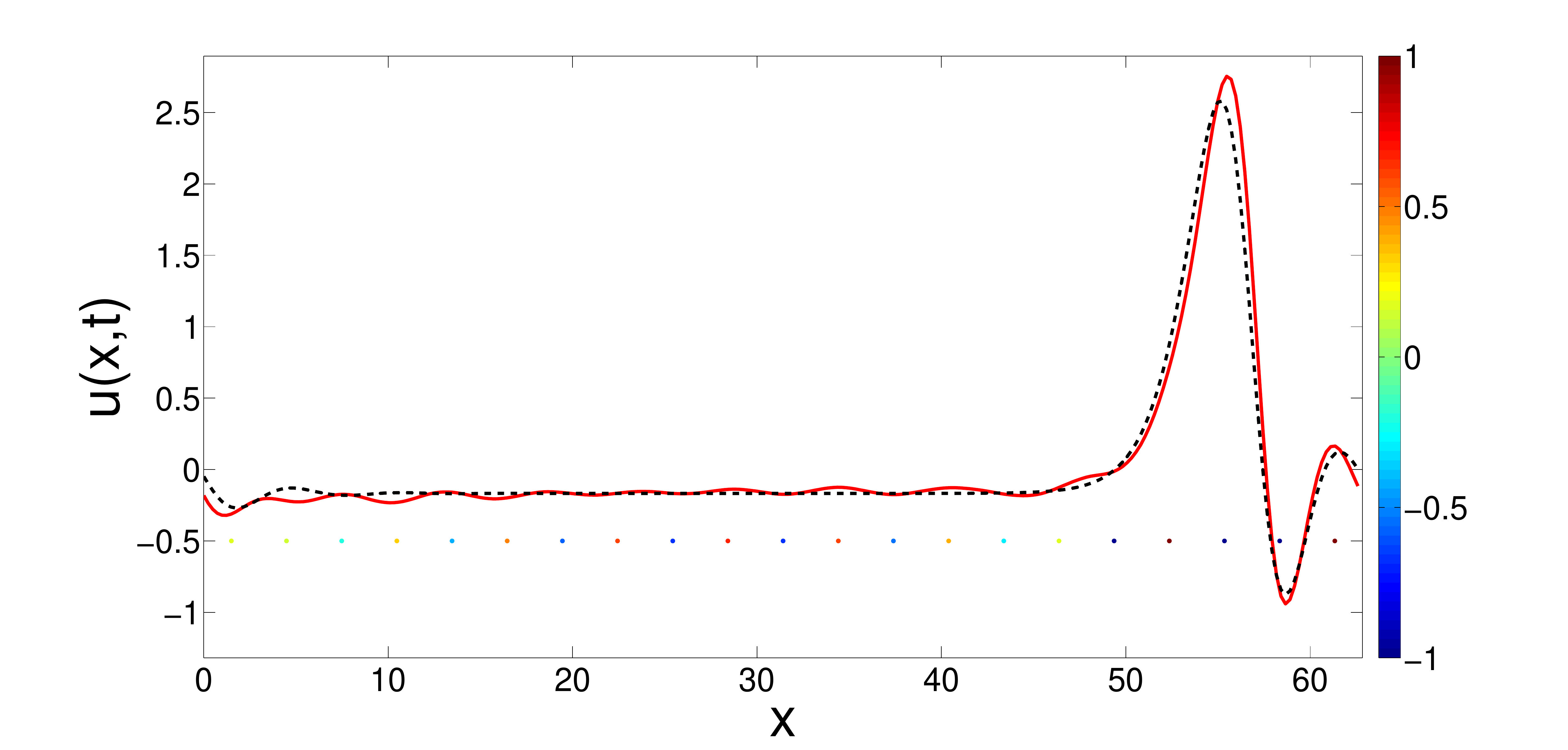}\label{fig:t60_nu}}

	\subfloat[$t = 90$]{\includegraphics[width=0.5\linewidth]{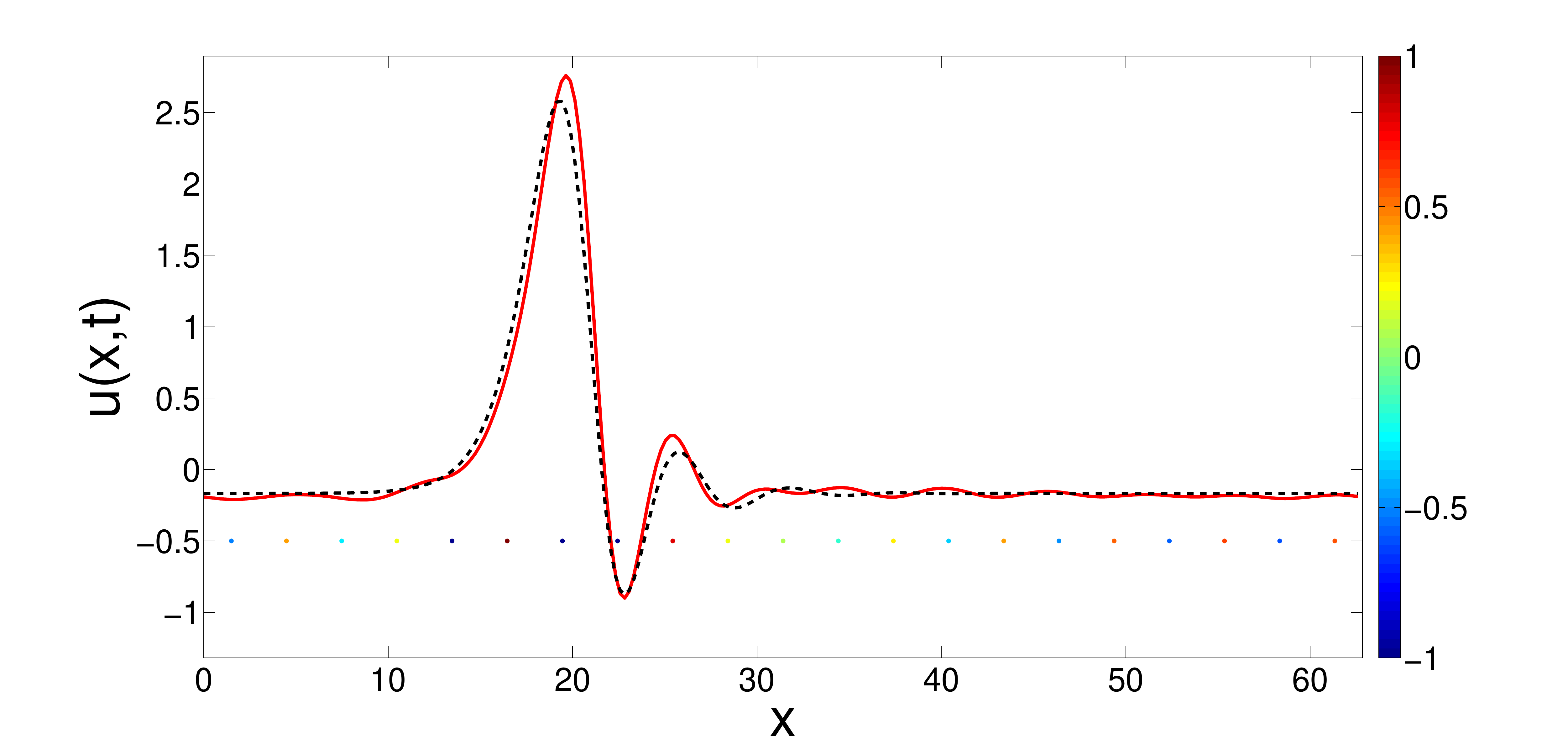}\label{fig:t90_nu}}
	\subfloat[$t = 200$]{\includegraphics[width=0.5\linewidth]{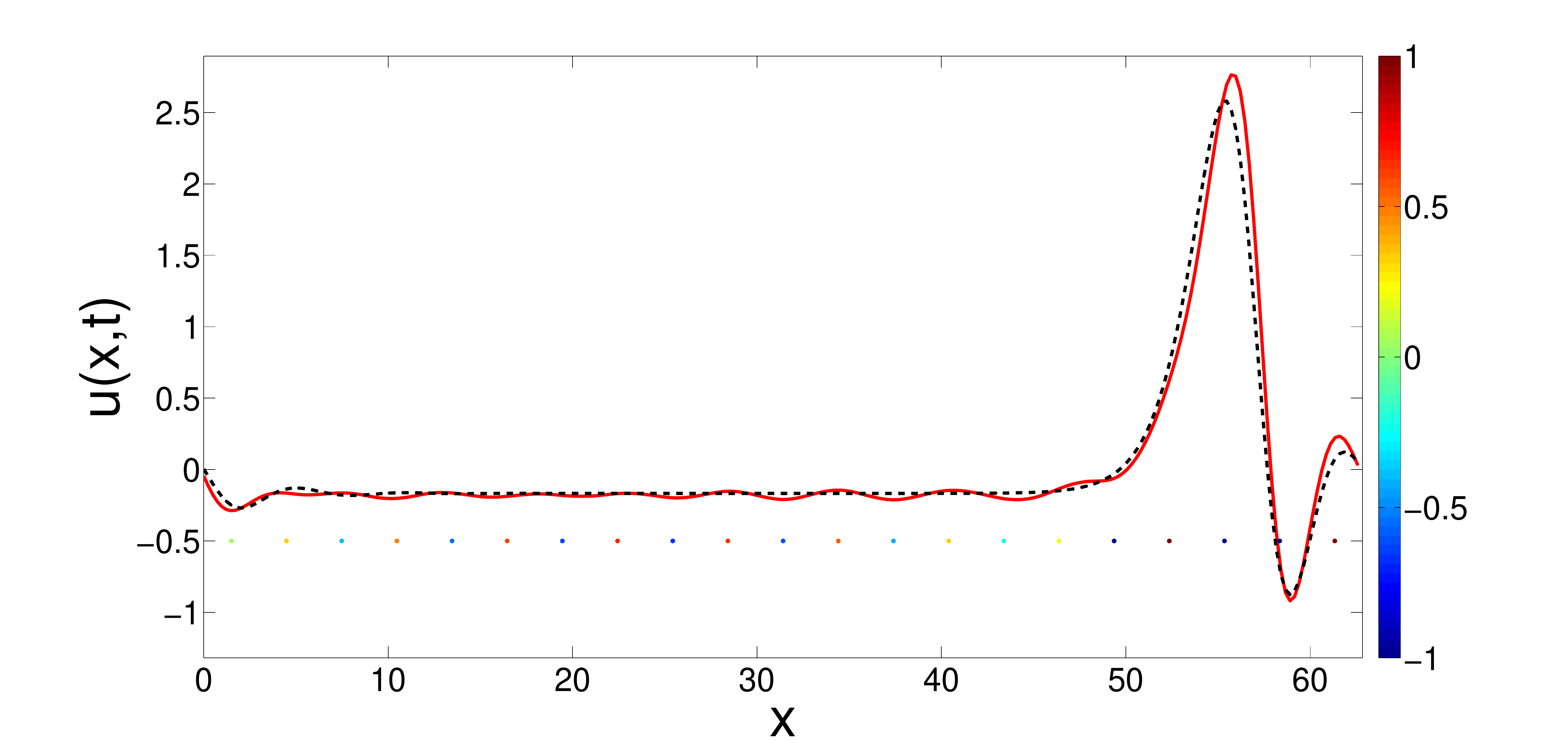}\label{fig:t200_nu}}
	 \caption{Snapshots of the time evolution of a stabilised travelling wave solution for $\delta = \mu = 0$ and assuming uncertainty in the parameter $\nu$. Black dashed line is the desired travelling wave (which is the correct solution for $\nu = 0.013 \Leftrightarrow L \approx 55$) and red full line is the controlled solution assuming $\nu = 0.01\Leftrightarrow L \approx 62$.}
	 \label{fig:compareControls_nu}
 \end{figure}

\begin{figure}[h!]
 \centering
	\subfloat[$t = 5$]{\includegraphics[width=0.5\linewidth]{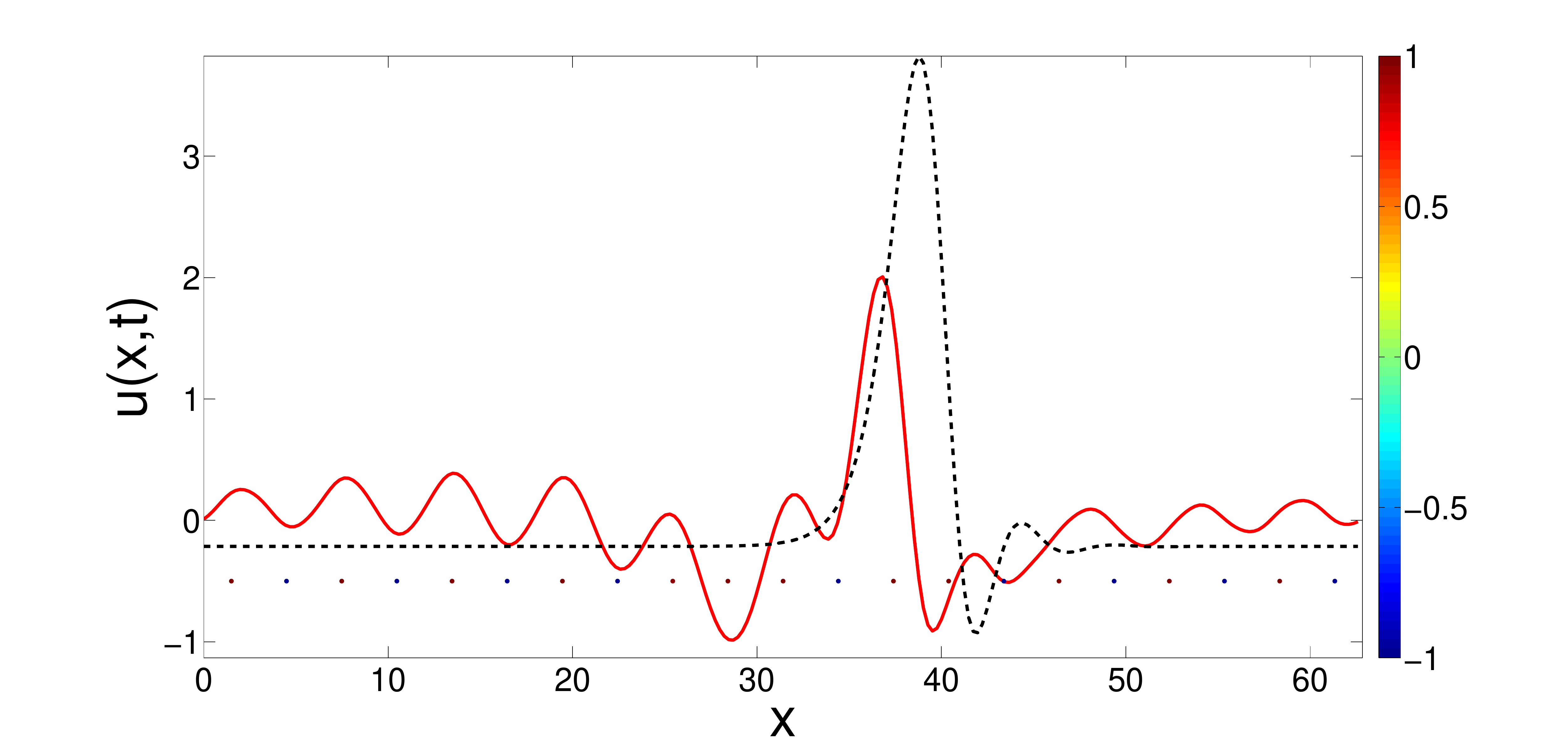}\label{fig:t5_delta}}
	\subfloat[$t = 20$]{\includegraphics[width=0.5\linewidth]{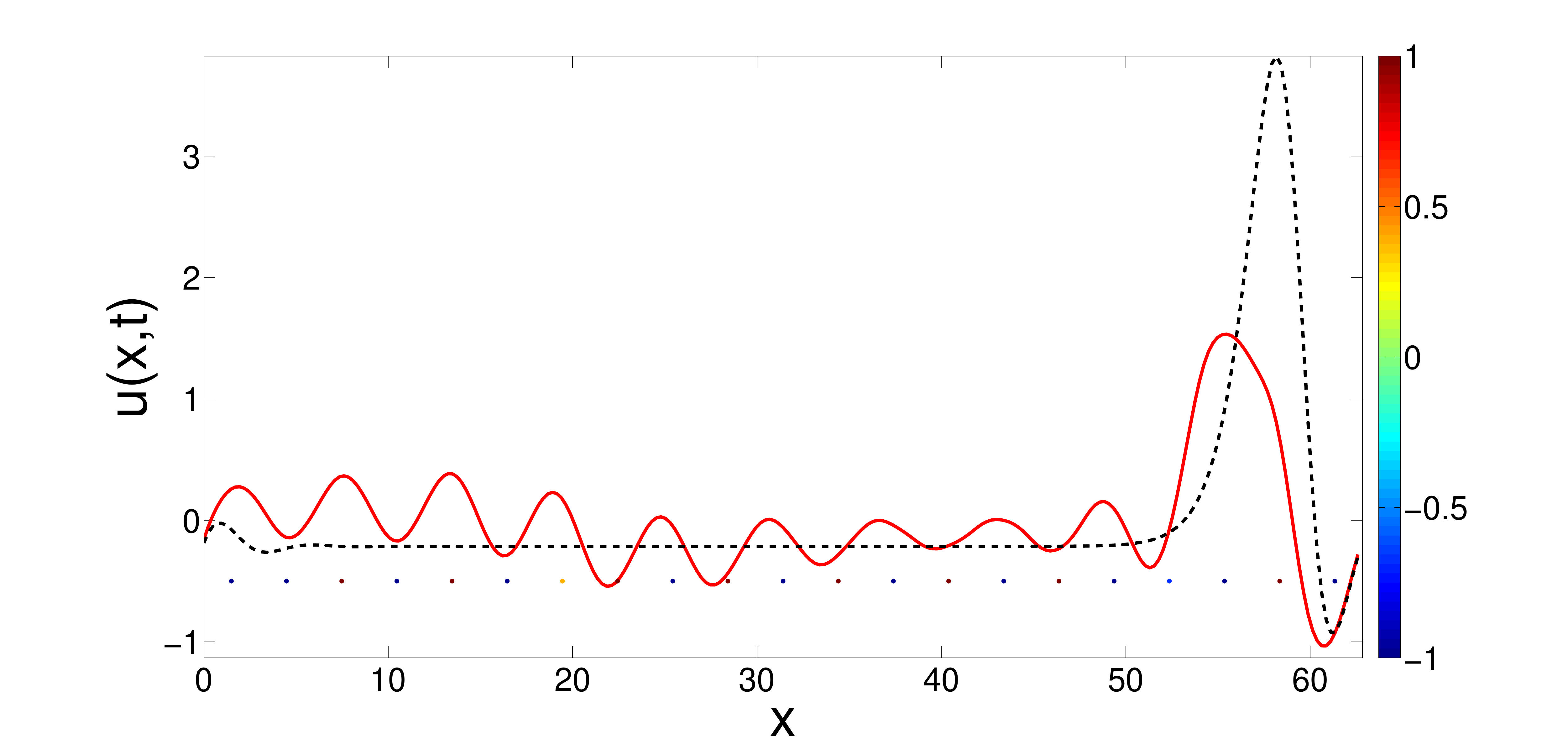}\label{fig:t20_delta}}

	\subfloat[$t = 30$]{\includegraphics[width=0.5\linewidth]{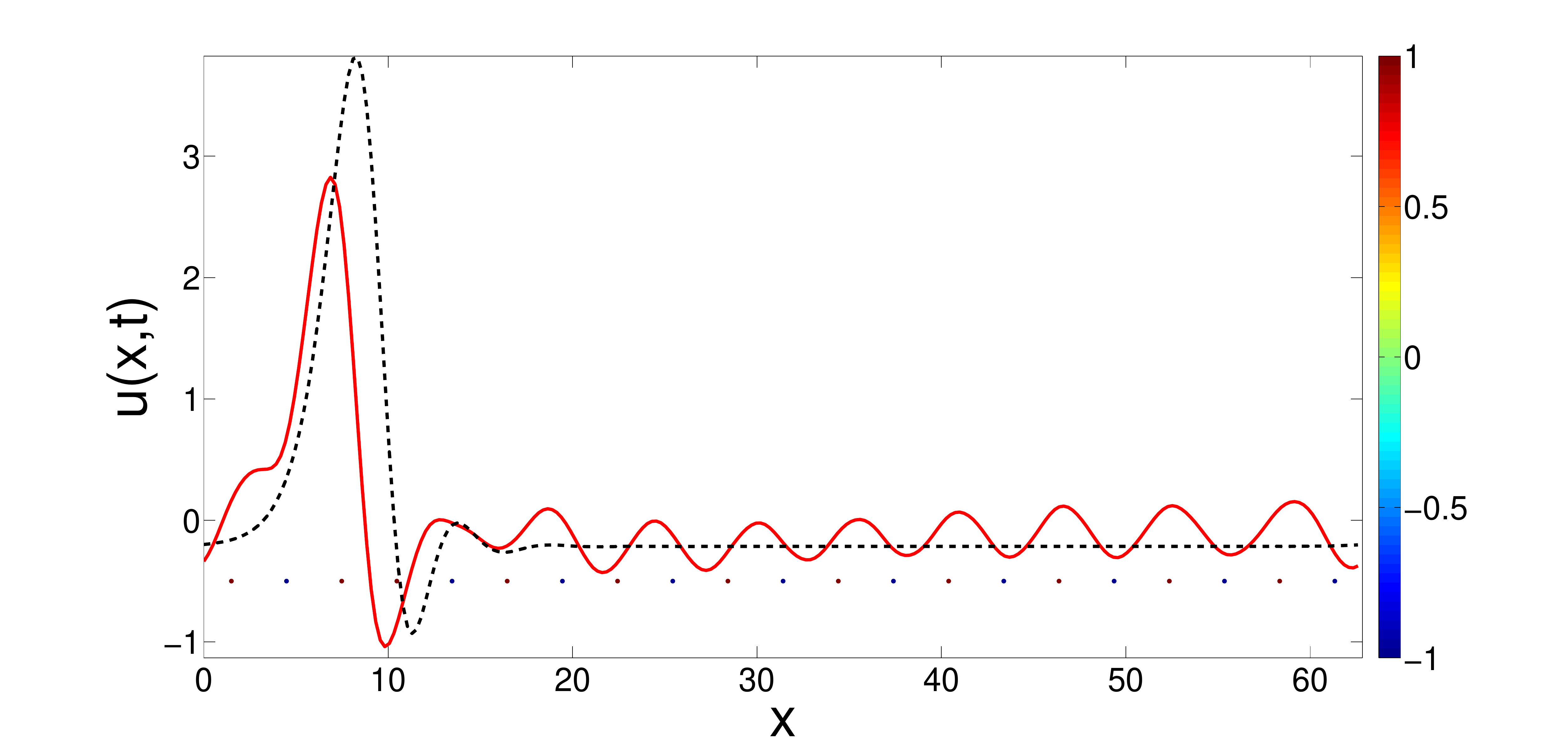}\label{fig:t30_delta}}
	\subfloat[$t = 60$]{\includegraphics[width=0.5\linewidth]{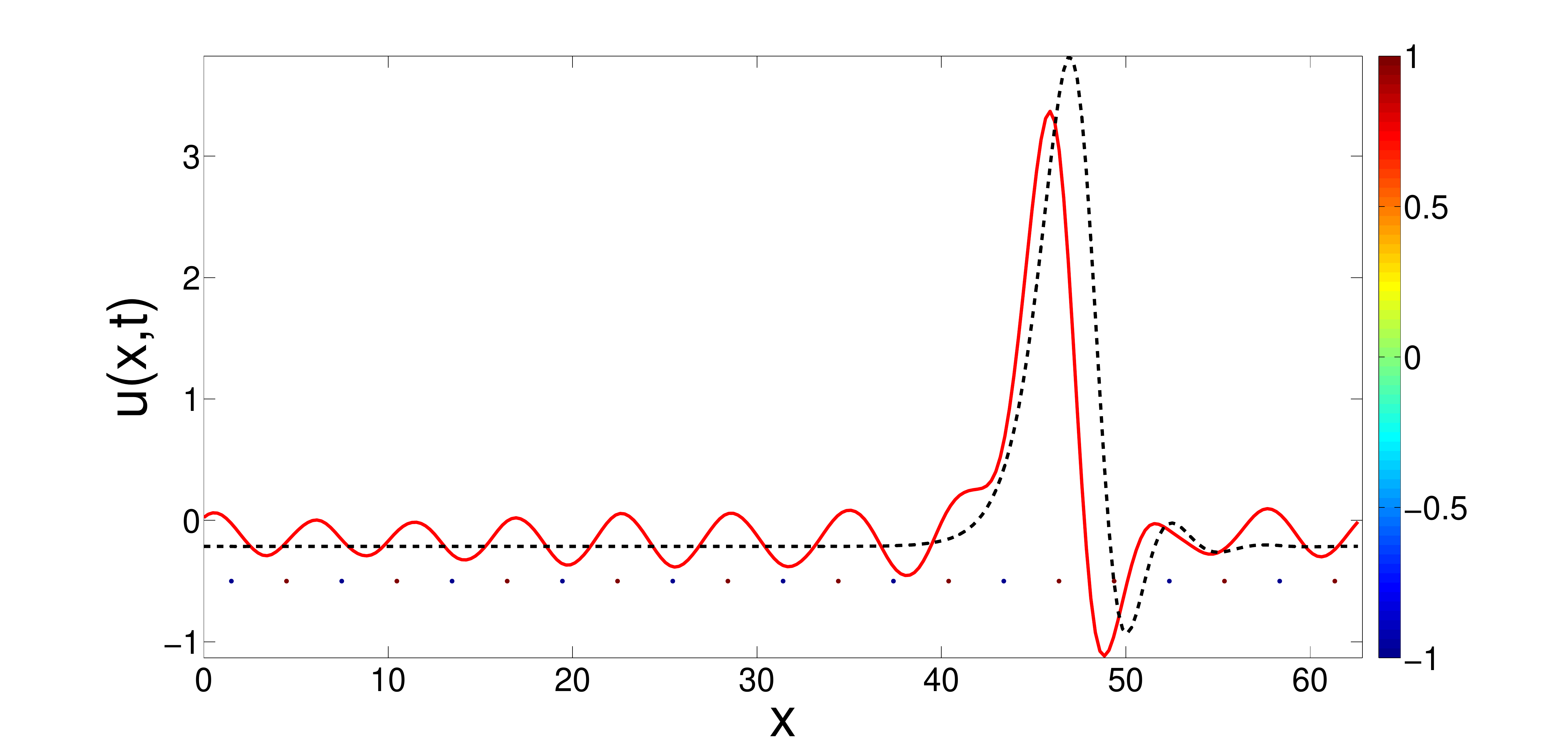}\label{fig:t60_delta}}

	\subfloat[$t = 90$]{\includegraphics[width=0.5\linewidth]{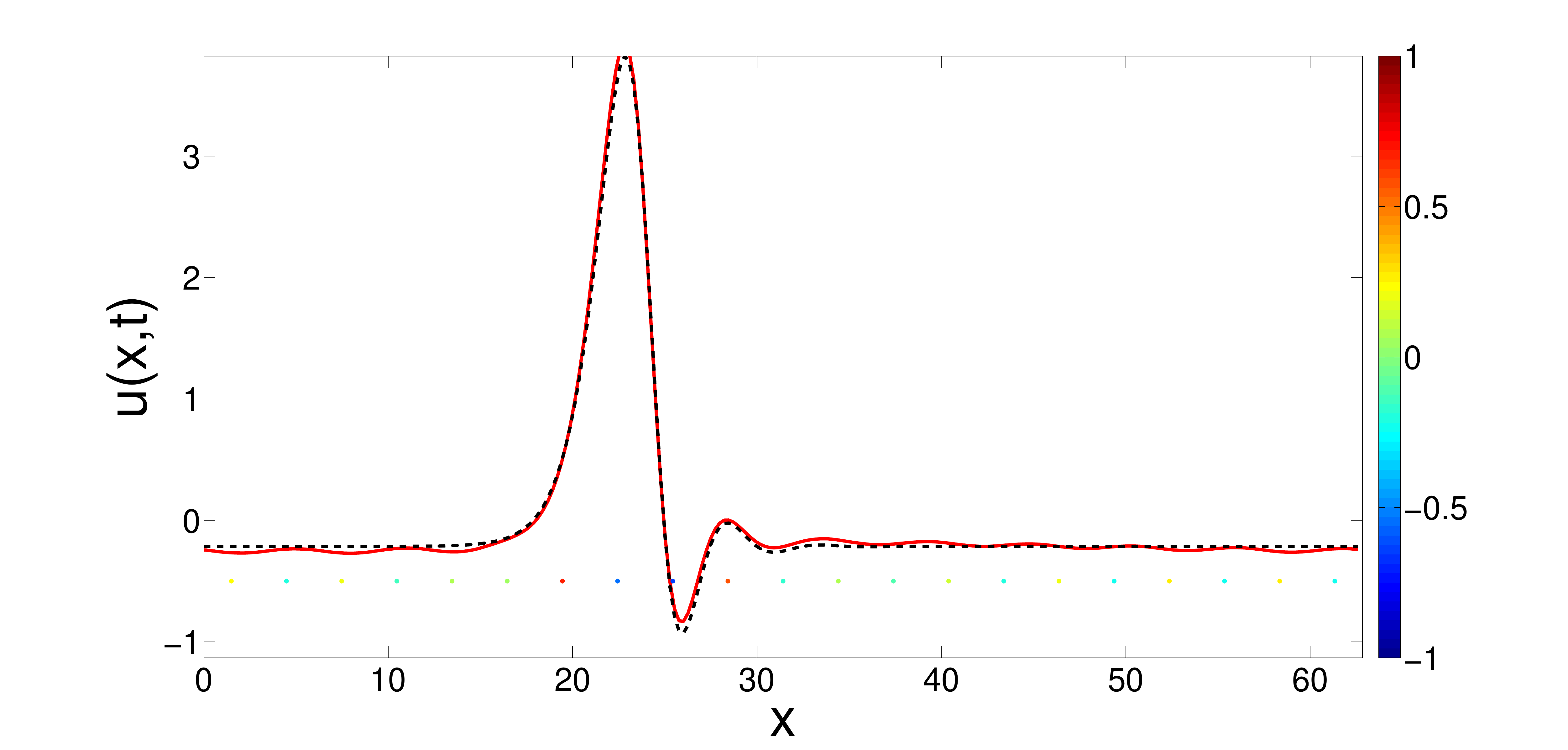}\label{fig:t90_delta}}
	\subfloat[$t = 200$]{\includegraphics[width=0.5\linewidth]{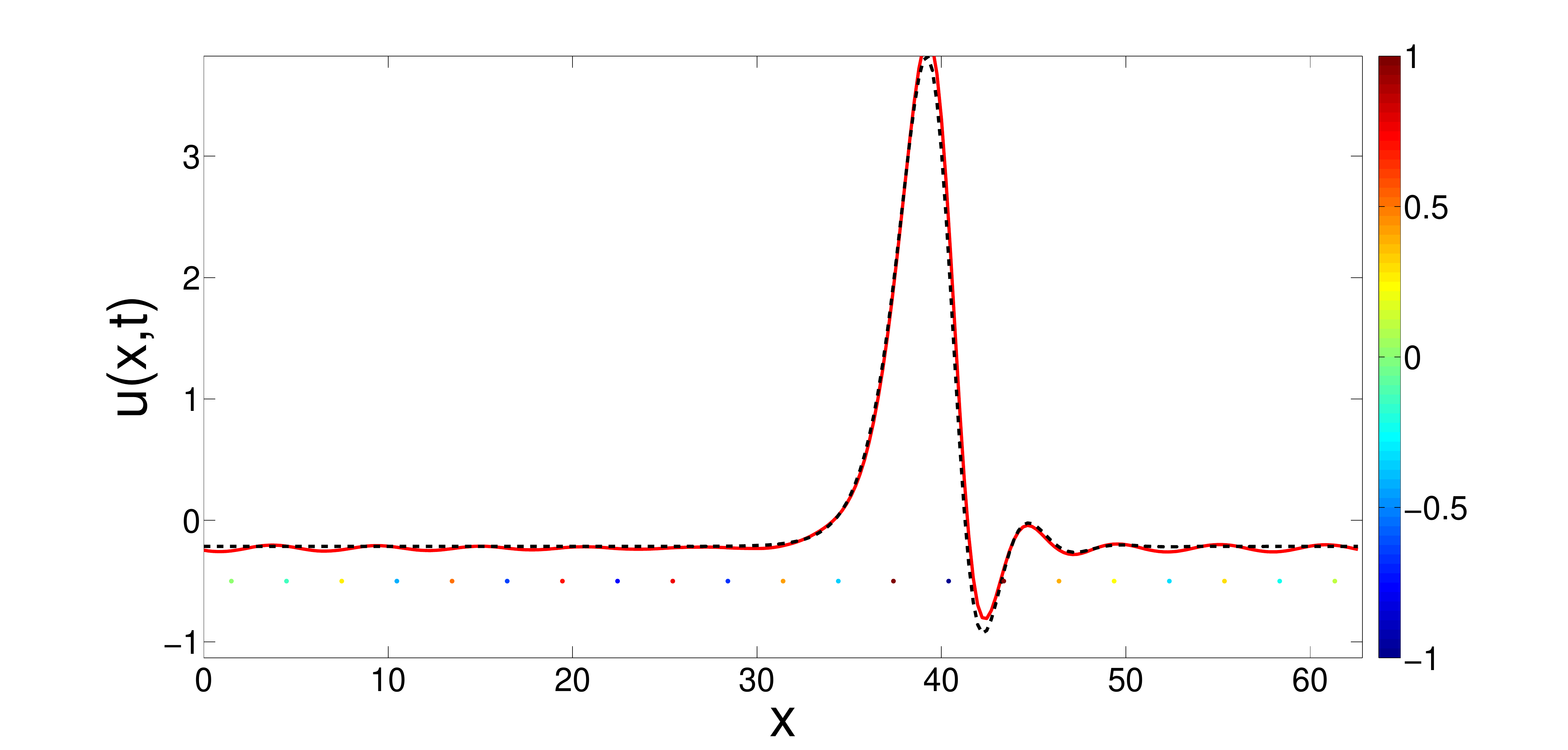}\label{fig:t200_delta}}
	 \caption{Snapshots of the time evolution of a stabilised travelling wave solution for $\mu = 0, \, \nu = 0.01$ and assuming uncertainty in the parameter $\delta$. Black dashed line is the desired travelling wave (which is the correct solution for $\delta = 0.03$) and red full line is the controlled solution assuming $\delta = 0.04$.}
	 \label{fig:compareControls_delta}
 \end{figure}

%
%
\section{Numerical Results}\label{sec:Numerical}
%
%
Section \ref{sec:stabilize} was devoted to proving rigorously that steady states and steady-state travelling wave
solutions of the generalised KS equation can be stabilised using linear feedback controls. The number of controls
is predicted to be at least as large as the number of linearly unstable modes, and robustness with respect to
uncertainty in the parameters $\nu$ and $\delta$ was also proved. In this section we implement the linear
feedback controls numerically and undertake an extensive computational study of the stabilisation and control in
practical situations.
\subsection{Construction of non-uniform steady states and travelling waves}
One of the main objectives of the present work is the stabilisation of unstable solutions of the gKS equation.
To obtain steady state solutions $\bar{u}(x)$ (in the absence of dispersion), we need to solve the equation
\begin{eqnarray}  \nu  \bar{u}_{xxxx} + \mu \hilb[\bar{u}_{xxx}] + \bar{u}_{xx} + \bar{u}
\bar{u}_x = 0,\label{eq:steadystate}\end{eqnarray}
in the interval $[0, 2 \pi]$, subject to periodic boundary conditions. Travelling waves of speed $c$ are found by looking for solutions of the form
$\bar{u}(x,t) = U(x-ct) = U(\xi)$ and solving
\begin{equation}\label{TWeq}
-cU' + \nu U'''' + \mu \hilb[U'''] + \delta U'''+  U'' +UU' = 0,
\end{equation}
subject to periodic boundary conditions, where primes denote differentiation with respect to $\xi$. We note that 
equation~\eqref{eq:steadystate} is a particular case of \eqref{TWeq}. Expressing the solutions in Fourier series
\begin{equation}\label{GalerkinTW}
U(\xi) = \sum_{n=1}^\infty U_n^s \sin(n\xi) + U_n^c \cos(n\xi),
\end{equation}
and substituting into \eqref{eq:steadystate} and \eqref{TWeq} we obtain an infinite system of nonlinear algebraic equations for the coefficients $U_n^s, \, U_n^c$, $n=1, \dots, \infty$, or for the coefficients and the velocity $c$, in the case of travelling waves. 
The resulting system of equations for steady states is
\begin{subequations}\label{coefsystem-SS}
\begin{align}
\label{first}(\nu n^4 - \mu n^3 -n^2)U_n^c +g_n^c  =   0, & \qquad n=1,\dots,\infty,\\
\label{second}(\nu n^4 - \mu n^3 - n^2)U_n^s + g_n^s  =  0, & \qquad n=1,\dots,\infty.
\end{align}
\end{subequations}
For travelling waves we can assume, without loss of generality due to translation invariance, that $U_1^s = 0$, to obtain
\begin{subequations}\label{coefsystem-TWs}
\begin{align}
\label{first}-(cn+\delta n^3)U_n^s + (\nu n^4 - \mu n^3 -n^2)U_n^c +g_n^c  =  0, & \qquad n=1,\dots,\infty,\\
\label{second} (cn+\delta n^3)U_n^c + (\nu n^4 - \mu n^3 - n^2)U_n^s + g_n^s  =  0, & \qquad n=2,\dots,\infty,\\
\label{third} (c+\delta)U_1^c + g_1^s  =  0, &
\end{align}
\end{subequations}
The systems were truncated and solved using a nonlinear solver (e.g. Matlab's \emph{fsolve}) to
find solutions to system \eqref{coefsystem-SS} by first setting $\mu = 0$ and carrying out a numerical continuation on $\nu$, and secondly by
fixing the desired value of $\nu$ and varying $\mu$.
For travelling waves we used continuation on $\nu, \, \mu$ and $\delta$. Without loss
os generality we also impose $c>0$: if $U(x-ct)$ is a solution of \eqref{TWeq} with $c<0$, then $-U(-x-(-c)t)$ is also a solution with $c>0$.

Given the Fourier coefficients and the velocity of a travelling wave, we can write the soluton of the KS equation as  
\begin{equation}
\begin{array}{rl}
\bar{u}(x,t) = U(x-ct) = &\sum_{n=1}^\infty \left(U_n^s\cos(nct) + U_n^c\sin(nct)\right)\sin(nx) +  \\
 & \sum_{n=1}^\infty \left( U_n^c\cos(nct) - U_n^s\sin(nct)\right)\cos(nx).\label{conversion}
\end{array}
\end{equation}

Our computational results are presented in the bifurcation diagram in Figure \ref{fig:mu>0} that depicts the variation of
the $L^2$-norm with $\nu$ of the steady states and travelling wave solutions of the gKS equation \eqref{eq:KS2} in
the absence of dispersion ($\delta=0$). Panels (a)-(d) correspond to $\mu=0,\,0.2,\,0.5,\,1.0$;
 steady-states are plotted with solid curves (blue online) and travelling waves with dashed curves (red dashed online).
We observe that the presence of the Hilbert transform increases the value of $\nu$ for which instability arises~\cite{Tseluiko2006}, 
but it does not change the shape of the bifurcation diagram. This is because the Hilbert transform term acts as a negative diffusion, see Equation~\eqref{eq:lambda}, and therefore its presence acts to
shift the bifurcation diagram to higher $\nu$, i.e. lower $\alpha=4/\nu$ as seen in the figure. 
We emphasise that the bifurcation diagrams in Figure~\ref{fig:mu>0} are not complete and
we expect additional unstable branches in analogy with known results for the KS equation \cite{Kevrekidis1990}.
This is not a restriction here, since we are interested in demonstrating the stabilisation of unstable
steady or travelling wave solutions, rather than the stabilisation of all such branches.
For the branches computed here, we analysed their stability
numerically by adding a small perturbation to the initial condition
(about $10\%$ or smaller of the amplitude of the steady state solution) and studied the time evolution to ensure that
we identified unstable steady solutions to be stabilised using linear feedback controls.

\subsection{Time dependent simulations and feedback control}
We used a Galerkin truncation~\cite{Trefethen2000} for the spatial discretisation of the PDE, with the number of modes varying between $32$, $64$ and $128$ depending on the number of unstable modes. Time integration is
carried out using second order implicit-explicit backward differentiation formulae (BDF) schemes \cite{Akrivis2011,Akrivis2012}.

To construct the matrix $K$ necessary for the stabilisation of the steady states, we used \emph{Matlab}'s command \emph{place}. 
Given the matrices $A$ and $B$, we sought a matrix $K$ such that the eigenvalues of the matrix $A+BK$ were:
\begin{itemize}
\item $-1$ if it is the eigenvalue corresponding to the constant eigenfunction $\frac{1}{\sqrt{2\pi}}$.
\item $\pm\lambda$ if $\lambda$ is an eigenvalue of $A$ with negative/positive real part.
\item $-10\delta\lambda$  instead of $-\lambda$ if $\delta>0$. We do this because the
 amplitude of the solutions grows with $\delta$ - \cite{Kawahara1998}, so we need to account for this when building the controls.
\end{itemize}
We begin by presenting numerical results in the absence of electric fields and dispersion ($\mu=0$, $\delta=0$) and for
two values of $\nu=0.2$ and $\nu=0.4$ (note that the number of unstable eigenvalues is $2l+1$ where $l=[\nu^{-1/2}]$,
where $[\cdot]$ denotes the integer part.
The number of controls used is $5$ and $3$, respectively, i.e. equal to $2l+1$; these are placed equidistantly and the
initial condition is
\[
u_0(x) = \frac{1}{\sqrt{2\pi}} + \frac{1}{\sqrt{\pi}}\sum_{n = 1}^5 \left(\sin(nx) + \cos(nx)\right).
\]
The results are presented in Figure \ref{fig:compareChristofides} and clearly show that the system is controlled to the
zero solution long before the final computed time of $t=5$; our results are also in good agreement with those in  \cite{Christofides2000}.
\begin{figure}[h!]
 \centering
	\subfloat[$\nu = 0.2$]{\includegraphics[width=0.5\linewidth]{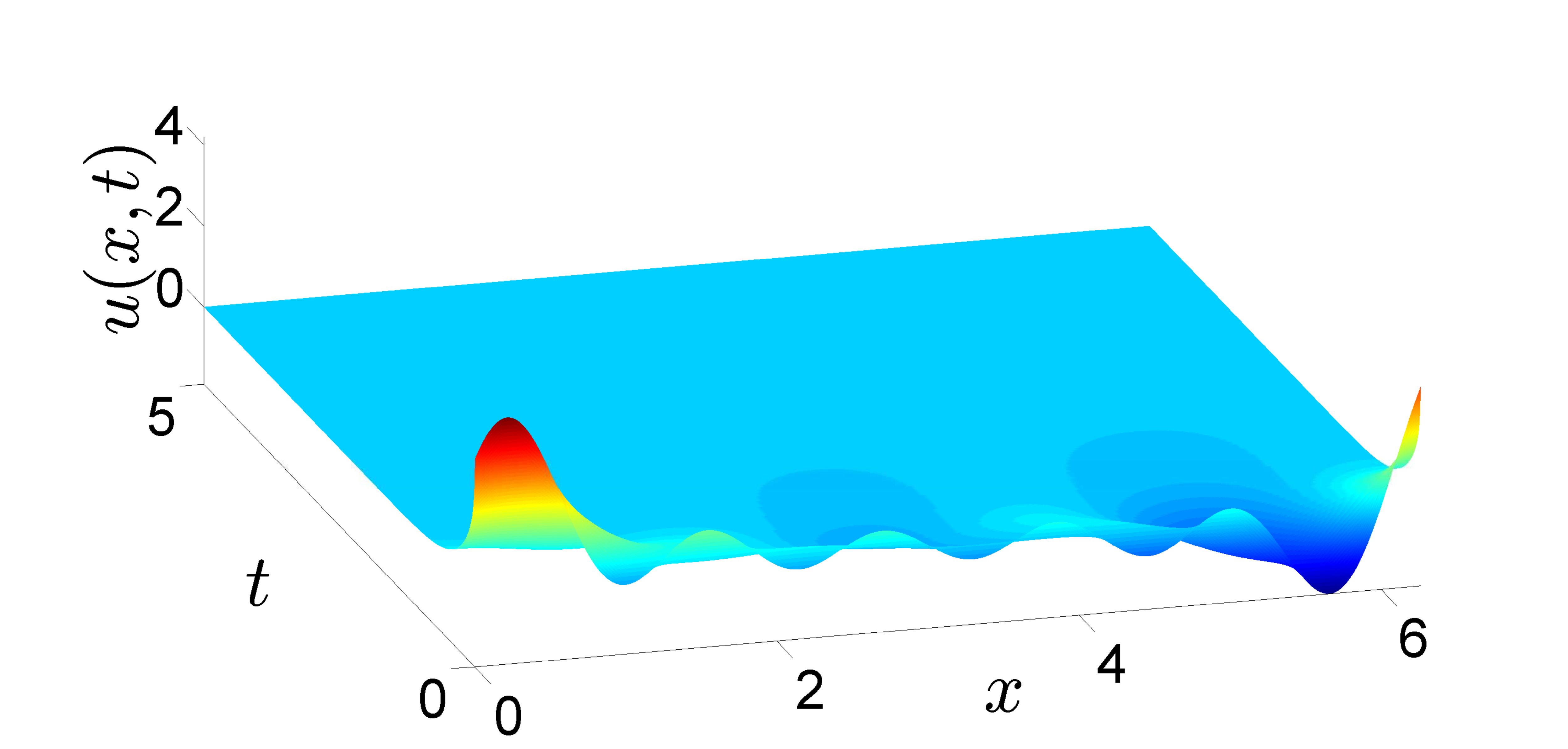}\label{fig:nu02}}
	\subfloat[$\nu = 0.4$]{\includegraphics[width=0.5\linewidth]{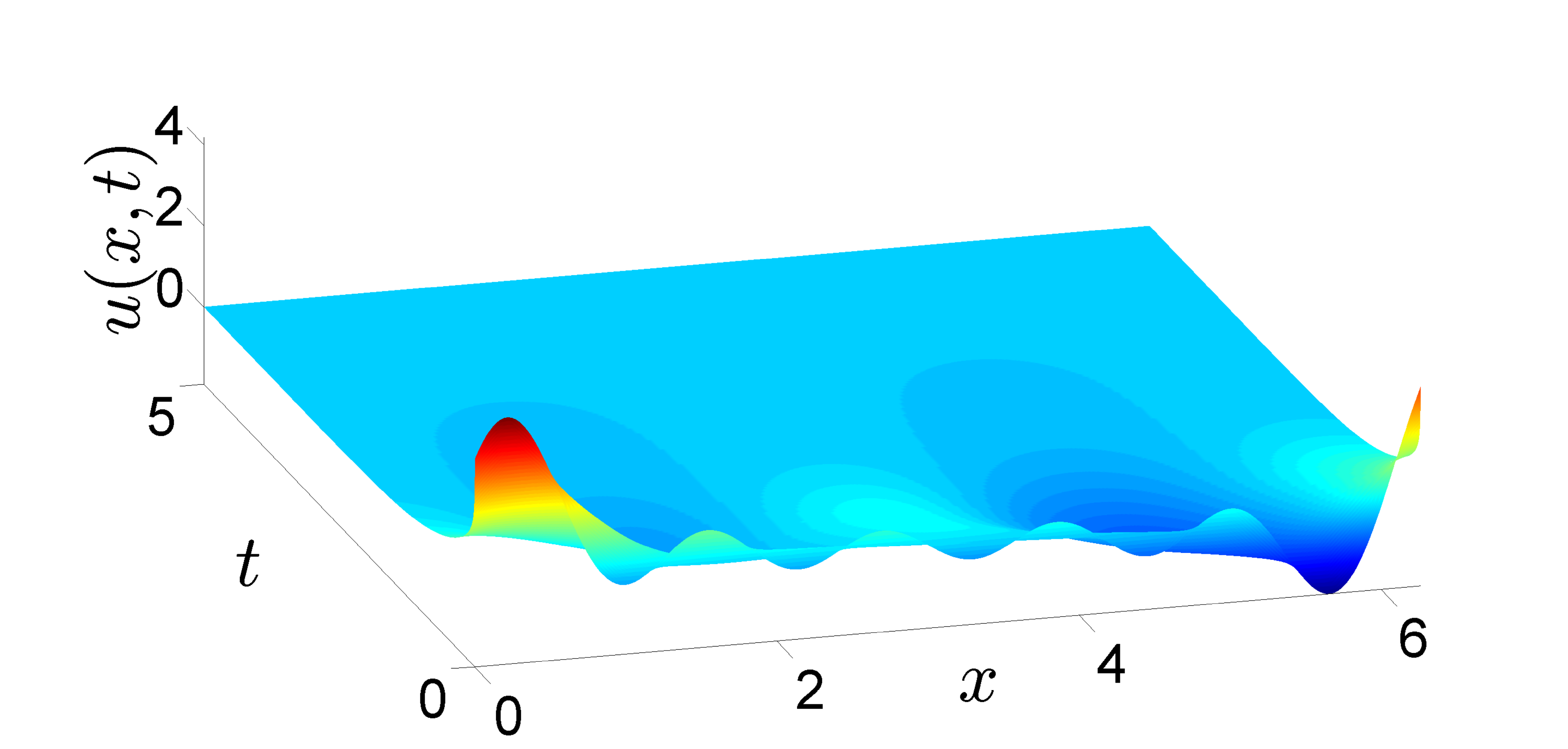}\label{fig:nu04}}
	 \caption{Spatiotemporal evolution showing stabilisation to the zero solution of the KS equation for (a) $\nu = 0.2$ ($\alpha = 20$),
	   and (b) $\nu = 0.4$  ($\alpha = 10$).}
	 \label{fig:compareChristofides}
 \end{figure}

Results analogous to those presented in Figure \ref{fig:compareChristofides} were found regarding the
stabilisation of the zero solution to the KS equation in the presence of an electric field. 
In what follows we use the following initial condition unless stated otherwise:
\begin{equation}\label{eq:initial}
u_0(x) = \frac{1}{\sqrt{\pi}}\left(\sin(x) + \cos(x)\right).
\end{equation}
Note that the number of unstable modes is $2l+1$ where $l=\left[\frac{\mu + \sqrt{\mu^2+4\nu}}{2\nu}\right]$: see 
Proposition \ref{prop1}- and this is the number of controls used in the numerical experiments.
The numerical results for $\nu = 0.2$ and $\mu = 0.5$ with $5$ equidistant controls
are shown in Figure~\ref{fig:nu02mu05}, where we again clearly observe
stabilisation to the zero solution.
\begin{figure}[h!]
	\centering
	\includegraphics[scale=.3]{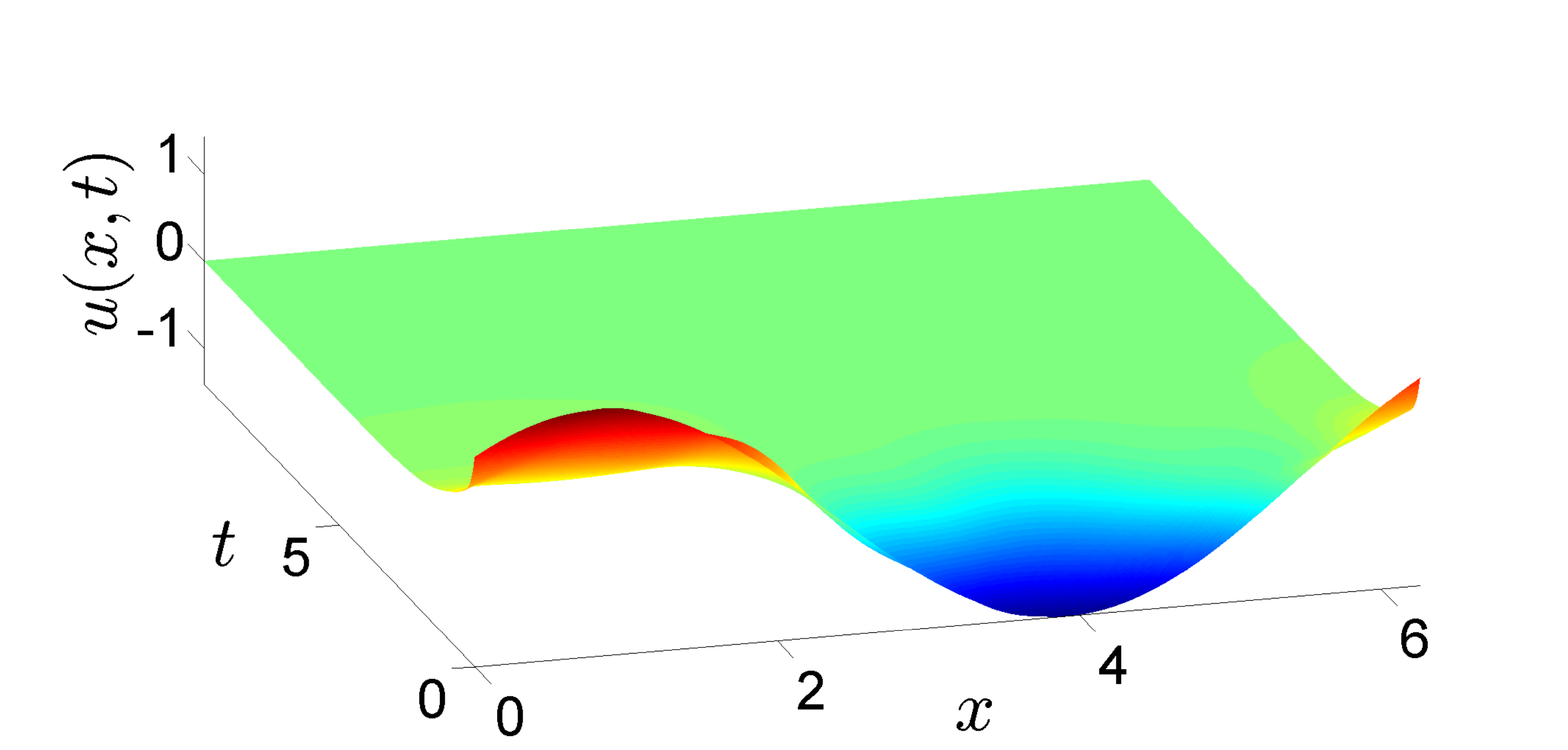}
	\caption{Spatiotemporal evolution showing stabilisation to the zero solution of the KS equation in the presence of an electric field 
	with $\mu = 0.5$ and $\nu = 0.2$ ($\alpha = 20$).}
	\label{fig:nu02mu05}
\end{figure}

Having shown the stabilisation of zero states for relatively small values of $\nu$, we turn next to the stabilisation of
nontrivial steady states of the generalised KS equation \eqref{eq:KS2}, in the absence of dispersion.
We illustrate the feasibility of our control methodology for two typical cases that yield unstable steady states as computed
in the bifurcation diagram of Figure~\ref{fig:mu>0}. In the first case we use $\nu = 0.1115, \, \mu = 0$, and in the second
$\nu = 0.35, \, \mu = 0.3$. In both cases we used $2l+1$ equidistant controls, i.e., the same as the number of unstable eigenvalues of the system.
The results of our numerical experiments are presented in Figures~\ref{fig:nu01115steadystate} and~\ref{fig:nu035mu03steadystate}, respectively.
\begin{figure}[h!]
	\centering
	\includegraphics[width = 0.8\linewidth]{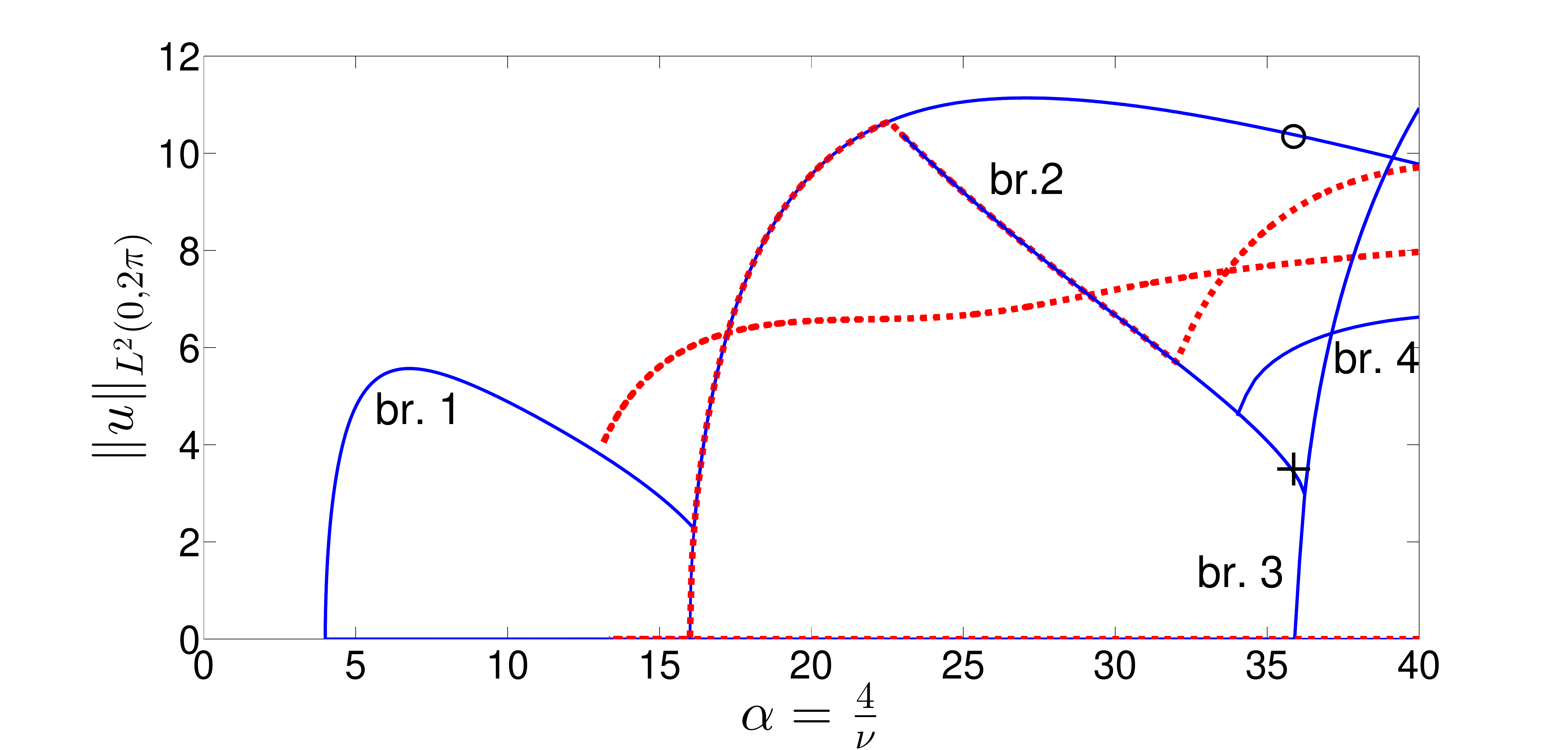}
	\caption{Zoom in of Panel \ref{fig:mu>0}(a) with $\nu \in [0.1,1]$. Branches are labelled as used in Tables~\ref{tab2.1}-\ref{tab2.3}: Branch 1 unimodal steady states;
	branch 2 - bimodal steady states; branch 3 - trimodal steady states; branch 4 - tetramodal steady states.
	The cross and open circle symbols indicate the steady states (stable and unstable, respectively) that are shown in the Figure \ref{fig:nu01115steadystate}.}
	\label{fig:zoomBD}
\end{figure}
When $\nu=0.1115,\,\mu=0$, i.e. $\alpha\approx 35.87$, both stable and unstable steady states coexist
and the solution of the PDE with a given initial condition, e.g. \eqref{eq:initial}, evolves to the most attracting
stable state. This is shown in figure~\ref{fig:nu01115steadystate}(a) where it is seen that the solution evolves to a stable
bimodal steady state, marked with a circle in Figure~\ref{fig:zoomBD}. We are interested in using feedback control to stabilise one of the coexisting unstable steady states, and the
results of achieving this are presented in Figures \ref{fig:nu01115steadystate}(b)-(c); panel \ref{fig:nu01115steadystate}(b) shows the
evolution of the initial condition \eqref{eq:initial} using $2l+1=5$ equidistant controls and stabilisation of the steady state marked with a $+$ in Figure~\ref{fig:zoomBD} is achieved relatively quickly
after approximately $2$ time units. The evolution of the amplitudes of the $5$
applied controls is shown in Figure \ref{fig:nu01115steadystate}(c), and we see that the required energy tends to values
very close to zero as time evolves. Note that the control amplitudes remain small and close to zero once the unstable
controlled state is reached, but they cannot be identically zero due to the unstable nature of the controlled solution.
\begin{figure}[h!]
	\centering 
	\subfloat[Uncontrolled Solution]{	\includegraphics[width = 0.5\linewidth]{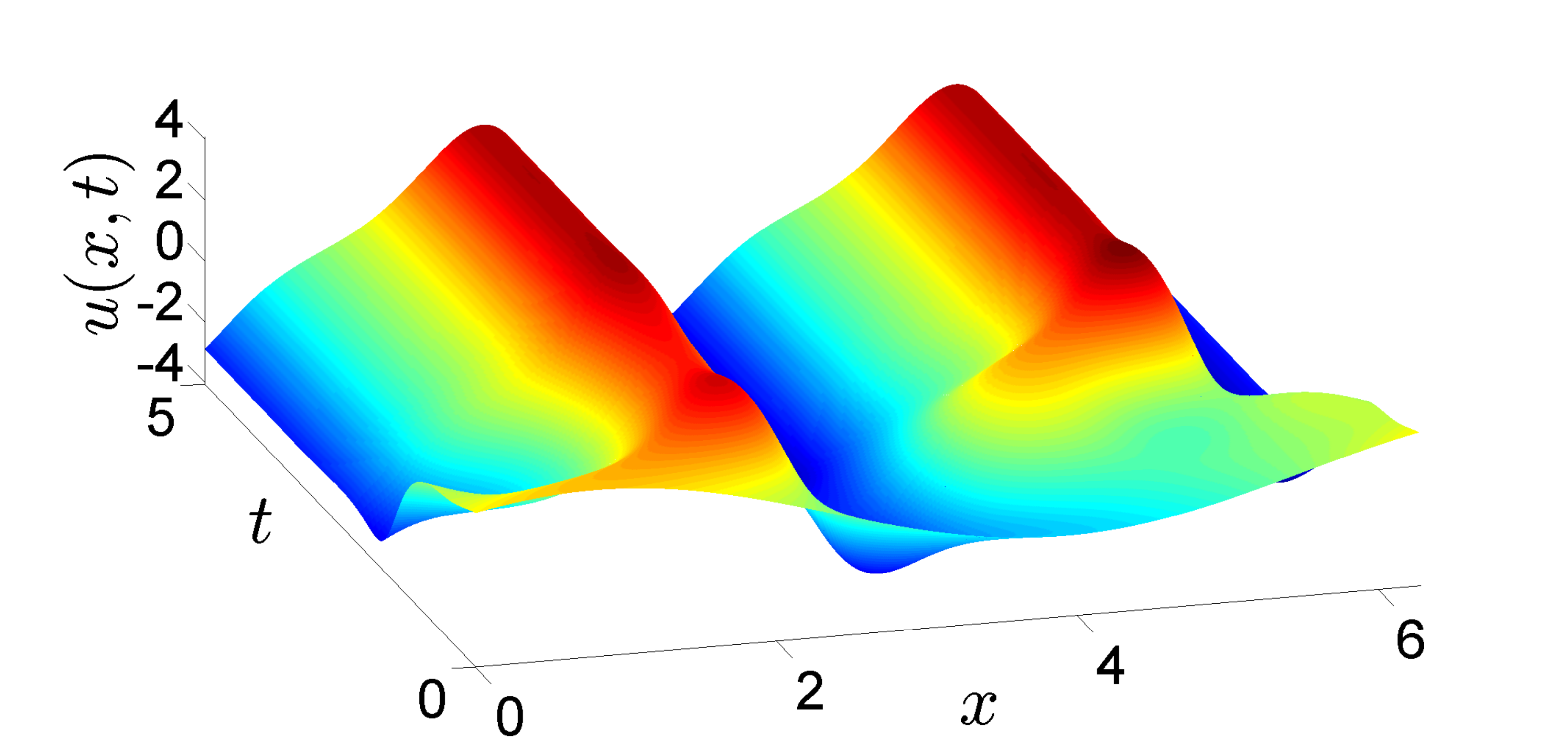}\label{unc}}
	\subfloat[Controlled Solution]{	\includegraphics[width = 0.5\linewidth]{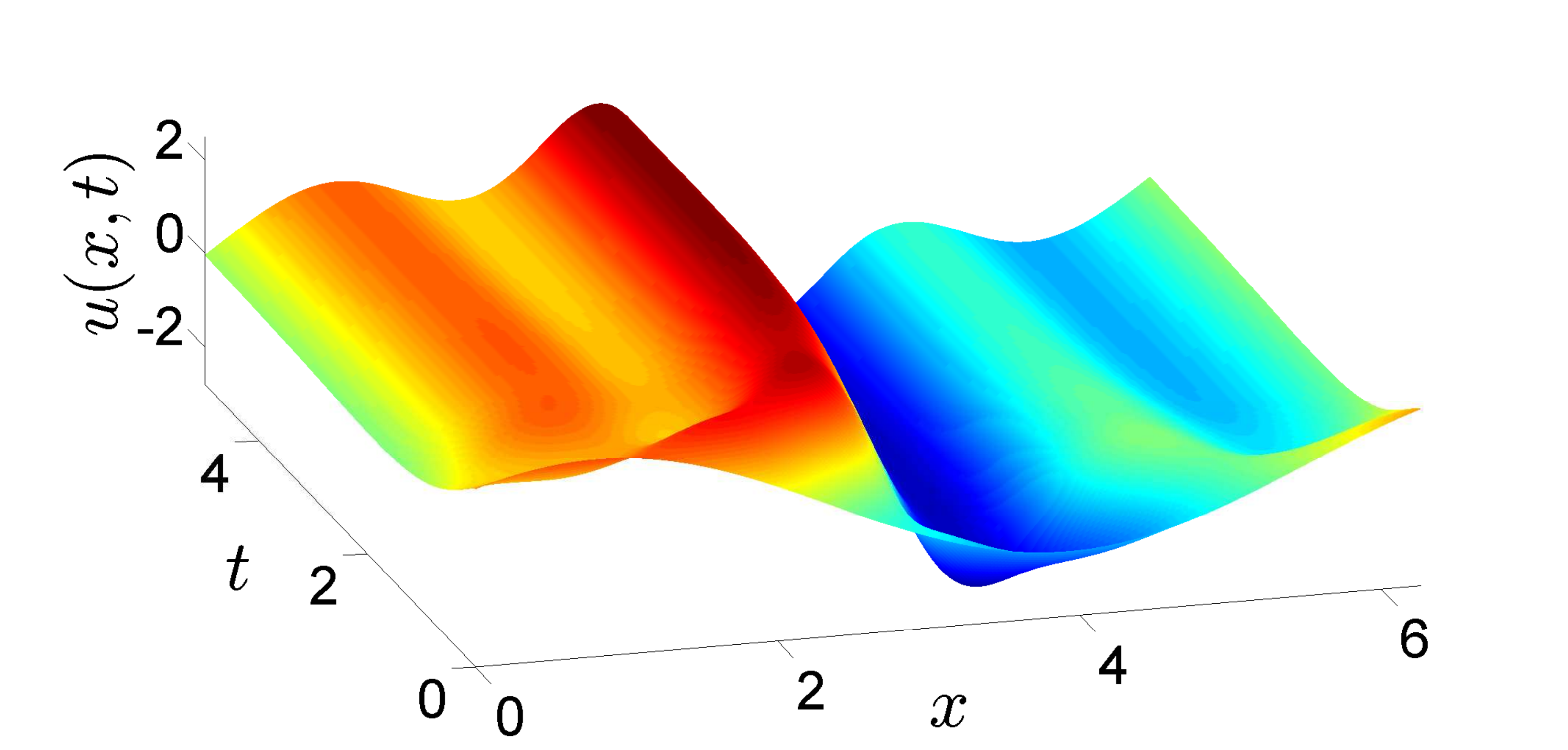}\label{contr}}

	\subfloat[Controls]{	\includegraphics[width = 0.6\linewidth]{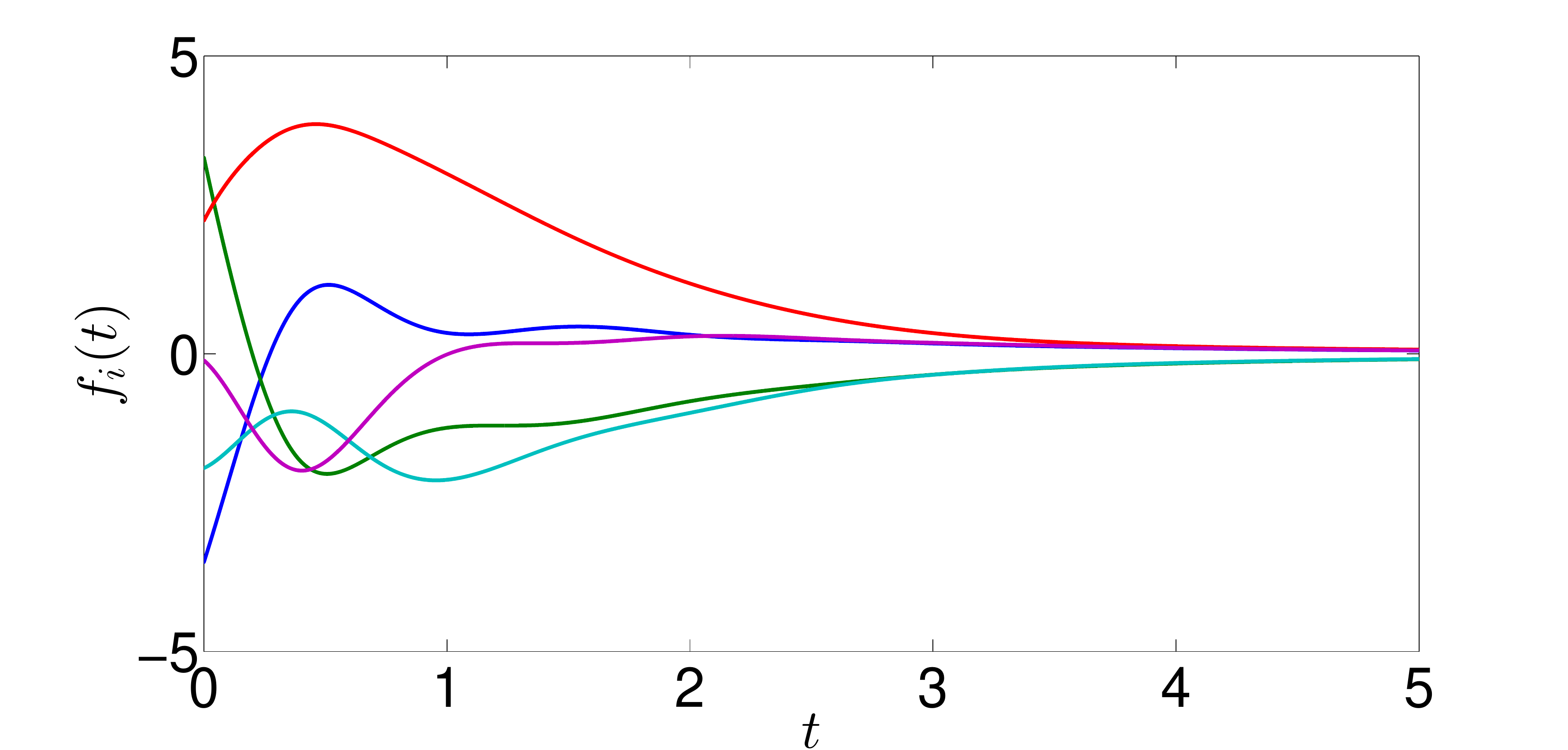}\label{ctrls}}
	 \caption{Control of non-uniform solutions of the KS equation for $\nu = 0.1115$; panel (a) spatiotemporal
	 evolution without controls (solution belongs to branch 1 of the bifurcation diagram in Figure~\ref{fig:bifdiag3});
	 panel (b) controlled to the steady state in branch 4 of the bifurcation diagram in Figure~\ref{fig:bifdiag3};
	 panel (c) evolution of the amplitude of the $5$ applied controls.}
	 \label{fig:nu01115steadystate}
\end{figure}
Figure \ref{fig:nu035mu03steadystate} shows the results for $\nu=0.35,\,\mu=0.3$. The solution we choose to stabilise
 at these values is an unstable 
bimodal steady state and Figure \ref{fig:nu035mu03steadystate} shows how it is stabilised using $2l+1=5$ controls.

\begin{figure}[h!]
	\centering 
	\includegraphics[scale=.3]{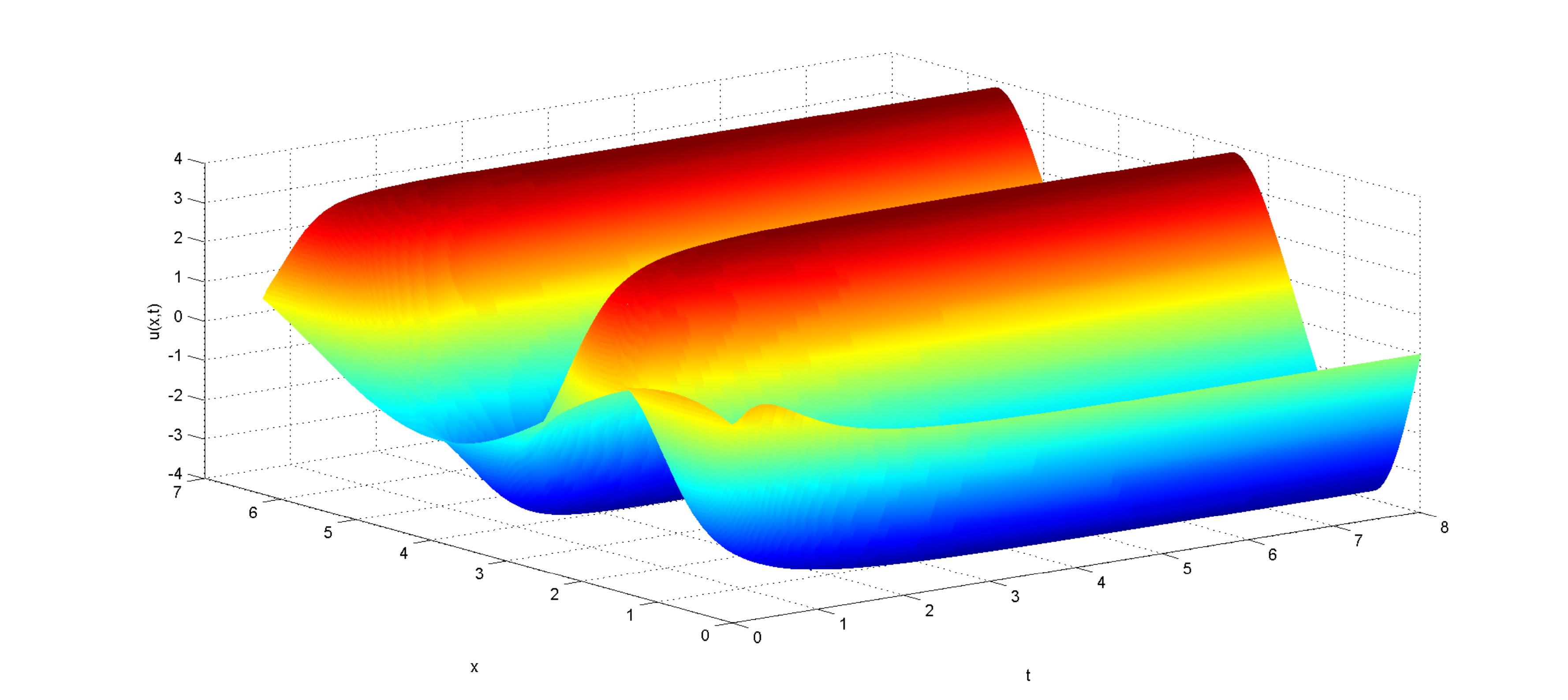}
	 \caption{Spatiotemporal evolution of the stabilised steady state of the Kuramoto-Sivashinsky equation for $\nu = 0.35$ ($\alpha\approx 11.43$), $\mu = 0.3$ .}
	 \label{fig:nu035mu03steadystate}
\end{figure}

Our last task is to stabilise travelling wave solutions of the equation with and without dispersion and electric field. Figure~\ref{fig:delta0} illustrates the stabilisation of three different travelling wave solutions to the KS equation \eqref{eq:KS2} with no dispersion or electric field ($\delta = \mu = 0$) and for a small value of $\nu$ ($\nu = 0.01$) which corresponds to a very large domain ($L = 20\pi \approx 62$) that enables the existence of single pulse travelling waves as well as two- or three-pulse bound states. However, due to the small value of $\nu$, when solving the PDE the initial condition evolves to a solution that exhibits the spatiotemporal chaotic behaviour that is characteristic of this equation. Panel~\ref{fig:delta0}(a) shows this chaotic behaviour while Panels~(b)-(d) show the evolution of the controlled solution to $1$, $2$ and $3$ pulses, respectively. We used $m=21$ equidistant controls in each case.

\begin{figure}[!htbp]
	\subfloat[Uncontrolled solution]{\includegraphics[width=0.5\linewidth]{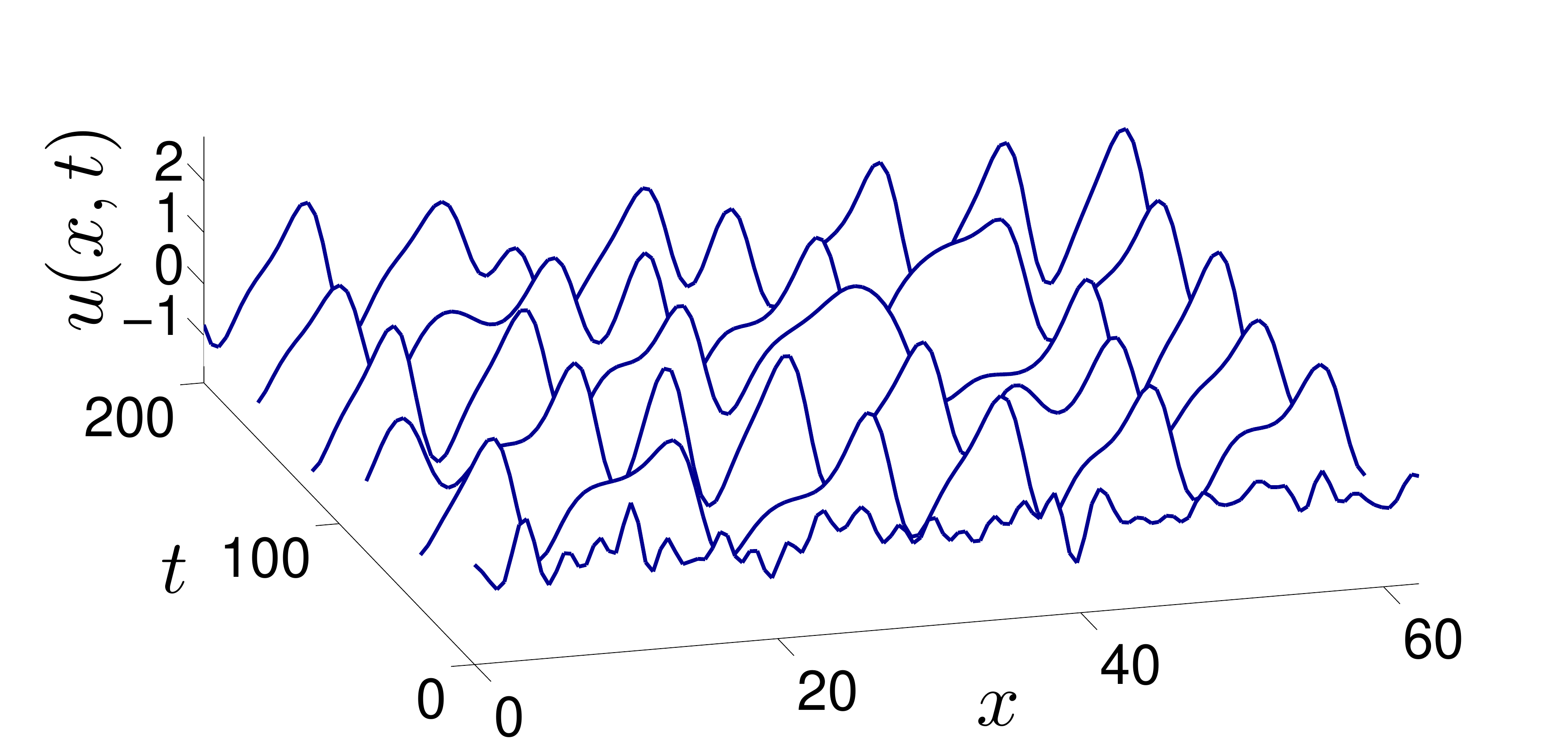}\label{fig:delta0_unc}}
	\subfloat[Controlled to one pulse]{\includegraphics[width=0.5\linewidth]{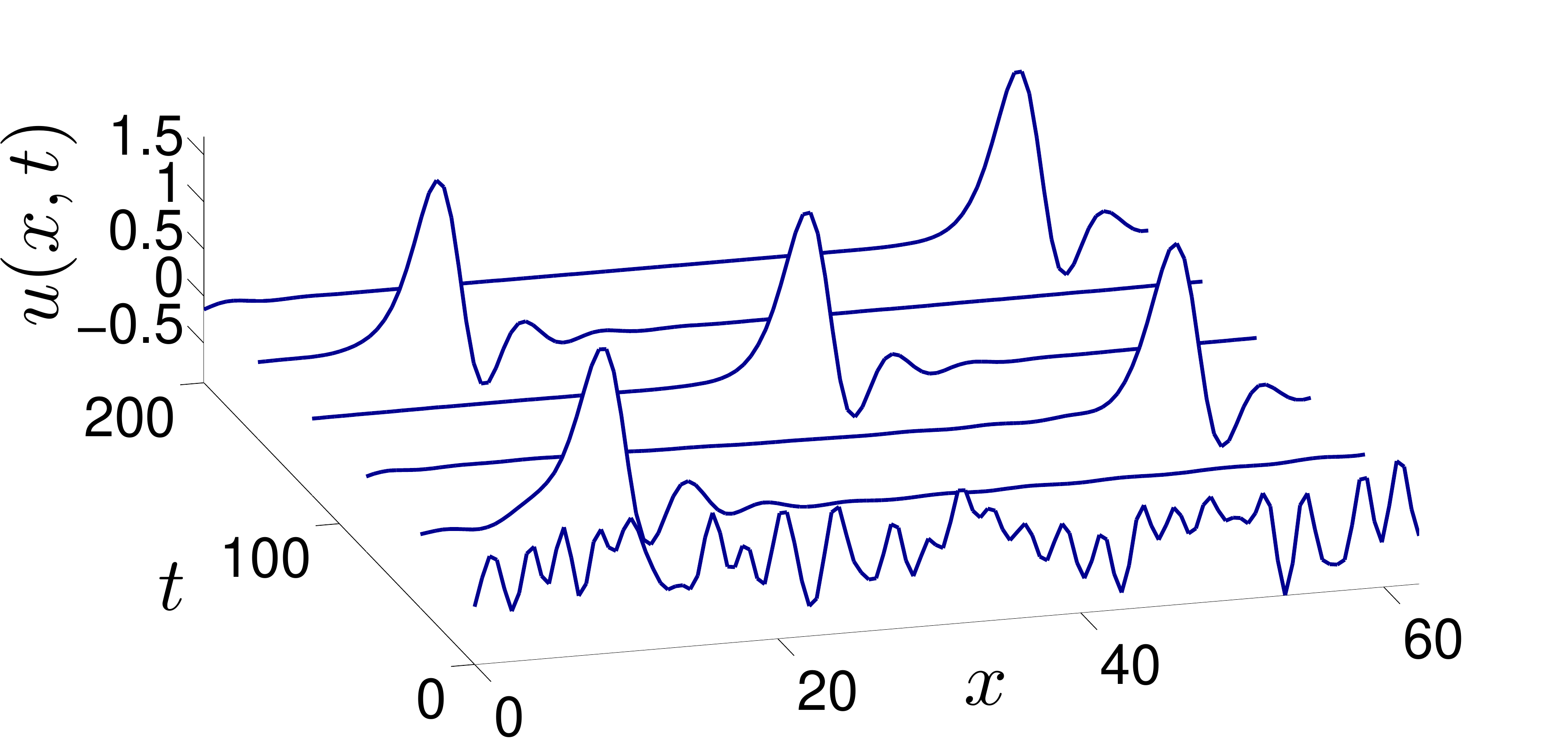}\label{fig:delta0_1pulse}}

	\subfloat[Controlled to two pulses]{\includegraphics[width=0.5\linewidth]{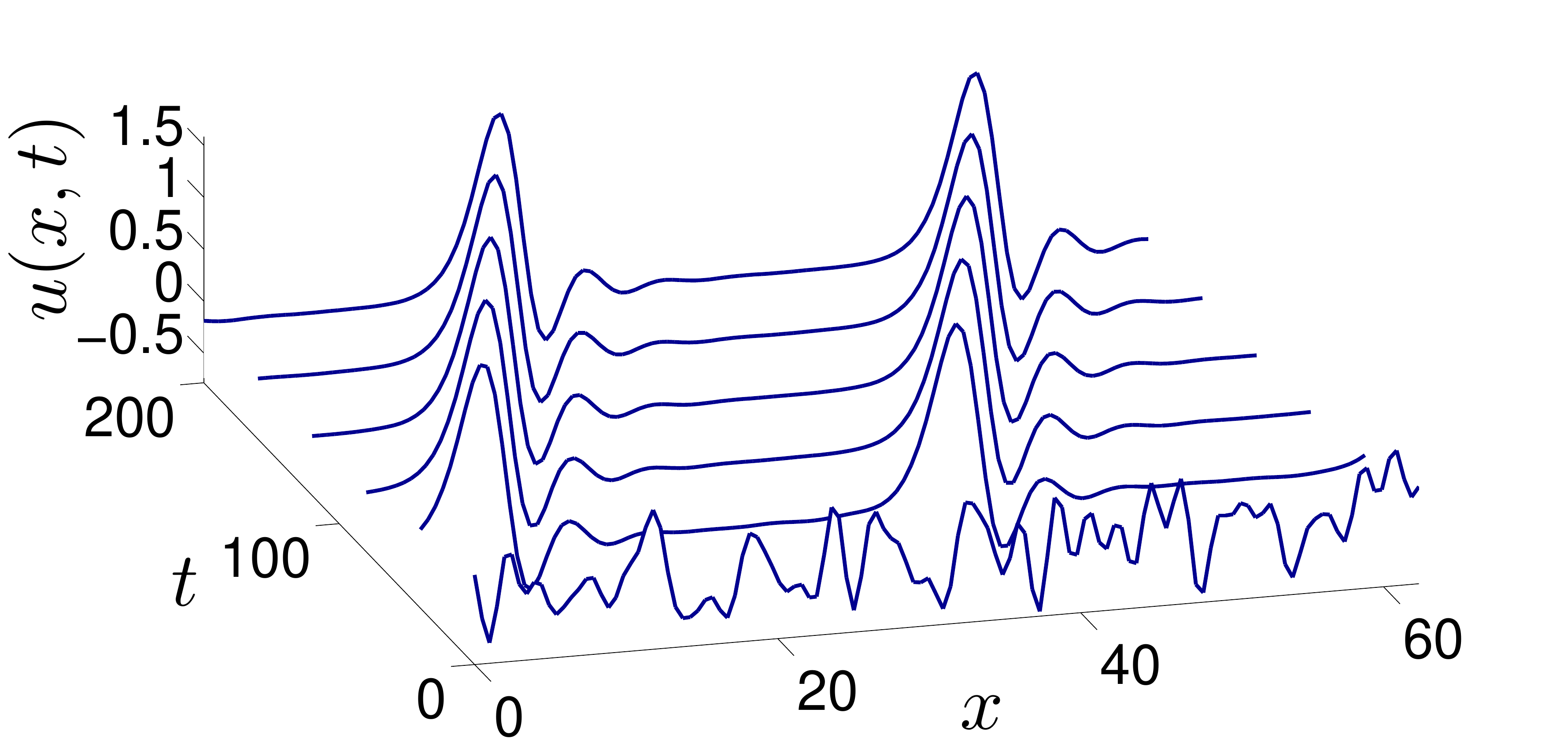}\label{fig:delta0_2pulses}}
	\subfloat[Controlled to three pulses]{\includegraphics[width=0.5\linewidth]{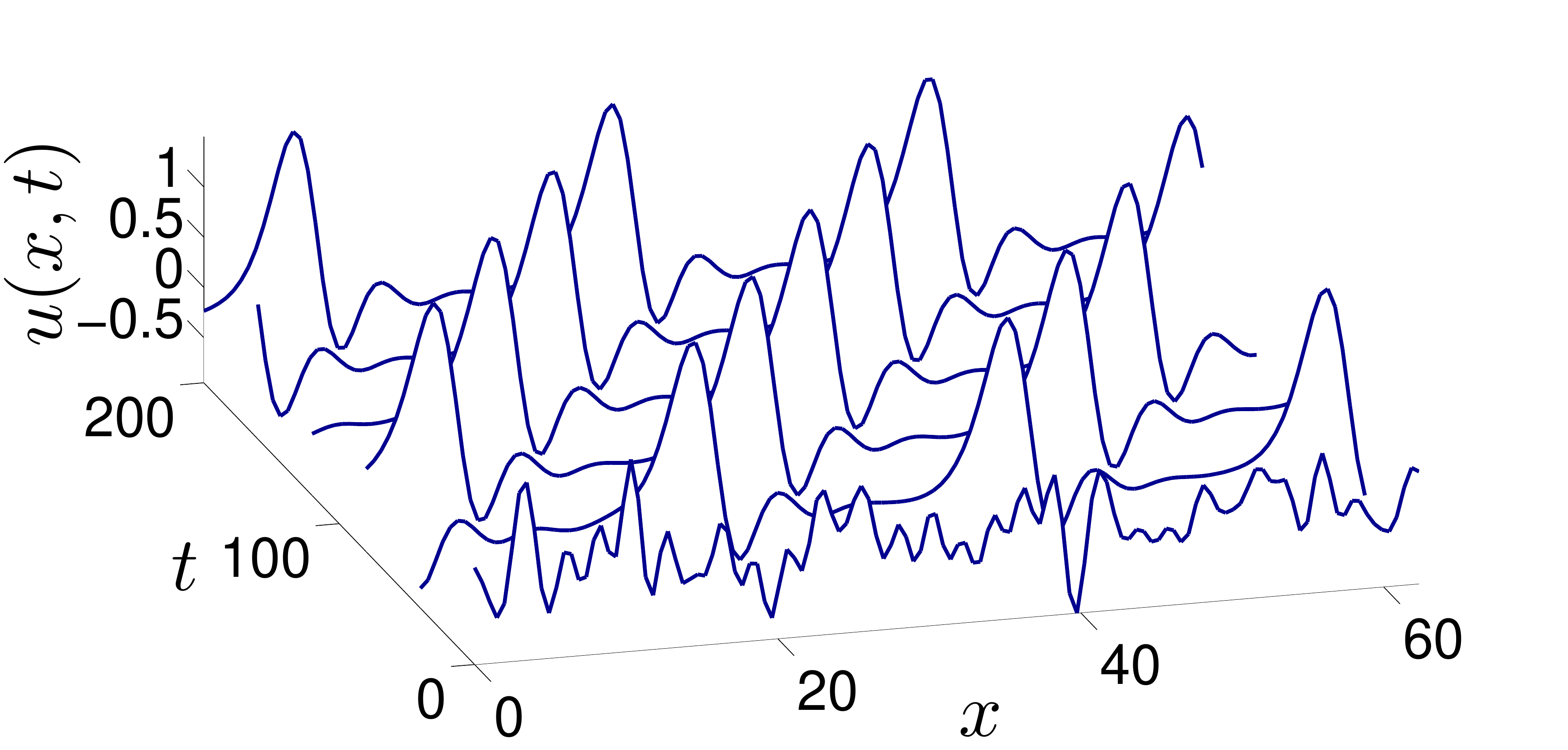}\label{fig:delta0_3pulses}}

	 \caption{Solution to the KS equation for $\nu = 0.01$ \eqref{fig:delta0_unc} with no controls, and controlled to \eqref{fig:delta0_1pulse} one solitary pulse, \eqref{fig:delta0_2pulses} two solitary pulses and \eqref{fig:delta0_3pulses} three solitary pulses.}
	 \label{fig:delta0}
 \end{figure}

The control algorithms presented and analysed here also apply to the gKS equation with $\delta>0$. The findings are similar and for brevity we do not present them here.
Detailed results and animations can be found in \cite{Gomes2015} and its supplemental material.

\section{Optimal control for the generalised Kuramoto-Sivashinsky equation}\label{sec:optimization}

In practical applications it is important to apply controls that minimise the cost associated with their use, and this leads to
considerations of an optimal control problem based on some measure of the energy cost of the controls.
In what follows we
consider the energy of the controls given by their $L^2$-norm. Since the controls decay to almost zero relatively fast in time 
(see Figure \ref{fig:nu035mu03steadystate}, for instance), we expect that minimising their $L^2$-norm should decrease their amplitude.

The objective, then,  is to achieve control of the zero solution or other
unstable steady-state or travelling wave solutions of the gKS, and to do this while spending the least energy possible. 
To that end, we consider a cost functional that includes the distance between the solution and the desired state, as well
as the $L^2$-norm of the controls used. 
Different distance norms $\| u-\bar{u}\|$ can be used and in our computations we employ the 
$L^2$, $H^1$ or $H^2$ norms. 
The reason we consider different norms is that the solutions are expected to belong to $H^2(0,2\pi)$, and we wish to analyse the effects of the regularity and the oscillations of the solutions on the cost functional.

In flow control it is possible to use point actuator functions 
\cite{Antoniades2001,Christofides2000a,Christofides2000,Dubljevic2010}, implying that we can take $b_i(x) = \delta (x-x_i)$. 
Assuming that the cost of placing a control at $x_i$ is the same for all actuator positions $x \in (0,2\pi)$, 
it makes sense to seek a solution that minimises the norms of the control functions $f_i(t)$. 
Since the delta functions in the feedback controls are not $L^2$ functions, the standard results of 
constrained optimisation for PDEs \cite{Lions1971,Troltzsch2010} do not apply. 
Because of this hurdle we will first prove existence of optimal controls for the case of general controls, $f(x,t) \in L^2(0,T;\dot{L}^2(0,2\pi))$, i.e.  mean zero spatially periodic controls in $L^2(0,2\pi)$ that are also $L^2$ functions of time, and focus on 
the case of feedback controls where we can apply standard optimisation techniques.

We consider cost functionals of the form
\begin{equation}\label{costfunc}
\cost\left(u, F\right) = \displaystyle{\frac{1}{2}\int_0^T \| u(\cdot,t)-\bar{u}\|^2 \ dt + \frac{1}{2} \| u(\cdot, T) - \bar{u}\|^2 } \displaystyle{+ \frac{\gamma}{2} \int_0^T \sum_{i=1}^m f_i(t)^2 \ dt}
\end{equation}
where $\bar{u}$ is the desired steady state, $F = [f_1 \ \cdots \ f_m]$, and the norm (e.g. $L^2$, $H^1$ or $H^2$) is left unspecified.
The choice of the parameter $\gamma$ depends on how large we are willing to allow the norm of the controls to become: 
if we need to maintain a small norm of the controls while allowing the solution to be considerably different from the steady state, 
then we use $\gamma> 1$. If, on the other hand, we have a very large amount of energy to spend on the controls and want the solution to be as close as possible to the desired steady state, then we choose $\gamma<< 1$ so that the weight of the controls does not influence 
significantly the value of the cost functional.
The terminal time term $ \frac{1}{2} \| u(\cdot, T) - \bar{u}\|^2 $ is introduced to provide us with a condition for the final-value problem obtained when solving the adjoint equation for the optimisation problem. Before proceeding to the optimisation problem it is useful
to introduce the following definitions.

\begin{defn}\label{def1}
The space of admissible contols, $F_{ad}$, is a bounded convex subset of $L^2(0,T;\dot{L}^2(0,2\pi))$.
\end{defn}

\begin{defn}\label{def2}
A control $f^* \in F_{ad}$ is said to be optimal, and $u^* = u(f^*)$ is the associated optimal state, if
$\cost(u(f^*),\bar{u},f^*) \leq \cost(u(f),\bar{u},f) \ \forall f \in F_{ad}.
$
\end{defn}
Our numerical experiments presented in Section~\ref{sec:num-optcontr} suggest that, 
given an initial condition and a desired steady state, there exists at least one optimal placement of the control actuators for every value of $\nu$ and $\mu$. 
However, here we prove existence of an optimal control in the case of an open-loop control
using the quadratic cost functional
\begin{equation}\label{costfuncquad}
\cost\left(u,f\right) = \displaystyle{\frac{1}{2}\int_0^T \| u(\cdot,t)-\bar{u}\|_{L^2}^2 \ dt + \frac{1}{2} \| u(\cdot, T) - \bar{u}\|_{L^2}^2 }
 \displaystyle{+ \frac{\gamma}{2} \int_0^T\int_0^{2\pi}  f(x,t)^2 \ dx \ dt}.
\end{equation}
The point actuated controls in the form of delta functions are not in $L^2$ and hence an analogous proof in this case
requires distribution theory which is beyond the scope of the present study.
The optimisation problem is:
\begin{subequations}\label{optmrob}
\begin{align}
\label{costfunctional} \mathrm{minimise} \quad&  \cost\left(u,f\right) \\
\label{stateequation} \mathrm{subject \ to} \quad &  \displaystyle{ u_t + \nu u_{xxxx} +  u_{xx} + uu_x = f(x,t),} \\
\label{initialcondition} &u(x,0) = u_0(x)\in \dotH{2}(0,2\pi), \\
\label{boundaryconditions} &\frac{\partial^j u}{\partial x^j}(x+2\pi) = \frac{\partial^j u}{\partial x^j}(x), \quad j = 0,1,2,3,\\
\label{restriction1}& f \in F_{ad}.
\end{align}	
\end{subequations}

The main result of this section is the following.
\begin{thm}\label{theorem}
Assume that $F_{ad}\subset L^2(0,T;\dot{L}^2(0,2\pi))$. Then~\eqref{optmrob} has at least one optimal control $f^*$ with associated optimal state $u^*$.
\end{thm}

\begin{remark}\label{rmk0}
Since the Hilbert transform and third derivative terms are linear functionals of $u$, Theorem~\ref{theorem} can be easily generalised to the case when $\mu,\delta>0$.
\end{remark}
\begin{remark}\label{rmk1}
The presence of the Burgers nonlinearity in the PDE makes the optimisation problem no longer convex. Consequently, we do not expect the solution of the optimal control problem to be unique. 
\end{remark}
The nonlinearity in our problem, defined by $\mathcal{N}(u) = uu_x$, is twice Fr\'{e}chet differentiable with respect to $u$ but 
is neither an increasing functional of $u$, nor is it globally Lipschitz continuous. 
Furthermore, it depends explicitly on the derivative $u_x$. Consequently, the well developed theory of optimal control for systems of reaction-diffusion equations~\cite{Lions1971,Troltzsch2010} does not apply to our problem; see Equation~\eqref{optmrob}. 

\noindent{\it Proof of Theorem~\ref{theorem}.} 
Let $X = H^1(0,T;\dotH{2}(0,2\pi))\times F_{ad}$ and $e(\cdot,\cdot)$ be a functional defined in $X$ by
\begin{equation}\label{e(u,f)} e(u,f) = \left[\begin{array}{c} u_t+\nu u_{xxxx} + u_{xx} + uu_x - f \\
u(\cdot,0)-u_0(x) \end{array}\right].\end{equation}
Our optimisation problem is that of minimising the cost functional $\cost$ subject to $e(u,f) = 0$, periodic boundary conditions and $(u,f)\in X$;; see Equation~\eqref{optmrob}.

Let $(u,f) \in X$ satisfy $e(u,f) = 0$. Since $\cost$ is a function of the sum of the norms of $u$ and $f$, it is clear that $\cost$ is nonnegative and 
\begin{equation}\label{infinity} \cost(u,f) \rightarrow \infty \quad \textrm{ for } \quad \|(u,f)\|_X \rightarrow \infty. \end{equation}
Therefore, there exists a constant $c\geq 0$ such that $  c = \inf_{e(u,f) = 0} \cost(u,f) = \lim_{n\rightarrow\infty} \cost(u^n,f^n)$, 
where $(u^n,f^n)$ is a minimising sequence in $X$, which exists due to the reflexivity of $L^2$. From Equation \eqref{infinity} 
we can conclude that $\left\{(u^n,f^n)\right\}_{n \in \NN}$ is bounded, and therefore there exists $(u^*,f^*)\in X$ such that $ (u^n,f^n)  \rightharpoonup (u^*,f^*) \quad \textrm{ for } \quad n\rightarrow\infty$.
This means that all the linear functionals of $u^n$ and $f^n$, and in particular their derivatives, 
also converge weakly to the same functionals of $u^*$ and $f^*$ in the appropriate space. 
Hence, we only need to prove the convergence of the nonlinearity.

Following an argument similar to that in \cite{Volkwein2000} for the Burgers equation, we notice that since for every $t\in[0,T]$
we have $u^*(\cdot,t) \in \dotH{2}(0,2\pi)$, then $u^* (\cdot,t)\in C([0,2\pi])$ and therefore if $\varphi \in X$, $(u^*\varphi)(\cdot,t) \in L^2([0,2\pi])$. Hence, $u^*\varphi \in H^1(0,T;L^2(0,2\pi))$ and 
\begin{equation}\label{firstnonl} \int_0^T\int_0^{2\pi} (u^n-u^*)_x u^*\varphi \ dx \ dt  \longrightarrow_{n\rightarrow\infty} 0, \quad \forall\varphi\in H^1(0,T;\dotH{2}(0,2\pi)).\end{equation}

Finally, from the estimates \eqref{L2estimates} we know that $\|u^n_x\|$ is bounded, and since $H^2(\Omega)$ is compactly embedded in $L^2(\Omega)$, we deduce that
\begin{multline}\label{estimate}
\int_0^T\int_0^{2\pi} (u^n-u^*)u_x^n\varphi\ dx \ dt \leq \|u^n-u^*\|\|u_x^n\|\|\varphi\|_{L^\infty(0,2\pi)} \longrightarrow_{n\rightarrow\infty} 0, \\
  \forall \varphi \in \dotH{2}(0,2\pi).
\end{multline}

Hence, by adding and subtracting appropriate terms we have
\begin{multline}\label{nonlinearity}
\int_0^T\int_0^{2\pi} (u^n u_x^n - u^*u_x^*)\varphi \ dx \ dt = \int_0^{2\pi} (u^n-u^*)u_x^n\varphi + (u^n-u^*)_xu^*\varphi \ dx \longrightarrow_{n\rightarrow\infty} 0,\\
 \forall \varphi \in \dotH{2}(0,2\pi),
\end{multline}
and therefore the nonlinearity $u^nu^n_x$ is weakly convergent to $u^*u^*_x$ in $X$.
Now, noticing that $u^*$ and $u^*_x$ are continuous in $[0,2\pi]\times[0,T]$, 
we observe that $u^*$ satisfies the periodic boundary conditions and the initial condition. If we now consider $\varphi \in X$ satisfying 
$\varphi(x,T)= 0$, and use the weak convergence of the derivatives of $u$ and equation \eqref{nonlinearity}, we conclude that $(u^*,f^*)$ is a weak solution of the state equation.
The optimality of the pair $(u^*,f^*)$ follows from the weak lower semi-continuity of $\cost$ (cf. proof of Theorem 4.15 in \cite{Troltzsch2010}).
\qed

%
%
\subsection{ Algorithm and Numerical Experiments}
\label{sec:num-optcontr}

We note that the dependence of the cost functional on the positions $x_i$, $i = 1,\dots,m$ is in the controls, since the matrix $K$ necessary to define them depends on the positions chosen. However, when defining the Lagrangian we will assume that only 
the functions $b_i(x)$ depend on $x_i$, and treat the controls $f_i(t)$ as if they were independent of the positions $x_i$. 
Under this assumption we are able to obtain very satisfactory results, as evidenced by the results presented in the tables below.
We begin by introducing the Lagrangian
\begin{equation}\label{lagrangean}
\begin{array}{ll}
\Lag\left(u,p,[x_1,x_2,\dots,x_m]^T\right) =& \displaystyle{\frac{1}{2}\int_0^T \| u(\cdot,t)-\bar{u}\|^2 \ dt + \frac{1}{2} \| u(\cdot, T) - \bar{u}\|^2}\\
&\displaystyle{ + \frac{\gamma}{2} \int_0^T \sum_{i = 1}^m f_i(t)^2 \ dt }\\
&\displaystyle{- \int_0^T\int_0^{2\pi} \left(u_t + \nu u_{xxxx} + u_{xx} + uu_x\right)p(x,t) \ dx \ dt}\\
&\displaystyle{+ \int_0^T\int_0^{2\pi}\sum_{i=1}^m \delta(x-x_i)_if_i(t)p(x,t) \ dx \ dt.}
\end{array}
\end{equation}
Integrating by parts in space and time and computing the Fr\'{e}chet derivative with respect to $u$ (with test functions $h(x,t)$ satisfying $h(x,0) = 0$), we obtain the adjoint equation
\begin{equation}\label{adjoint}
\left\{\begin{array}{rcl}
-p_t + \nu p_{xxxx} + p_{xx} - up_x &=&  \sum_{i = 1}^m \delta(x-x_i)K_{i\cdot}z^p + u - \bar{u},\\
p(x,T) &=& u(x,T),\\
\frac{\partial^j p}{\partial x^j} (x+2\pi) & = & \frac{\partial^j p}{\partial x^j} (x),
\end{array}\right.
\end{equation}
where $x \in [0,2\pi]$ and $t \in [0,T]$. This PDE is backwards in time but is well-posed since it is a final value problem. 
To solve it, we obtain the discretised ODE system $-\dot{z}^p = \mathcal{A}z^p + G^{adj}(z^p, z^u) +z^u-z^{\bar{u}}$, where the elements of $G^{adj}(z^p, z^u)$ are given by
\begin{multline*}
g^s_{n,adj} = \frac{1}{2\sqrt{\pi}} \sum_{j + k = n} k(u^s_j p^s_k - u^c_j p^c_k) + \frac{1}{2\sqrt{\pi}}\sum_{j-k = n}\left(k(u^s_j p^s_k +  u^c_j p^c_k) - j(u^s_k p^s_j + u^c_k p^c_j)\right), \\
g^c_{n,adj} = \frac{1}{2\sqrt{\pi}} \sum_{j + k = n} k(u^c_j p^s_k + u^s_j p^c_k) + \frac{1}{2\sqrt{\pi}}\sum_{j-k = n}\left(k(u^c_j p^s_k -  u^s_j p^c_k) + j(u^c_k p^s_j - u^s_k p_j^c)\right),\end{multline*}
and we have used the Fourier series representation 
$p(x,t) = \frac{p_0^c}{\sqrt{2\pi}} + \sum_{n=1}^{\infty} p^s_n(t) \frac{\sin(nx)}{\sqrt{\pi}} + \sum_{n=1}^{\infty} p^c_n(t) \frac{\cos(nx)}{\sqrt{\pi}}$.

Differentiating with respect to the positions of the control actuators, we also obtain a descent direction using the variational inequality, or first variation
\begin{equation}\label{variationalinb}
\int_0^T \left[f_1(t)p_x(\bar{x}_1,t) \cdots f_m(t)p_x(\bar{x}_m,t)\right]^T \cdot (x - \bar{x})  \ dt \geq 0, \quad \forall x = [x_1 \cdots x_m]^T,
\end{equation}
where $\bar{x} = [\bar{x}_1,\dots,\bar{x}_m]$ are the optimal positions.
To proceed with the optimisation, we will use a gradient descent method, see~\cite[Sec. 5.9]{Troltzsch2010}, and
consider $F_{ad} = (0,2\pi)^m$. The algorithm is as follows.

\begin{figure}\label{algorithm}
\begin{center}
 \fbox{
\begin{minipage}[htbf]{13cm}
\begin{center}{\bf Algorithm for optimal control for the KS equation}\end{center}
\begin{description}
\item 	Given $\nu, \, \mu, \,  \gamma, \, T, \, u_0(x), \, A, \, x^0, \,  B^0, \, \bar{u}$, compute the matrix $K^0$.
\end{description}
{\bf while} $\cost(\textrm{current iteration}) < \cost(\textrm{previous iteration})$ {\bf do} 

\begin{enumerate}
\item Solve the state equation to obtain $u^{k-1}$ and compute $\cost(u^{k-1},\bar{u},F(x^{k-1}))$;
		\item Solve the adjoint equation to obtain $p^{k-1}$;
		\item Define $P_x = \left[p^{k-1}_x(x^{k-1}_1,t) \cdots p^{k-1}_x(x^{k-1}_m,t)\right]$, \\ $P_{k-1} = \displaystyle{\int_0^T K^{k-1}(z^{u^{k-1}}-z^{\bar{u}})P_x \ dt}$ and $h_k = - P_{k-1}$;
		\item Find $s = \min_{s>0} \left\{C(u(x^{k-1} + sh_k),\bar{u},F(x^{k-1} + sh_k))\right\}$;
		\item Project $x^{k-1} + sh_k$ into $(0,2\pi)^m$, obtaining $x^k$;
		\item Compute the matrix $B^k$;
		\item Compute the matrix $K^k$ with \emph{Matlab}'s command \emph{place}.
\end{enumerate}
{\bf end}
\end{minipage}}
\end{center}
\caption{Algorithm for optimal control of the KS equation.}
\end{figure}
 
Note that as mentioned earlier we consider the following three different cost functionals:
\begin{equation}
\label{L2cost} \cost_1\left(u,\bar{u}, f\right) = \frac{1}{2}\int_0^T \| u(\cdot,t)-\bar{u}\|^2 \ dt + \frac{1}{2} \| u(\cdot, T) - \bar{u}\|^2+ \frac{\gamma}{2}  \sum_{i=1}^m \|f_i(t)\|_{L^2(0,T)}
\end{equation}
\begin{equation}\label{H1cost}
\begin{array}{rl}
\cost_2\left(u,\bar{u}, f\right) =& \frac{1}{2}\int_0^T \left(\| u(\cdot,t)-\bar{u}\|^2 + \| u_x(\cdot,t)-\bar{u}_x\|^2\right) \ dt\\
&\displaystyle{ + \frac{1}{2}\left( \| u(\cdot, T) - \bar{u}\|^2 + \| u_x(\cdot, T) - \bar{u}_x\|^2\right) }\\
&\displaystyle{ + \frac{\gamma}{2} \sum_{i=1}^m \|f_i(t)\|_{L^2(0,T)} }
\end{array}
\end{equation}
\begin{equation}\label{H2cost}
\begin{array}{rl}
\cost_3\left(u,\bar{u}, f\right) = &\frac{1}{2}\int_0^T \left(\| u(\cdot,t)-\bar{u}\|^2 +  \| u_x(\cdot,t)-\bar{u}_x\|^2 + \| u_{xx}(\cdot,t)-\bar{u}_{xx}\|^2\right) \ dt \\
& \displaystyle{+ \frac{1}{2} \left(\| u(\cdot, T) - \bar{u}\|^2 + \| u_x(\cdot, T) - \bar{u}_x\|^2 + \| u_{xx}(\cdot, T) - \bar{u}_{xx}\|^2\right) }\\
& \displaystyle{+\frac{\gamma}{2} \sum_{i=1}^m \|f_i(t)\|_{L^2(0,T)}}.
\end{array}	\end{equation}

%
%
\subsection{Numerical Experiments}
\label{sec:num-optcontr}

Computations were carried out using the algorithmpresented in Figure~\ref{algorithm} for various values of $\nu$ and $\mu$.
The number of controls used was equal to the number of unstable eigenvalues, and $m$ equidistant points
were used as an initial guess for the position of the controls. 
In all the computations the
initial condition is $u_0(x) = \frac{\sin(x)}{\sqrt{\pi}} + \frac{\cos(x)}{\sqrt{\pi}}$, 
and the final time is $T = 10$. 
Numerical results are presented in Tables \ref{tab1.1}-\ref{tab1.3} for the stabilisation of the zero solution
of the KS equation, and in Tables \ref{tab2.1}-\ref{tab2.3} for the stabilisation on non-trivial unstable steady
states as computed in the bifurcation diagram of Figure~\ref{fig:bifdiag3}.
Each entry in Tables \ref{tab1.1}-\ref{tab1.3} contains the value of $\nu$, the value of the cost functional ($\cost_1$, $\cost_2$
and $\cost_3$ for tables \ref{tab1.1}, \ref{tab1.2} and \ref{tab1.3}, respectively), the cost of the controls
$\sum_{i=1}^m \|f_i(t)\|_{L^2(0,T)}$, the number of iterations required to obtain an optimal state, and in the last
column the spatial distribution of the controls over the domain $[0,2\pi]$ - a heavy dot is placed where a control acts.
Tables \ref{tab2.1}-\ref{tab2.3} are presented in an analogous manner, with the difference that the first column provides information
on the unstable solution that is being controlled, and in particular the branch on Figure~\ref{fig:bifdiag3}
where the solution was taken from is stated along with the value of $\nu$. As the results indicate, several
distinct unstable solutions at a given value of $\nu$ are controlled (e.g. for $\nu=0.1$, three solutions are stabilised
coming from branches 1, 3 and 4 respectively).

\begin{table}[h!]
\begin{tabular}{| c | c | c | c | c | }
\hline
 $\nu$ & Cost $\cost_1$ & Cost of Controls & Iterations & Optimal Positions  \\
\hline

$0.9$  & $8.9647$	& $0.5592$ & $3$	& \includegraphics[width=0.38\linewidth]{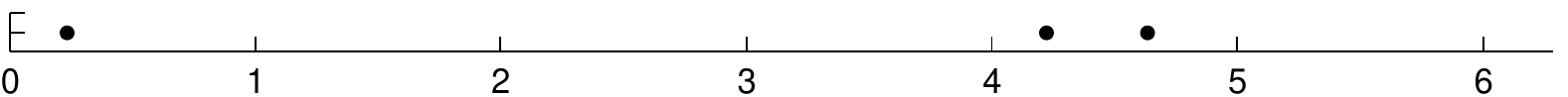} \\

$0.8$ & $6.5012$	& $1.1274$ & $6$	&  \includegraphics[width=0.38\linewidth]{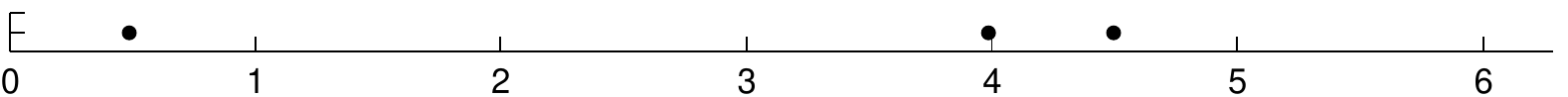}  \\

$0.7$  & $5.5760$ & $1.7204$ & $4$	& \includegraphics[width=0.38\linewidth]{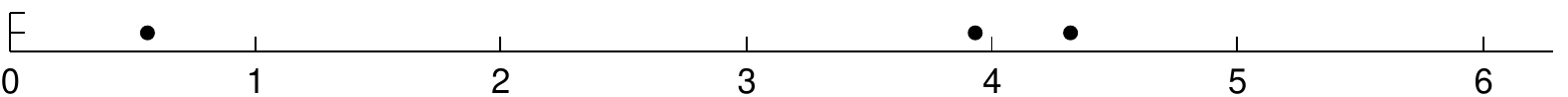} \\

$0.6$  & $5.2803$	& $2.3018$ & $3$	&  \includegraphics[width=0.38\linewidth]{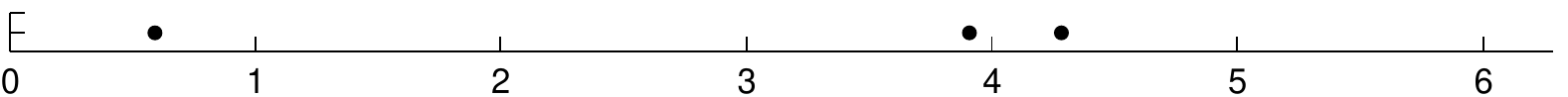}  \\

$0.5$ & $5.2230$	& $2.8204$ & $5$	&  \includegraphics[width=0.38\linewidth ]{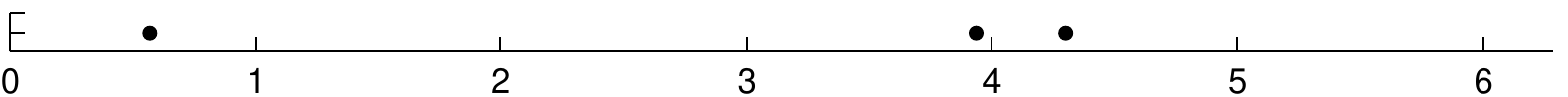} \\

$0.4$ & $5.9339$	& $3.8404$ & $5$	&  \includegraphics[width=0.38\linewidth]{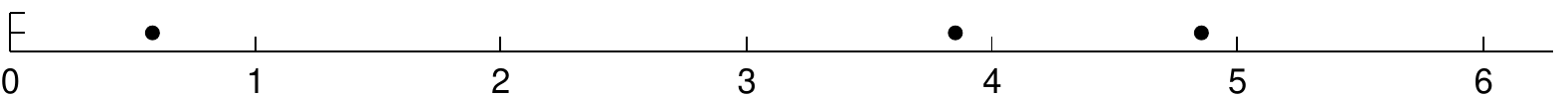}  \\

$0.3$ & $6.2152$	& $3.9813$ & $2$	&  \includegraphics[width=0.38\linewidth]{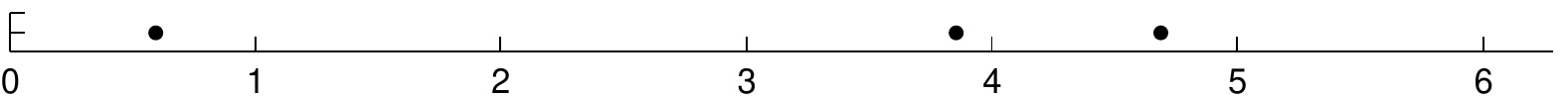} \\

$0.2$ & $6.3127$	& $4.5261$ & $2$	&  \includegraphics[width=0.38\linewidth]{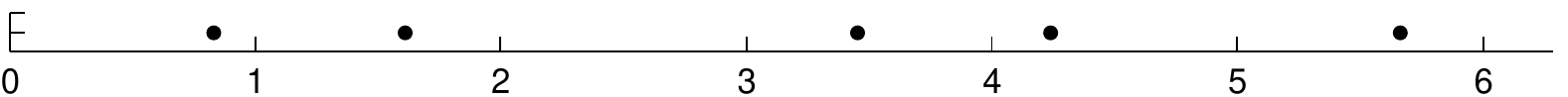} \\

$0.1$ & $7.1652$	& $5.5759$ & $2$	& \includegraphics[width=0.38\linewidth]{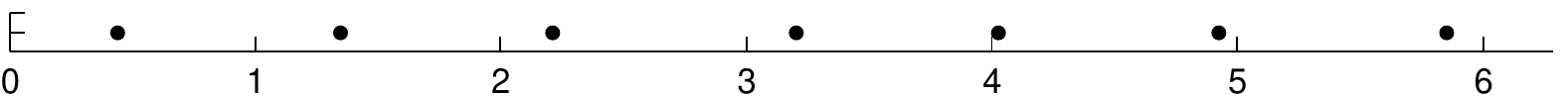} \\

\hline
\end{tabular}
\caption{Optimal positions and value of the cost functional considered in the $L^2$-norm for different values of $\nu$ when stabilising the zero solution to the KS equation.}\label{tab1.1}
\end{table}

\begin{table}[h!]
\begin{tabular}{| c | c | c | c | c | }
\hline
 $\nu$ & Cost $\cost_2$ & Cost of Controls & Iterations & Optimal Positions  \\
\hline

$0.9$ & $17.6354$ & $0.9698$ &  3 &\includegraphics[width=0.38\linewidth]{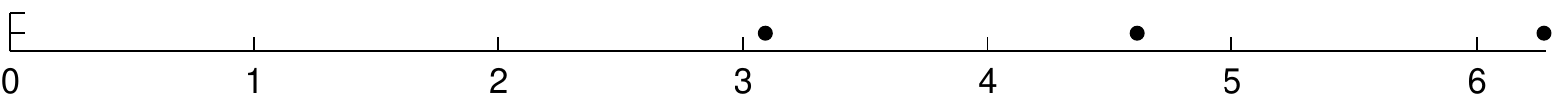}  \\

$0.8$ & $12.2553$ & $1.5486$ &  6 & \includegraphics[width=0.38\linewidth]{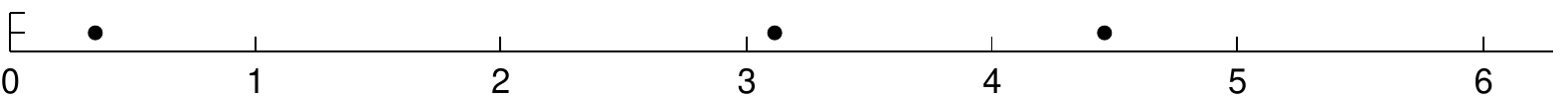}  \\

$0.7$ & $10.8270$ & $2.0016$ &  2 & \includegraphics[width=0.38\linewidth]{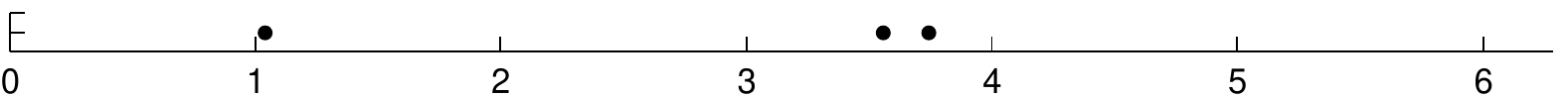}  \\

$0.6$ & $10.2340$ & $4.0181$ &  6 & \includegraphics[width=0.38\linewidth]{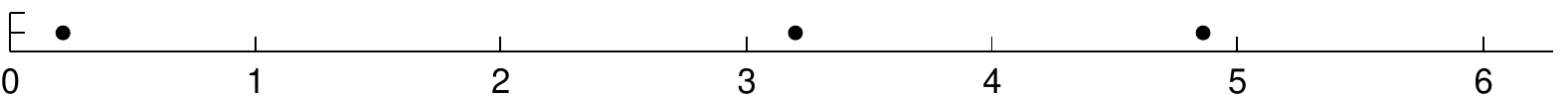}  \\

$0.5$ & $10.1517$ & $4.5407$ &  4& \includegraphics[width=0.38\linewidth]{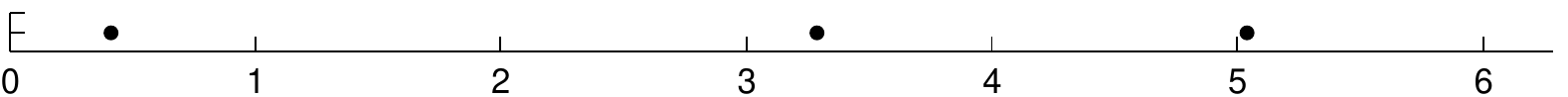}  \\

$0.4$& $9.3345$ & $4.2298$ &  2 & \includegraphics[width=0.38\linewidth]{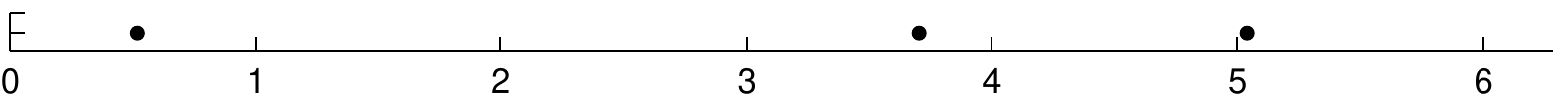}  \\

$0.3$ & $10.9671$ & $4.8110$ &  2 & \includegraphics[width=0.38\linewidth]{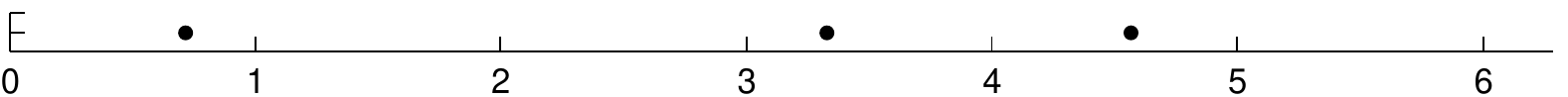}  \\

$0.2$ & $10.7154$ & $5.5308$ &  3 & \includegraphics[width=0.38\linewidth]{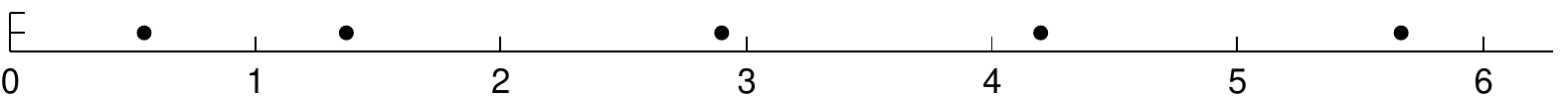}  \\

$0.1$ & $9.7088$ & $5.6463$ &  1 & \includegraphics[width=0.38\linewidth]{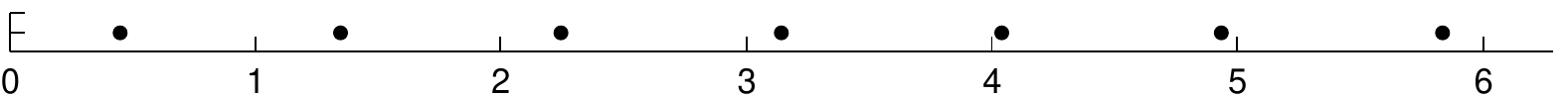}  \\

\hline
\end{tabular}
\caption{Optimal positions and value of the cost functional considered in the $H^1$-norm for different values of $\nu$ when stabilising the zero solution to the KS equation.}\label{tab1.2}
\end{table}

\begin{table}[h!]
\begin{tabular}{| c | c | c | c | c | }
\hline
 $\nu$ & Cost $\cost_3$  & Cost of Controls & Iterations & Optimal Positions  \\
\hline

$0.9$  & $27.2098$ & $1.1313$ & 3 &  \includegraphics[width=0.38\linewidth]{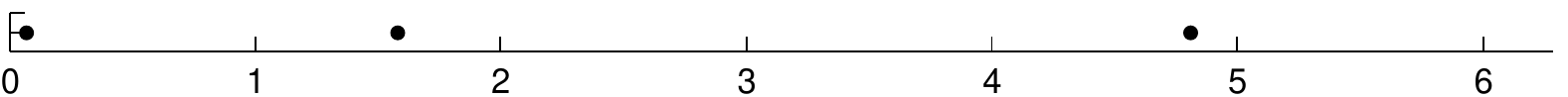}  \\

$0.8$ & $19.3431$ & $2.3490$ & 5 &  \includegraphics[width=0.38\linewidth]{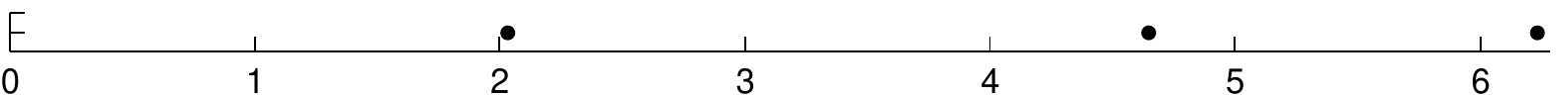}   \\

$0.7$ & $15.8522$ & $2.5815$ & 4 &  \includegraphics[width=0.38\linewidth]{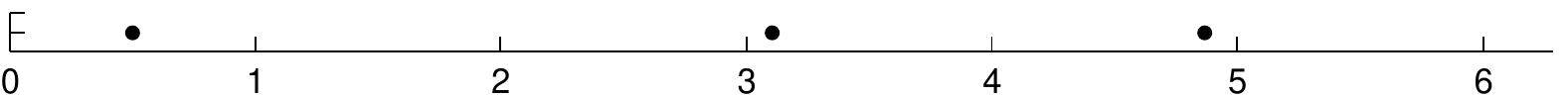}   \\

$0.6$ &  $14.0865$ & $3.3384$ & 5 &  \includegraphics[width=0.38\linewidth]{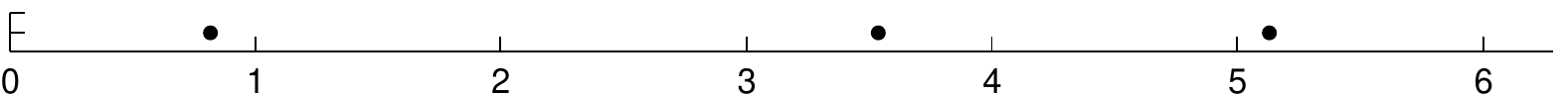}   \\

$0.5$ &  $17.0462$ & $6.4166$ & 1 &  \includegraphics[width=0.38\linewidth]{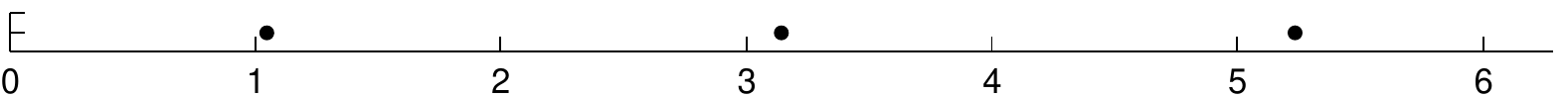}   \\

$0.4$ & $20.7720$ & $8.3217$ & 1 &  \includegraphics[width=0.38\linewidth]{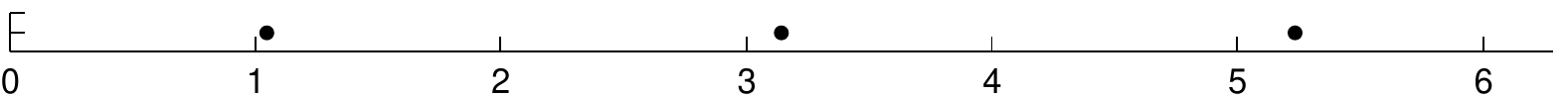}  \\

$0.3$ & $14.4393$ & $5.3865$ & 3 &  \includegraphics[width=0.38\linewidth]{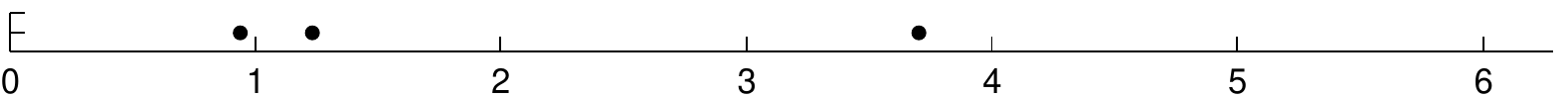}   \\

$0.2$ & $21.5856$ & $6.1456$ & 1 &  \includegraphics[width=0.38\linewidth]{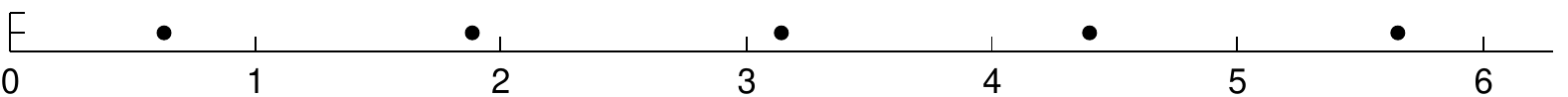}   \\

$0.1$ & $21.1636$ & $5.6463$ & 1 &  \includegraphics[width=0.38\linewidth]{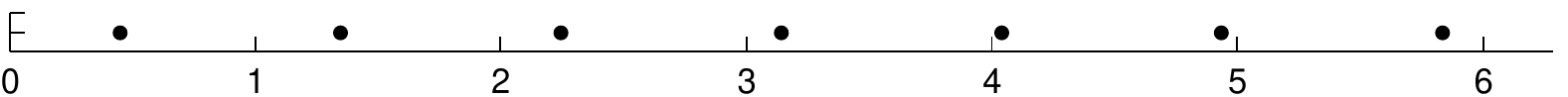}  \\
\hline
\end{tabular}
\caption{Optimal positions and value of the cost functional considered in the $H^2$-norm  for different values of $\nu$ when stabilising the zero solution to the KS equation.}\label{tab1.3}
\end{table}

As expected we observe that the value of the cost functionals $\cost_1,\,\cost_2,\,\cost_3$ given by ~\eqref{L2cost}-\eqref{H2cost}, increases as $\nu$ decreases. Furthermore, the value of the cost functional also increases as we increase the desired regularity of the solution from $L^2$ to $H^1$ to $H^2-$norms. Comparing the results and in particular
the positions of the optimal controls for the three different cost functionals in Tables \ref{tab1.1}-\ref{tab1.3}, we can
conclude that for the stabilisation of the zero steady states, 
the optimal control problem is more robust (in the sense that the optimal positions of the controls 
do not change much as $\nu$ is reduced) when the $L^2$ cost functional $\cost_1$  is used.

\begin{table}[h!]
\begin{tabular}{| c | c | c | c | c | }
\hline
 $\nu$ & Cost $\cost_1$ & Cost of Controls& Iterations & Optimal Positions  \\
\hline

$0.3$& $14.3164$	& $9.7419$ & $5$	& \includegraphics[width=0.38\linewidth]{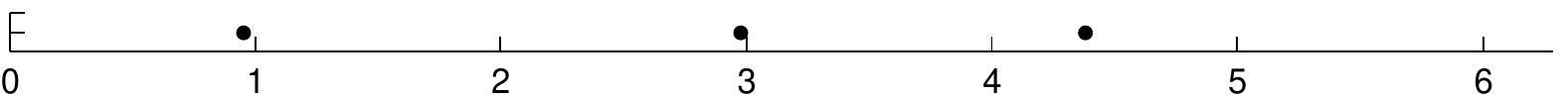} \\

$0.2$,br.1 & $30.0588$ & $21.8691$ & $1$	&  \includegraphics[width=0.38\linewidth]{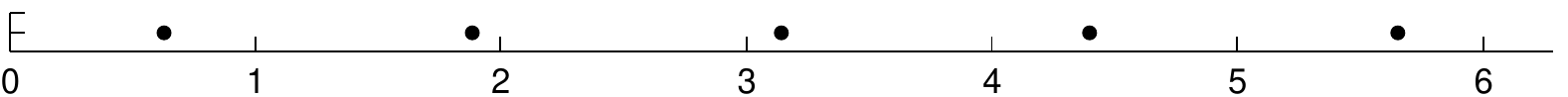} \\

$0.2$,br.3 & $24.0520$	& $16.2832$ & $2$	& \includegraphics[width=0.38\linewidth]{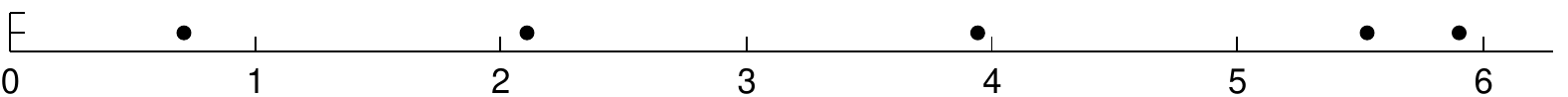} \\

$0.1$,br.1 & $28.5591$	& $32.9859$ & $2$	&  \includegraphics[width=0.38\linewidth]{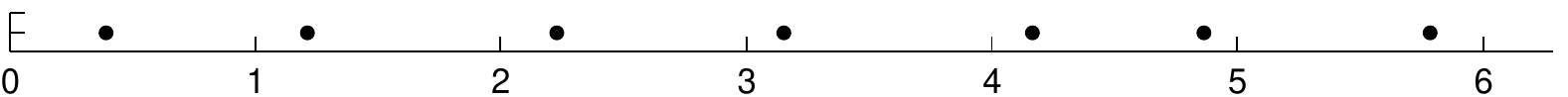}\\

$0.1$,br.3 & $37.8902$	& $32.4264$ & $1$	&  \includegraphics[width=0.38\linewidth]{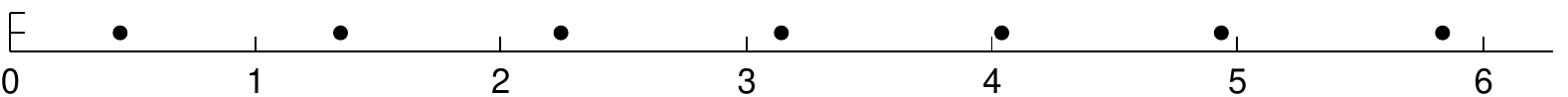}\\

$0.1$,br.4 & $62.3916$	& $51.6820$ & $4$	&  \includegraphics[width=0.38\linewidth]{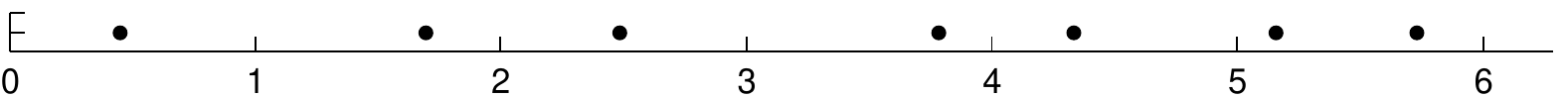} \\

\hline
\end{tabular}
\caption{Optimal positions and value of the cost functional considered in the $L^2$-norm for different values of $\nu$ when stabilising some of the nontrivial steady states from the bifurcation diagram~\ref{fig:bifdiag3}.}\label{tab2.1}
\end{table}

\begin{table}[h!]
\begin{tabular}{| c | c | c | c | c | }
\hline
 $\nu$ & Cost $\cost_2$ & Cost of Controls& Iterations & Optimal Positions  \\
\hline

$0.3$ & $24.6922$ & $9.6582$ &  3 & \includegraphics[width=0.38\linewidth]{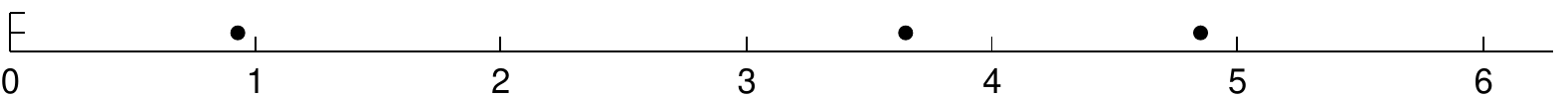}\\

$0.2$,br.1& $53.7672$ & $17.9929$ &  2 &  \includegraphics[width=0.38\linewidth]{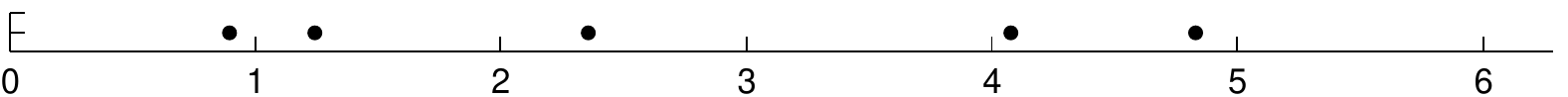} \\

$0.2$,br.3 & $52.2630$ & $15.4089$ &  3 &  \includegraphics[width=0.38\linewidth]{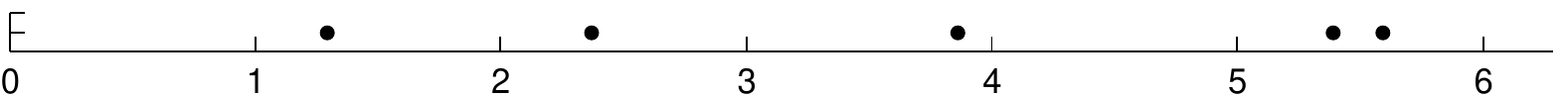} \\

$0.1$,br1 & $87.1636$ & $35.2169$ & 1&  \includegraphics[width=0.38\linewidth]{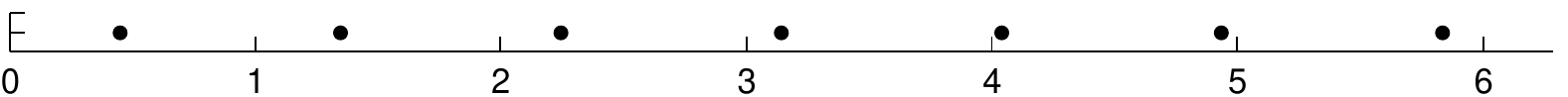}\\

$0.1$,br.3 & $83.9787$ & $33.7581$ &  2 &  \includegraphics[width=0.38\linewidth]{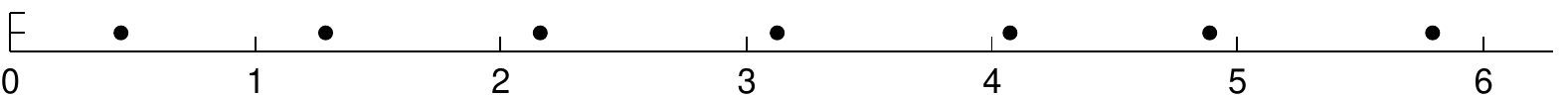}\\

$0.1$,br.4 & $171.6040$ & $62.1381$ &  2 &  \includegraphics[width=0.38\linewidth]{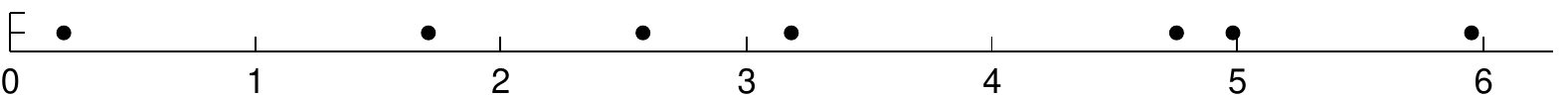} \\

\hline

\end{tabular}
\caption{Optimal positions and value of the cost functional considered in the $H^1$-norm for different values of $\nu$ when stabilising some of the nontrivial steady states from the bifurcation diagram~\ref{fig:bifdiag3}.}\label{tab2.2}
\end{table}

\begin{table}[h!]
\begin{tabular}{| c | c | c | c | c | }
\hline
 $\nu$ & Cost $\cost_3$ & Cost of Controls& Iterations & Optimal Positions  \\
\hline

$0.3$ & $54.5441$ & $13.8436$ & 2 &  \includegraphics[width=0.38\linewidth]{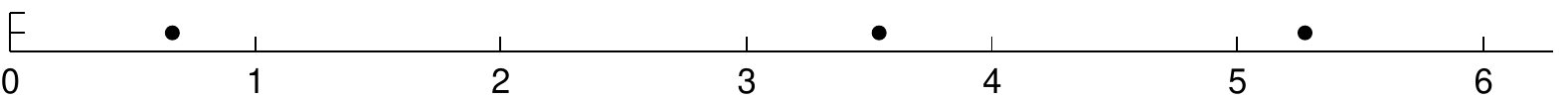} \\

$0.2$,br.1 & $262.7363$ & $21.1679$ & 3 & \includegraphics[width=0.38\linewidth]{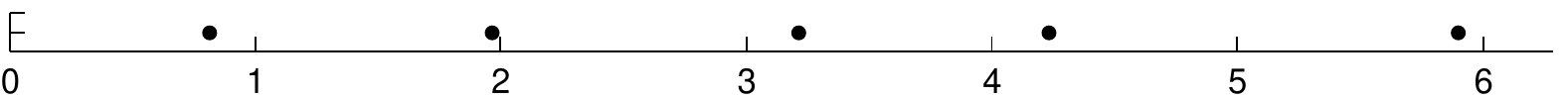}\\

$0.2$,br.3 & $266.0515$ & $32.4603$ & 3 &  \includegraphics[width=0.38\linewidth]{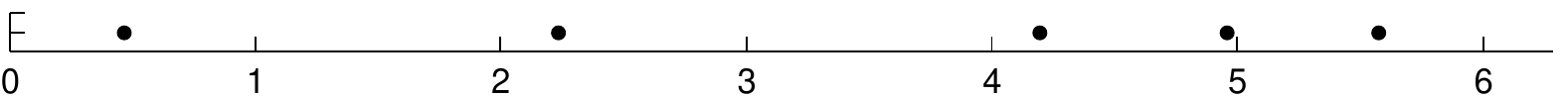} \\

$0.1$,br.1 & $702.4697$ & $35.2169$ & 1 &  \includegraphics[width=0.38\linewidth]{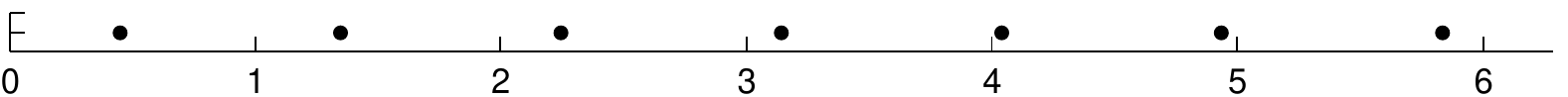} \\

$0.1$,br.3 & $745.6007$ & $32.4264$ & 1 &  \includegraphics[width=0.38\linewidth]{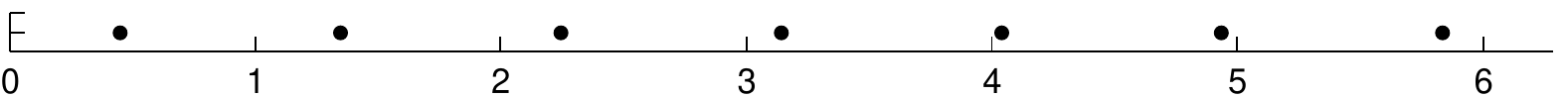} \\
 
$0.1$,br.4 & $1384.6689$ & $63.1272$ & 2 &  \includegraphics[width=0.38\linewidth]{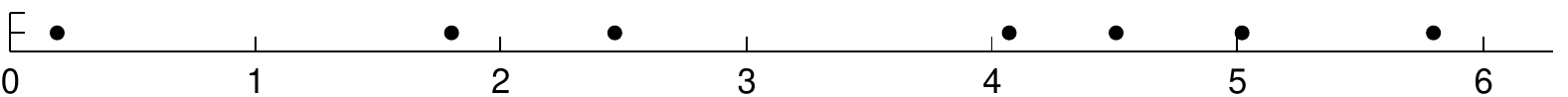} \\

\hline

\end{tabular}
\caption{Optimal positions and value of the cost functional considered inthe $H^2$-norm for different values of $\nu$ when stabilising some of the nontrivial steady states from the bifurcation diagram~\ref{fig:bifdiag3}.}\label{tab2.3}
\end{table}

Turning now to the results of Tables \ref{tab2.1}-\ref{tab2.3} that deal with the stabilisation of unstable nonuniform steady states,
we observe once again that there is an increase in the cost functionals as $\nu$ decreases. We also observe that in this case
(and in contrast to the stabilisation of the zero solution) the higher order norms give optimal controls that are more robust, with respect to changing $\nu$, in comparison to utilising the $L^2$ cost functional.

Similar numerical experiments were performed for the optimal control problem for the KS equation in the presence of an electric field, $\mu>0$. 
The results are quite similar to the ones already presented in this section and we omit providing additional graphs and tables. However, we
summarise the conclusions drawn from the nonzero electric field numerical experiments as follows:
with few exceptions, an increase in the intensity of the electric field parameter $\mu$ increases the cost of the controls. 
In addition, it is found that
the optimal controls for stabilising zero steady states are more robust, with respect to changes in $\mu$, when using the $L^2$ cost functional.
Similarly, when stabilising nontrivial steady states, more robust optimal positions for the controls arise when
the $H^1$ and $H^2$ cost functionals are used. Both of these findings are analogous to those
for the non-electrified control problem.

\begin{figure}[h!]
	\centering 
	\subfloat[Controlled solution]{	\includegraphics[width = 0.75\linewidth]{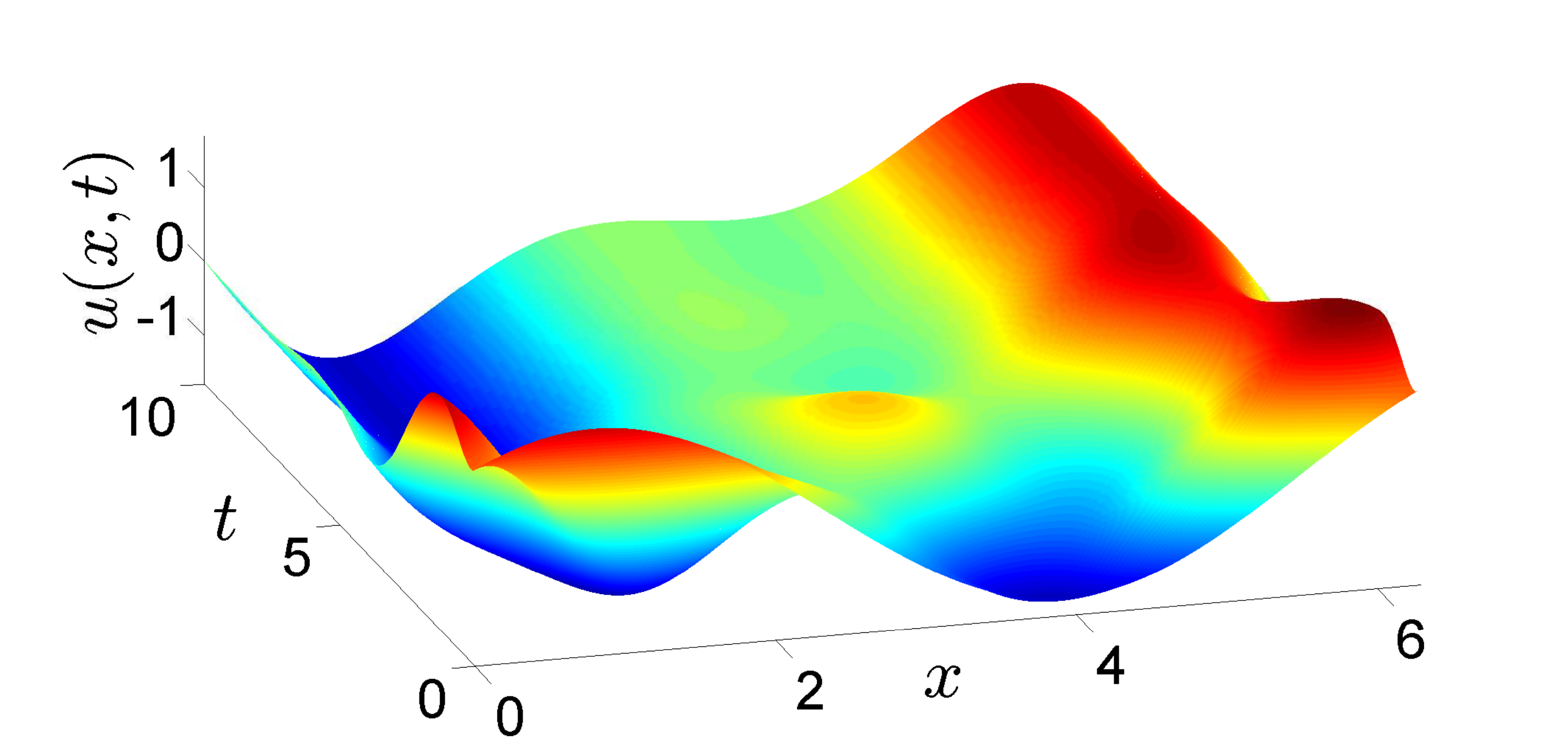}\label{fig:nu03contr}}

	\subfloat[Equidistant controls]{	\includegraphics[width = 0.5\linewidth]{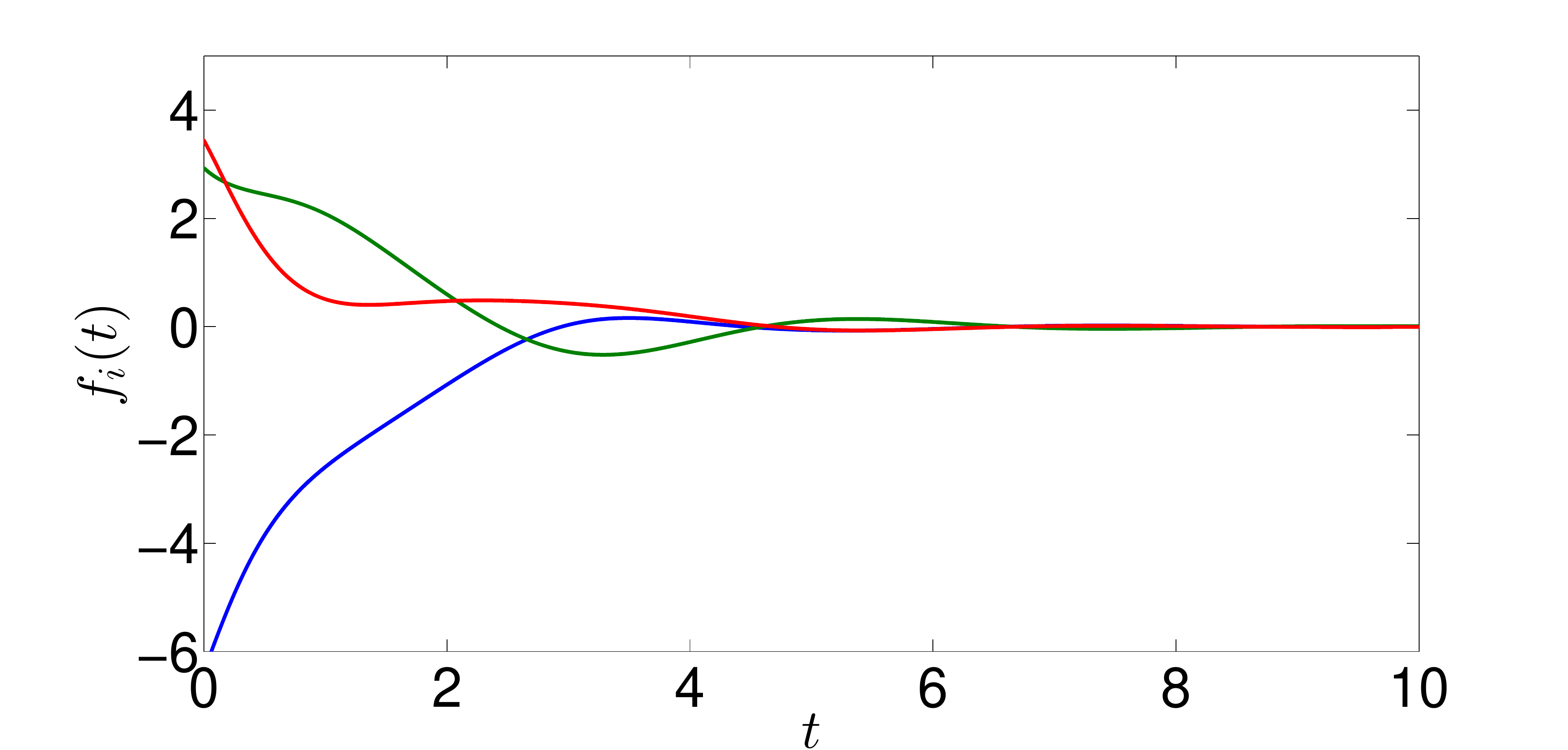}\label{fig:nu03ctrls}}
	\subfloat[Optimal controls]{	\includegraphics[width = 0.5\linewidth]{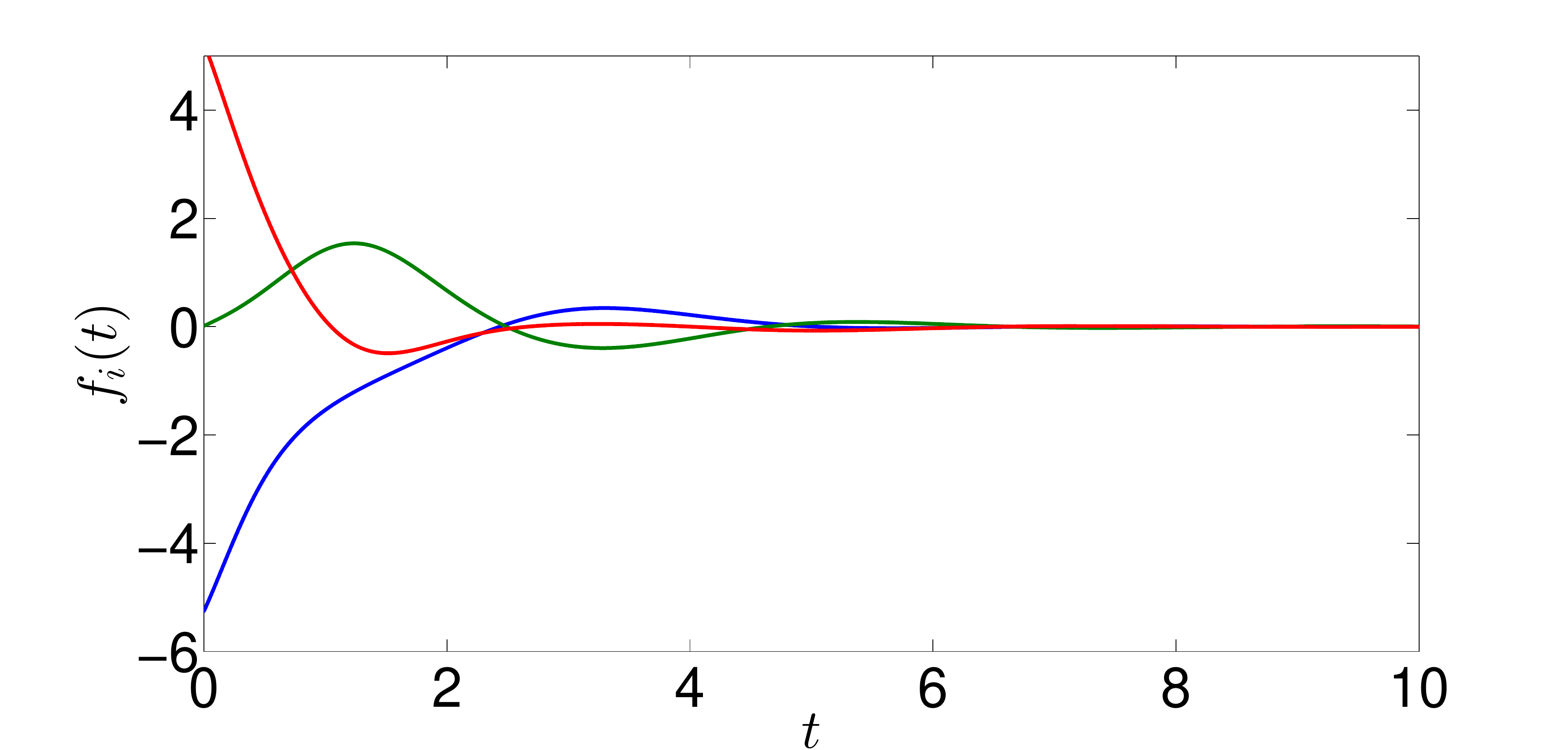}\label{fig:nu03optctrls}}
	 \caption{Controlled steady state of the KS equation for $\nu = 0.3$ (\ref{fig:nu03contr}) and controls applied: (\ref{fig:nu03ctrls}) equidistant and (\ref{fig:nu03optctrls}) optimal.}
	 \label{fig:nu03steadystate}
\end{figure}

\begin{figure}[h!]
	\centering 
	\subfloat[Controlled solution]{	\includegraphics[width = 0.75\linewidth]{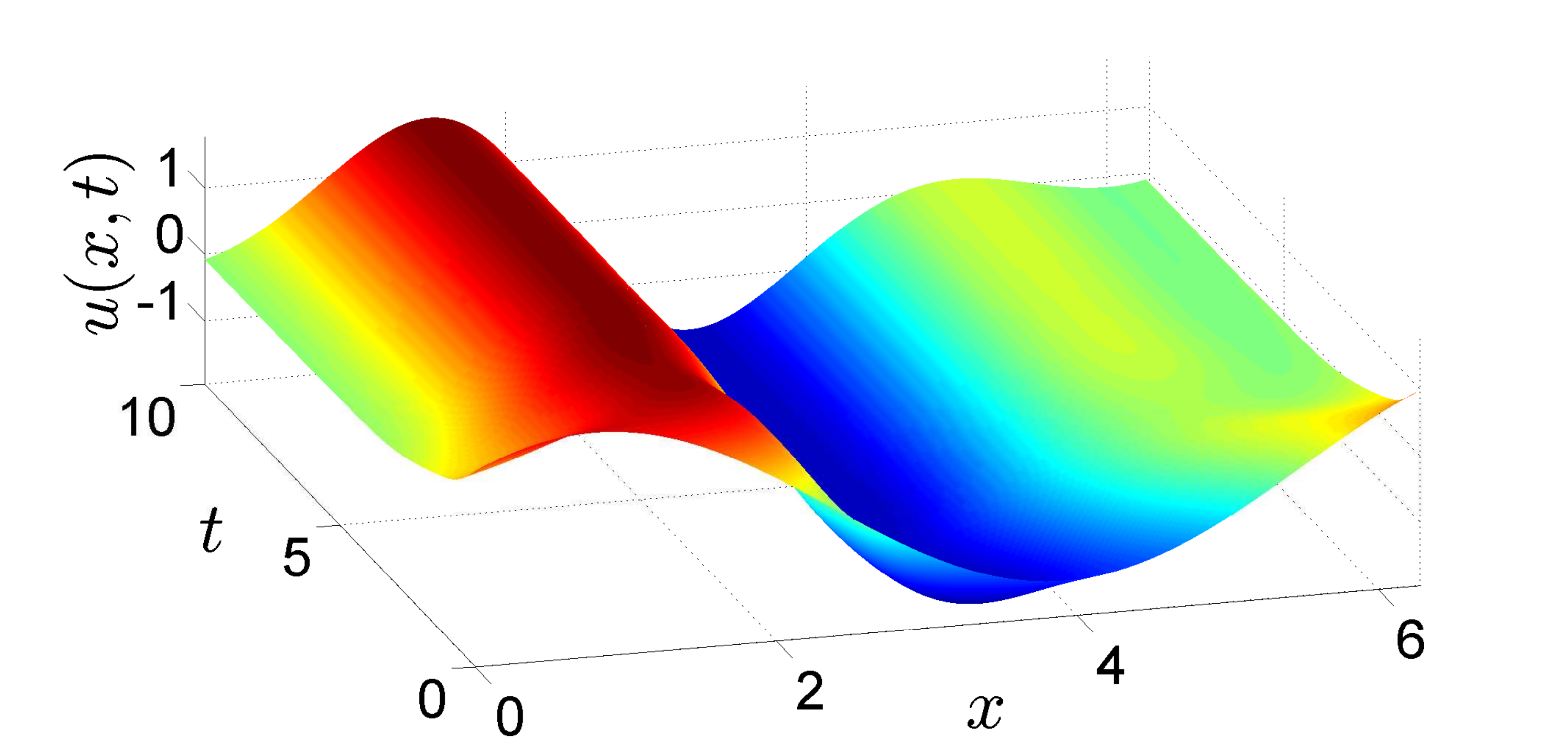}\label{fig:nu05mu04ss}}

	\subfloat[Equidistant controls]{	\includegraphics[width = 0.5\linewidth]{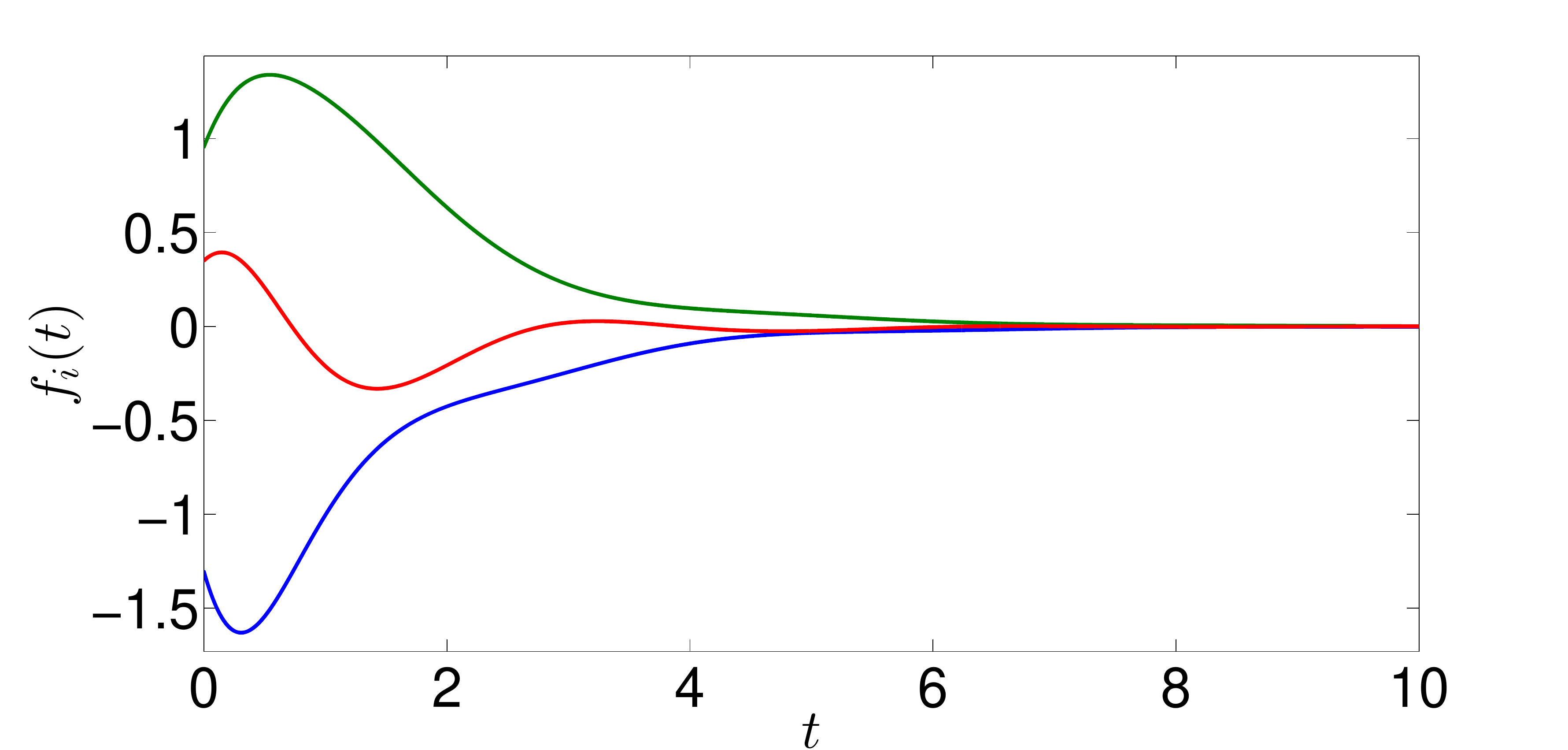}\label{fig:nu05mu04ctrls}}
	\subfloat[Optimal controls]{	\includegraphics[width = 0.5\linewidth]{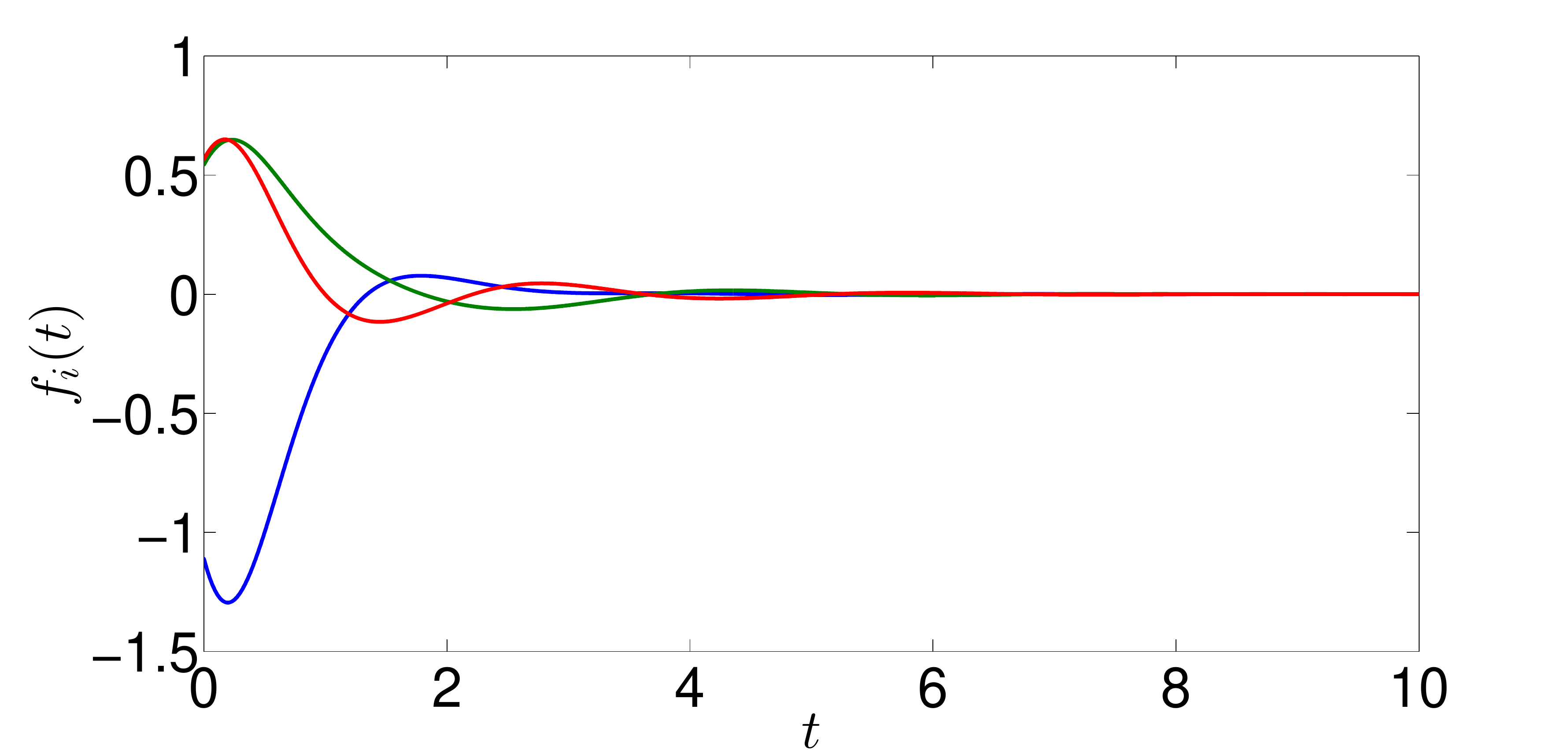}\label{fig:nu05mu04optctrls}}
	 \caption{Controlled steady state of the KS equation for $\nu = 0.5$, $\mu=0.4$ (\ref{fig:nu03contr}) and controls applied: (\ref{fig:nu05mu04ctrls}) equidistant and (\ref{fig:nu05mu04optctrls}) optimal.}
	 \label{fig:nu05mu04ss}
\end{figure}

A more detailed comparison of the energy required to control different solutions using equidistant actuators or
optimally computed positions as described above, is provided in Figures~\ref{fig:nu03steadystate} and~\ref{fig:nu05mu04ss}.
Figure~\ref{fig:nu03steadystate} shows the stabilisation of a nonuniform steady state for the KS equation, $\nu=0.3$ and $\mu=0$,
while Figure~\ref{fig:nu05mu04ss} shows analogous results but for the electrified problem with parameters $\nu=0.5$
and $\mu=0.4$ (in both cases dispersion is absent, $\delta=0$). Panel (a) shows the spatiotemporal evolution to the
desired state in the presence of controls, while panels (b) and (c) depict the evolution of the control amplitudes (there are three controls
in each case) for equidistant or optimally positioned actuators, respectively.
The results show that the amplitudes of optimally placed controls decay to zero faster than those of the equidistantly placed ones.

\subsubsection{Optimal control of travelling waves}
We also performed similar numerical experiments to find the optimal position of the control actuators when stabilising travelling waves.  We found that in most cases, we cannot do better than equidistant controls.

We believe that this is due to the following reasons. First, the length of the domain needed for the existence of (unstable) travelling waves is long, 
and therefore the number of unstable modes (and hence of the number of controls) is large (e.g., in the example of Fig.~\ref{fig:delta0} we are using $m=21$ controls); thus, shifting the position of the controls in a relatively large domain  should not have a big effect on their amplitudes. Second, 
solitary pulses on long domains necessarily have large flat regions which are susceptible to linear instabilities leading to the nonlinear wavy perturbations seen in
panels (b) and (c) in Figure~\ref{fig:compareControls}. Iit is interesting to note that there are 10 wavy structures corresponding to the number of linearly unstable modes;
thus, we expect optimality when the controls are approximately equally spaced thus guaranteeing  one control under each wavy structure.
Shifting the controls can introduce instability and nonlinear growth to a different state.

\section{Feedback and optimal control for coupled Kuramoto-Sivashinsky equations}\label{sec:coupled}

In Sections \ref{sec:stabilize}-\ref{sec:optimization} we studied analytically and computationally the feedback control
and optimal control problems for the generalised Kuramoto-Sivashinsky equation. In applications, systems of equations emerge with two or more
nonlinear coupled PDEs, and this section is concerned with the control of such systems.
As an example we refer to the system of two coupled KS equations that arises in the weakly 
nonlinear asymptotic analysis of a three layer flow
of immiscible viscous fluids stratified in a channel and driven by gravity and/or a stream wise
pressure gradient - see~\cite{Papaefthymiou2013}. The fully coupled system is a challenging PDE problem and 
questions such as existence and uniqueness of solutions, steady states and bifurcation theory are still poorly understood. Such problems are currently under investigations and our findings will be reported elsewhere.
In this Section we consider the problem of feedback and of optimal control for a system of KS equations that are coupled only through the second derivatives. Such a coupling is special but can arise in the application of three-layer flow. More generally, the nonlinearities are
also coupled and in fact the nonlinear flux functions can generically have real or complex eigenvalues implying hyperbolic elliptic
transitions thus complicating the analysis significantly; see \cite{Papaefthymiou2015} for a detailed study of such effects.

In what follows we consider the following coupled system of Kuramoto-Sivashinsky equations
\begin{equation}\label{System}
\left\{
\begin{array}{rcl}
u_{1,t} & = & -\nu u_{1,xxxx} - u_{1,xx} - u_1 u_{1,x}  - \alpha_1 u_{2,xx} \\
u_{2,t} & = & -\nu u_{2,xxxx} - u_{2,xx} - u_2 u_{2,x}  - \alpha_2 u_{1,xx} .
\end{array}
\right.
\end{equation}
We consider the equations in the interval $(0,2\pi)$ with periodic boundary conditions and initial conditions $u_1(x,0) = u_{10}(x)$ and $u_2(x,0) = u_{20}(x)$ and $u_{10}, \ u_{20} \in \dotH{2}(0,2\pi)$.

We can prove, using the background flow method~\cite{Collet1993,Nicolaenko1985,Tseluiko2007} that the solutions to the system are bounded:
\begin{prop}\label{boundsSystem}
Assume that  $u_{10}, \ u_{20} \in \dotH{2}(0,2\pi)$. Then there exists a constant $C = C(\nu, \alpha_1,\alpha_2)$ such that
\begin{equation}\label{boundSystem}
\|u_1\|_{L^2} + \|u_2\|_{L^2} \leq C.
\end{equation}
Similar bounds can be obtained for the $H^1$- and $H^2$-norm, and we can therefore conclude that the solutions to system~\eqref{System} are in $L^\infty(0,2\pi)$.
\end{prop}
The proof is not included here due to space constraints and will appear elsewhere.
A similar result can be proved for the case when the coupling comes only through the fourth order derivatives 
(i.e. there is a non-diagonal negative definite fourth order viscosity matrix), and we believe that similar results can also be proved for the fully coupled system
with non-diagonal second as well fourth order viscosity matrices.

\subsection{Feedback control for the coupled KS equations}

Since Equations \eqref{System} are coupled linearly, analogous results to the ones presented earlier
for the single KS are obtained. First, we can prove that it is possible to stabilise any steady state solution 
(either the zero solution or any nontrivial steady state) for this system. 
We proceed in the same way as for the single KS equation and write the controlled system 
(for completeness we have also included coupling through the fourth order derivatives)
\begin{equation}\label{ControlledSystem}
\left\{\begin{array}{rcl}
 u_{1,t} & = & -\nu u_{1,xxxx} - u_{1,xx} - u_1 u_{1,x} - \alpha_1 u_{2,xx} + \sum_{j_1=1}^m \delta(x-x_{j_1})f_{j_1}(t),\\
 u_{2,t} & = & -\nu u_{2,xxxx} - u_{2,xx} - u_2 u_{2,x} - \alpha_2 u_{1,xx} + \sum_{j_2=1}^m \delta(x-x_{j_2})f_{j_2}(t) .
\end{array}\right.
\end{equation}
Defining
\begin{equation}\label{galerkinsystem}
U(x,t) = \left[\begin{array}{c} u_1(x,t)\\ u_2(x,t)\end{array}\right] = \sum_{n=1}^\infty\left[\begin{array}{c}  u_{1n}^s(t)\\ u_{2n}^s(t)\end{array}\right] \sin(nx) + \sum_{n = 0}^{\infty}\left[\begin{array}{c} u_{1n}^c(t)\\ u_{2n}^c(t)\end{array}\right] \cos(nx),
\end{equation}
and taking the inner product with the functions $\frac{1}{\sqrt{2\pi}}$, $\frac{\sin(nx)}{\sqrt{\pi}}$ and $\frac{\cos(nx)}{\sqrt{\pi}}$ 
yields the following infinite system of ODEs 
\begin{equation}\label{ODEsystemSystem}
\left\{
\begin{array}{rclr}
\dot{u}_{in}^s &=& \left(-\nu n^4 +  n^2\right) u_{in}^s + \alpha_i n^2 u_{jn}^s + g_{in}^s + \sum_{j_i = 1}^m b_{j_in}^sf_{j_i}(t) & n = 1,\dots,\infty, \\
\dot{u}_{in}^c &=& \left(-\nu n^4 +  n^2\right) u_{in}^c + \alpha_i n^2 u_{jn}^c + g_n^c + \sum_{j_i = 1}^m b_{j_in}^cf_{j_i}(t) & n = 0,\dots,\infty, 
\end{array}\right.
\end{equation}
where $i,j=1,2$, $i\neq j$, and the functions $b$ and $g$ are defined as in the single KS case.
We truncate the system at $N$ modes and define
$$ z^U = \left[u_{10}^c \ u_{11}^s \ u_{11}^c \ \cdots\ u_{1N}^s \ u_{1N}^c \ u_{20}^c \ u_{21}^s \ u_{21}^c \ \cdots\ u_{2N}^s \ u_{2N}^c\right]^T,$$ 
$$G =  \left[0 \ g_{11}^s \ g_{11}^c \ \cdots \ g_{1N}^s \ g_{1N}^c \ 0 \ g_{21}^s \ g_{21}^c \ \cdots \ g_{2N}^s \ g_{2N}^c\right]^T,$$ 
$$F = \left[f_{11}(t) \ f_{12}(t) \ \cdots \ f_{1m}(t) \ f_{21}(t) \ f_{22}(t) \ \cdots \ f_{2m}(t)\right]^T.$$
Next we write
\[
 A = \left[\begin{array}{cc} A_0 & A_1 \\ A_2 & A_0\end{array}\right], \quad \quad B = \left[\begin{array}{c}B_1 \\ B_2\end{array}\right],
\]
where $A_0 = \operatorname{diag}(0, -\nu + 1, -\nu + 1, \cdots ,-\nu n^4 +n^2, -\nu n^4 +n^2,\cdots)$, $A_i = \operatorname{diag}(0, -\beta_i + \alpha_i, -\beta_i + \alpha_i, \cdots ,-\beta_i n^4 +\alpha_i n^2, -\beta_i n^4 + \alpha_i n^2,\cdots)$
and 
\[
B_i = \left[\begin{array}{cccc}
b_{1_i0}^c & b_{2_i0}^c & \cdots & b_{m_i0}^c \\
b_{1_i1}^s & b_{2_i1}^s & \cdots & b_{m_i1}^s \\
b_{1_i1}^c & b_{2_i1}^c & \cdots & b_{m_i1}^c \\
\vdots & \vdots & \cdots & \vdots
\end{array}\right],\]
for $i = 1,2$. Hence the infinite system of ODEs can be written as
\begin{equation}
\dot{z}^U = Az^U + G + BF.
\end{equation}
We can prove a result similar to Proposition~\ref{prop1}.
\begin{prop}\label{prop4}
Let $\bar{U} = \left[\begin{array}{c} \bar{u}_1 \\ \bar{u}_2\end{array}\right]$ be a nontrivial (unstable) 
steady state solution of \eqref{System}, and let $l = l_1+l_2$ be the 
number of unstable eigenvalues of the linearised system, i.e.
$l_1^2 < \frac{1+\sqrt{\alpha_1\alpha_2}}{\nu} < (l_1+1)^2 \textrm{ and } l_2^2 < \frac{1-\sqrt{\alpha_1\alpha_2}}{\nu} < (l_2+1)^2$. 
If \  $m =2(l+1)$ and  there exists a matrix $K$ 
such that all of the eigenvalues of the matrix $A + BK$ have negative real part, then the state feedback controls \begin{equation}\label{controlC} \left[f_{11}(t) \ f_{12}(t) \ \cdots \ f_{1m}(t) \ f_{21}(t) \ f_{22}(t) \ \cdots \ f_{2m}(t)\right]^T = F = K(z^U-z^{\bar{U}}),\end{equation} stabilise this nontrivial steady state solution of system \eqref{System}.
\end{prop}

The proof of this result follows the same argument as the proof of Proposition~\ref{prop1}.

We present in Figures \ref{fig:system_zero_sol} and \ref{fig:system_steady_sol} the numerical results of the stabilisation of the zero solution and a steady state solution, respectively, for system \eqref{System} with $\nu = 0.5$, $\alpha_1 = 0.8$ and $\alpha_2 = 0.5$. We used $m=4$ equidistant controls to control each solution, corresponding physically to applying $4$ controls in each wall. Upper panels correspond to the uncontrolled solution, and lower panels correspond to the stabilised solution.
We clearly observe in both figures the stabilisation of the desired steady state.

\begin{figure}[h!]
	\centering \includegraphics[width =\linewidth]{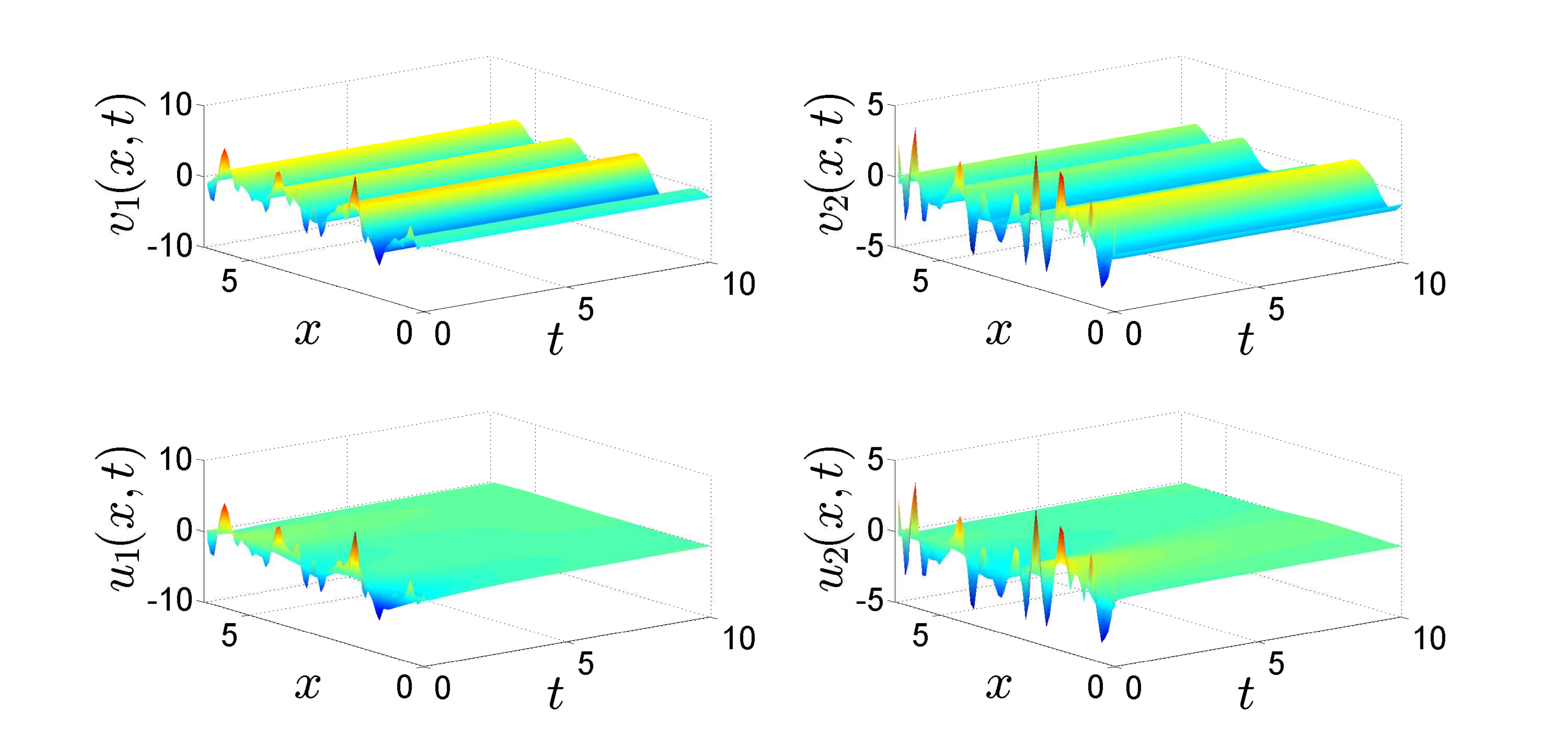}
	 \caption{Uncontrolled solution ($v_i$, $i=1,2$) and controlled zero solution ($u_i$, $i=1,2$) of the system of coupled KS equations for $\nu = 0.5$, $\alpha_1=0.8$ and $\alpha_2 = 0.5$.}
	 \label{fig:system_zero_sol}
\end{figure}

\begin{figure}[h!]
	\centering \includegraphics[width =\linewidth]{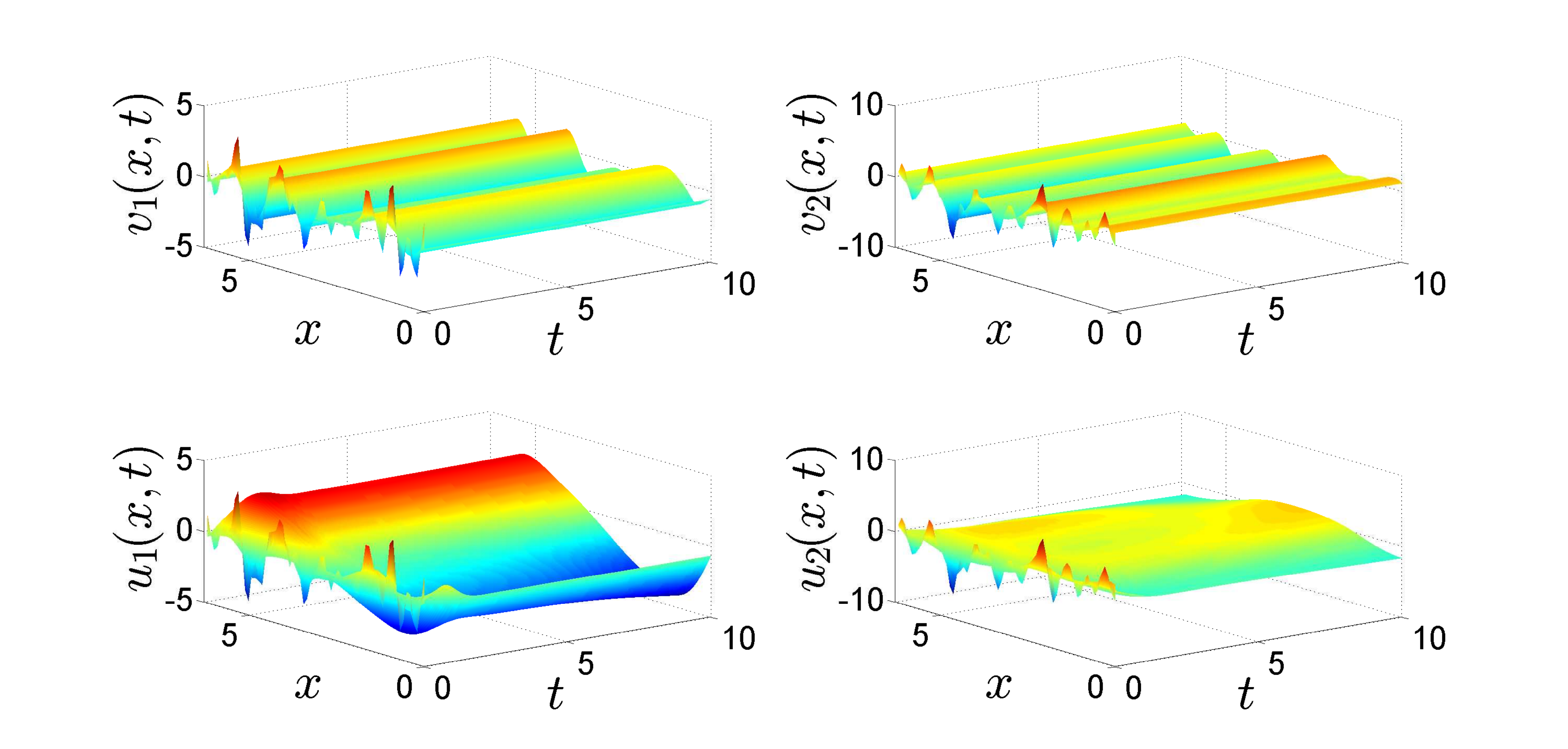}
	 \caption{Uncontrolled solution ($v_i$, $i=1,2$) and stabilised steady state solution ($u_i$, $i=1,2$) of the system of coupled KS equations for $\nu = 0.5$, $\alpha_1=0.8$ and $\alpha_2 = 0.5$.}
	 \label{fig:system_steady_sol}
\end{figure}

%
%
\subsection{Optimal control of the system of coupled KS equations}
We consider next the problem of controlling an arbitrary steady state 
$\bar{U} = \left[\begin{array}{c}\bar{u}_1 \\ \bar{u}_2\end{array}\right]$ in an optimal way. 
We introduce the cost functional
\begin{equation}\label{costfuncSystem}
\begin{array}{ll}
\cost\left(U, F \right) = &\displaystyle{\frac{1}{2}\int_0^T \left(\| u_1(\cdot,t)-\bar{u}_1\|_{L^2}^2 + \| u_2(\cdot,t)-\bar{u}_2\|_{L^2}^2\right) \ dt} \\
& \displaystyle{ +\frac{1}{2} \left(\| u_1(\cdot, T) - \bar{u}_1\|_{L^2}^2  + \| u_2(\cdot, T) - \bar{u}_2\|_{L^2}^2 \right) }\\
& \displaystyle{ +\frac{\gamma}{2} \int_0^T \left(\|f_1(x,t)\|_{L^2}^2 + \|f_2(x,t)\|_{L^2}^2\right)  \ dt}.
\end{array}
\end{equation}
The optimisation problem that we have to solve takes the form
\begin{subequations}\label{optmrobSystem}
\begin{align}
\label{costfunctionalSystem} \mathrm{minimise} \quad&  \cost\left(U,F\right) \\
	\label{stateequationSystem1} \mathrm{subject \  to} \quad &  \displaystyle{ u_{1,t} + \nu u_{1,xxxx} +  u_{1,xx} + u_1u_{1,x} + \alpha_1 u_{2,xx}=  f_1(x,t),} \\
	\label{stateequationSystem2}  &  \displaystyle{ u_{2,t} + \nu u_{2,xxxx} +  u_{2,xx} + u_2u_{2,x} + \alpha_2 u_{1,xx}=  f_2(x,t),} \\
\label{initialconditionSystem} &u_i(x,0) = u_{0,i}(x),\,\,\, i = 1,2, \\
\label{boundaryconditionsSystem} &\frac{\partial^j u_i}{\partial x^j}(x+2\pi) = \frac{\partial^j u_i}{\partial x^j}(x), \quad j = 0,1,2,3, \,\,\,i = 1,2,\\
\label{restriction1System}& f_i \in F_{ad}, \,\,\,i = 1,2.
\end{align}	
\end{subequations}
Here, $u_{0,i}\in \dotH{2}(0,2\pi)$ and $F_{ad}$ is a bounded, closed and convex subset of $L^2((0,2\pi)\times(0,T))$.

We can prove the following theorem.
\begin{thm}\label{theorem2}
If $F_{ad}\subset L^2((0,T);\dot{L}^2(0,2\pi))$, the optimal control problem \eqref{costfunctionalSystem}-\eqref{restriction1System} 
has at least one optimal control $F^* = \left[\begin{array}{c}f_1^* \\ f_2^*\end{array}\right]$ with associated 
optimal state $U^* = \left[\begin{array}{c}u_1^* \\ u_2^*\end{array}\right]$.
\end{thm}

{\bf Sketch of the proof.}
The proof follows the same steps as that of Theorem \ref{theorem} for the single Kuramoto-Sivashinsky equation. 
For the coupled system, we need to consider a different state space $X = \left(H^1(0,T;\dotH{2}(0,2\pi))\right)^2\times\left( F_{ad}\right)^2$, 
and redefine $e(\cdot,\cdot;\cdot,\cdot)$:
\begin{equation}\label{e(u,f)system} e(u_1,u_2;f_1,f_2) = \left[\begin{array}{c} u_{1,t} + \nu u_{1,xxxx} +  u_{1,xx} + u_1u_{1,x} + \alpha_1 u_{2,xx}-  f_1(x,t) \\
	 u_{2,t} + \nu u_{2,xxxx} +  u_{2,xx} + u_2u_{2,x} + \alpha_2 u_{1,xx}-  f_2(x,t) \\
u_1(\cdot,0)-u_{0,1}(x) \\ u_2(\cdot,0)-u_{0,2}(x)\end{array}\right].\end{equation}
The rest of the proof follows Theorem \ref{theorem}, but accounting for the fact that
for every $t\in[0,T]$ we have $U^*(\cdot,t) \in \left(\dotH{2}(0,2\pi)\right)^2$, 
and then $U^* (\cdot,t)\in \left(C([0,2\pi])\right)^2$ and therefore if $\varphi_i \in X$, $(u_i^*\varphi_i)(\cdot,t) \in L^2([0,2\pi])$ for $ i=1,2$.

Finally, we also need the estimates in Proposition~\ref{boundsSystem} to establish that $\|u_{i,x}^n\|_{L^2}$ is bounded, and since $H^2$ is compactly embedded in $L^2$, we deduce that
\begin{equation}\label{estimateSystem}
\int_0^T\int_0^{2\pi} (u_i^n-u_i^*)u_{i,x}^n\varphi_i\ dx \ dt \leq \|u_i^n-u_i^*\|_{L^2}\|u_{i,x}^n\|_{L^2}\|\varphi_i\|_{L^\infty} \longrightarrow_{n\rightarrow\infty} 0, \  \forall \varphi_i \in \dotH{2}(\Omega).
\end{equation}

\section{Conclusions}\label{sec:conclusion}

In this paper we studied the problem of controlling and stabilising solutions to the generalised Kuramoto-Sivashinsky equation. We studied both feedback and optimal control problems. For the optimal control problem, we proved existence of an optimal control and we investigated numerically the problem of optimal actuator placement. By extending earlier work by Christofides \emph{et al.} \cite{Armaou2000a,Armaou2000,Christofides1998,Christofides2000} we showed rigorously that we can control arbitrary nontrivial steady states of the Kuramoto-Sivashinsky equation, including travelling wave solutions, using only a finite number of point actuators. The number of point actuators needed is related to the number of unstable modes. We also investigated the robustness of the controllers with respect to changing the parameters in the equation. In particular, we showed that our proposed control methodology can be used in the presence of uncertainty. Our results can be extended to coupled systems of Kuramoto-Sivashinsky equations. 
 
We have not discussed about the practical implementation of the control methodologies studied in this paper. For example, we have assumed that complete information about the solution of the KS equation is available. This, and other issues related to the implementation of the control algorithm are discussed in \cite{Alice}.

We considered the case where the entire solution to the KS equation is available to us. It is straightforward, however, to apply our results to the case when only a finite number of observations is available, using the techniques presented in~\cite{Armaou2000a}. For brevity of exposition we have not done this for the KS equation. In a forthcoming paper \cite{Alice} we study the feedback control problem when only a finite number of observations is available to us for a more complicated PDE, the quasilinear Benney equation arising in falling film problems
(see \cite{Kalliadasis2012} and references therein).

There are several directions in which the results presented in this paper can be extended. First, the KS equation is a simplified model for thin film flows obtained using weakly nonlinear analysis and valid close to criticality~\cite{Craster2009,Kalliadasis2012}. We can apply the control methodologies studied in this paper to other simplified models that are closer to the full 2D Navier-Stokes dynamics, such as the Benney equation and the weighted residual model. These equations are more complicated since they are quasilinear and they can include additional degrees of freedom. It is possible to extend our results to such models~\cite{Alice}. Furthermore, the control of unstable travelling waves for the KS equation with dispersion can be analysed in detail, leading to an efficient and robust algorithm. This problem is studied further in~\cite{Gomes2015}. Finally, our techniques can be extended so that they apply to the noisy KS equation. Given that the noise itself can sometimes stabilise linearly unstable solutions~\cite{Pradas2012,Pradas2011}, the interaction between noise and controls can lead to very interesting dynamic phenomena. We think that this is a particularly interesting direction for further research, since the noisy KS equation is closely related to the noisy KPZ equation,  which is a universal model for weakly asymmetric processes~\cite{Hairer2013}. 

\section*{Acknowledgments}
We acknowledge financial support from Imperial College through a Roth PhD studentship, Engineering and Physical Sciences Research Council of the UK through Grants No. EP/H034587, EP/J009636, EP/K041134/1,  EP/L020564, EP/L025159/1 and EP/L024926.
We thank Professor Serafim Kalliadasis, Dr Marc Pradas and Dr Alice Thompson for useful discussions.

\nocite{Antoniades2001,Borzi2012,Rudin1974,Christofides2000a,Gong2014,Chu2001,Chu1986,Dubljevic2010,Frisch1986,Kunisch1999,Papageorgiou1993,Papageorgiou1991,Pradas2012,Pradas2011,Tseluiko2007,Zabczyk1992,Volkwein2000}

\bibliographystyle{plain}
 \bibliography{biblio}

\end{document}